\documentclass[reqno,oneside, 12pt]{amsart}
 
\usepackage[T1]{fontenc}
\usepackage{times}
\usepackage{amssymb, verbatim,   xcolor, mathabx, hyperref, enumerate, epsfig}
 \usepackage{mathrsfs}  

\usepackage[utf8]{inputenc}
\usepackage[paper=a4paper]{geometry}
\usepackage[english]{babel} 
 
\theoremstyle{plain}

\newtheorem{thm}{Theorem}[section]
\newtheorem{cor}[thm]{Corollary}
\newtheorem{pro}[thm]{Proposition}
\newtheorem{lem}[thm]{Lemma}

\newtheorem{mthm}{Theorem}

\newtheorem{mthmprime}{Theorem}

\newtheorem{mthmsecond}{Theorem}

\theoremstyle{definition}

\newtheorem{que}[thm]{Question}

\newtheorem{eg}[thm]{Example}
\newtheorem{rem}[thm]{Remark}

\numberwithin{equation}{section}   

\newtheorem*{eg*}{Example}
 
\newcommand{\e}{\varepsilon}

\newcommand{\cU}{\mathcal{U}}
\newcommand{\cV}{\mathcal{V}}

\newcommand{\ii}{{\mathsf{i}}}

\newcommand{\fr}{\partial}

\newcommand{\set}[1]{\left\{#1\right\}}
\newcommand{\norm}[1]{{\left\Vert#1\right\Vert}}
\newcommand{\abs}[1]{\left\vert#1\right\vert}

\newcommand{\pu}{{\mathbb{P}^1}}

\newcommand{\cX}{\mathcal{X}}
\newcommand{\rest}[1]{ \arrowvert_{#1}}

\newcommand{\tendvers}{\underset{n\to\infty}{\longrightarrow}}
\newcommand{\unsur}[1]{\frac{1}{#1}}

\newcommand{\cst}{\mathrm{C}^\mathrm{st}}
\newcommand{\lrpar}[1]{\left(#1\right)}

\newcommand{\inv}{^{-1}}

\DeclareMathOperator{\supp}{Supp}
\DeclareMathOperator{\Int}{Int}

\DeclareMathOperator{\mult}{mult}

\DeclareMathOperator{\dist}{d_{\mathbb H}}

\DeclareMathOperator{\Lim}{Lim}

\newcommand{\C}{\mathbf{C}}
\newcommand{\R}{\mathbf{R}}
\newcommand{\Q}{\mathbf{Q}}
\newcommand{\Z}{\mathbf{Z}}
\newcommand{\N}{\mathbf{N}}
\newcommand{\Qbar}{{\overline{\Q}}}

\newcommand{\bfk}{{\mathbf{k}}}

\newcommand{\bfA}{{\mathbf{A}}}

\newcommand{\Hyp}{\mathbb{H}}
\renewcommand{\P}{\mathbb{P}}
\newcommand{\NS}{{\mathrm{NS}}}
\newcommand{\Pic}{{\mathrm{Pic}}}

\newcommand{\Free}{{\mathrm{Free}}}
\newcommand{\Fr}{{\mathrm{fr}}}

\DeclareMathOperator{\Jac}{{Jac}}
\DeclareMathOperator{\Cur}{{Cur}}

\DeclareMathOperator{\Vect}{Vect}

\DeclareMathOperator{\Gal}{Gal}

\newcommand{\id}{{\rm id}}

\newcommand{\Aut}{\mathsf{Aut}}

 \newcommand{\PGL}{{\sf{PGL}}}

\renewcommand{\O}{{\sf{O}}}

\newcommand{\End}{{\sf{End}}}
\newcommand{\GL}{{\sf{GL}}}

\newcommand{\vol}{{\sf{vol}}}

\newcommand{\Kah}{{\mathrm{Kah}}}

\newcommand{\X}{\mathcal{X}}
\begin{document}
%
%%%%%%%%%%%%%%%%%%%%%%%%%%%%%%%%%%%%%%%%%%%%%%%%%%%%%%%%%%%%%%%%%%
%
%
%%%%%%%%%%%%%%%%%%%%%%%%%%%%%%%%%%%%%%%%%%%%%%%%%%%%%%%%%%%%%%%%%%
%
\setlength{\baselineskip}{0.53cm}        % Previous 0.53
%
%%%%%%%%%%%%%%%%%%%%%%%%%%%%%%%%%%%%%%%%%%%%%%%%%%%%%%%%%%%%%%%%%%
%
\title[Finite orbits for groups of automorphisms of projective surfaces]
{Finite orbits for large groups of automorphisms of projective surfaces}
\date{\today}

\author{Serge Cantat}
\address{Serge Cantat, IRMAR, campus de Beaulieu,
b\^atiments 22-23
263 avenue du G\'en\'eral Leclerc, CS 74205
35042  RENNES C\'edex}
\email{serge.cantat@univ-rennes1.fr}

\author{Romain Dujardin}
\address{ Sorbonne Universit\'e, CNRS, Laboratoire de Probabilit\'es, Statistique  et Mod\'elisation  (LPSM), F-75005 Paris, France}
\email{romain.dujardin@sorbonne-universite.fr}

\begin{abstract}
We study finite orbits for non-elementary groups of automorphisms of compact projective surfaces. 
In particular we prove that if the surface and the group are defined over a number field $\mathbf{k}$
and the group contains parabolic elements, then the set of finite orbits is not 
Zariski dense, except in  certain very rigid situations, known as Kummer examples. Related results are also established when 
$\mathbf{k} = \mathbf{C}$. An application is given  to the description of  ``canonical vector heights'' associated to such automorphism groups.
\end{abstract}

 \maketitle

\setcounter{tocdepth}{1}
\tableofcontents

%\renewcommand{\thefootnote}{}
%\footnotetext{\textit{2010 Mathematics Subject Classification}: 37P55, 	32H50}
%\footnotetext{\textit{Keywords}: polynomial endomorphism, orbit, valuative tree}

%%%%%%%%%%%%%%%%%%%%%%%%%%%%%%%%%%%%%%%%%%%%%%
%%%%%%%%%%%%%%%%%%%%%%%%%%%%%%%%%%%%%%%%%%%%%%
\section{Introduction}
%%%%%%%%%%%%%%%%%%%%%%%%%%%%%%%%%%%%%%%%%%%%%%
%%%%%%%%%%%%%%%%%%%%%%%%%%%%%%%%%%%%%%%%%%%%%%

\subsection{Setting} 
Let $X$ be a complex projective surface, and denote by $\Aut(X)$ its group of automorphisms. 
The group $\Aut(X)$ acts on the N\'eron-Severi group $\NS(X;\Z)$ (resp. on the cohomology group $H^2(X;\Z)$); 
this gives a linear representation 
\begin{equation}
f\mapsto f^*
\end{equation}
from $\Aut(X)$ to $\GL(\NS(X;\Z))$ (resp.   $\GL(H^2(X;\Z))$). 
By definition, a  subgroup $\Gamma$ of $\Aut(X)$ is {\bf{non-elementary}} if  its image $\Gamma^*\subset \GL(\NS(X;\Z))$ (resp. 
$\subset \GL(H^2(X;\Z))$) contains a free group of rank $\geq 2$; equivalently, $\Gamma^*$ does not contain any abelian subgroup
of finite index (see~\cite{stiffness} for details and examples).

Our purpose is to study the existence and abundance of  finite (or ``periodic'') orbits under such non-elementary group actions.   
Several possible scenarios  can be imagined:
\begin{enumerate}
\item[(a)]  a large --that is Zariski dense or dense--  set of finite orbits; 
\item[(b)] finitely many finite orbits; 
\item[(c)] no finite orbit at all. 
\end{enumerate}

 For a cyclic  
group generated by a single automorphism, 
the situation is well understood: in many cases the set of periodic points is large 
(see \cite{Cantat:Milnor} for an introduction to this 
topic, and \cite{Xie:Duke} for the case of birational transformations). 
On the other hand, for non-elementary groups, we expect the existence of a dense set of periodic points
to be a rare phenomenon; this expectation will be confirmed by our results. 
  
In fact, the only examples we know for situation~(a) are given by abelian surfaces and their siblings, Kummer surfaces. 
Here, by {\bf{Kummer surface}} we mean  a smooth surface $X$ which is a (non necessarily minimal) 
 desingularization of the quotient 
$A/G$ of an abelian surface $A=\C^2/\Lambda$ 
by a finite group $G\subset \Aut(A)$. For instance, if $G$ is generated 
by the involution $(x,y)\mapsto (-x,-y)$ on $A$, we find the so-called classical Kummer 
surfaces and their blow-ups 
(see~\cite{BHPVDV}). Given a subgroup $\Gamma\subset \Aut(X)$, we say that the pair $(X,\Gamma)$ is 
a {\bf{Kummer group}}  
if $X$ is a Kummer surface and $\Gamma$ comes from a subgroup of $\Aut(A)$ which normalizes $G$; precise 
definitions are given in  \S \ref{par:Kummer}. 
If $\Gamma$ is a group of automorphisms of an abelian surface $A$ fixing  the origin $0\in A$, 
then all torsion points are $\Gamma$-periodic. This implies that most Kummer groups have a dense set of 
finite orbits (see Proposition \ref{pro:Kummer_finite_orbits}). 

\subsection{Main results}
We first  illustrate property~(c) in the family of Wehler surfaces that is, smooth surfaces 
$X\subset \P^1\times \P^1\times \P^1$ defined by a polynomial equation of degree $(2,2,2)$. 
Such an  $X$ is a K3 surface and generically its automorphism group is generated by 
three involutions, each of them swapping one coordinate on $X$. 
We focus on these examples because they occupy a central position in the
dynamical study of surface automorphisms, both from the ergodic and arithmetic 
points of view (see e.g.~\cite{Silverman:1991, Kawaguchi:Crelle, McMullen:Crelle}).  

\begin{mthm}\label{mthm:wehler}
For a very general Wehler surface $X$, every orbit under $\Aut(X)$ is Zariski 
dense. In particular there is no finite orbit under the action of $\Aut(X)$.
\end{mthm}

Unfortunately,  from the nature of its proof, this theorem
has an obvious limitation: it does not allow to single out 
any explicit example satisfying Property~(c).

Our main result  concerns  Property~(b). To state it, recall that there are three types of automorphisms, 
characterized by the behavior of the linear endomorphism $f^*$ (see~\cite{Cantat:Milnor}). 
If $f^*$ has finite order, then $f$ is {\bf{elliptic}}. Otherwise, $f$ is either parabolic or loxodromic: it is {\bf{parabolic}} if $f^*$ 
has infinite order, but none of its eigenvalues has modulus $>1$; it is {\bf{loxodromic}} if some eigenvalue 
$\lambda(f)$ of $f^*$ has modulus $\vert \lambda(f)\vert >1$ (in that case $\lambda(f)$ is unique and 
$\lambda(f)\in (1,+\infty)$). A non-elementary group of automorphisms contains 
a non-abelian free group all of whose non-trivial elements are   loxodromic, 
and a group containing both loxodromic and parabolic elements is automatically non-elementary. 

\begin{mthm}\label{mthm:main}
Let $X$ be a smooth projective surface, defined over some number field~$\bfk$. Let 
$\Gamma$ be a subgroup of $\Aut(X)$, also defined over $\bfk$,   containing both parabolic 
and loxodromic automorphisms. If the set of finite
orbits of $\Gamma$ is Zariski dense in $X$, then $(X,\Gamma)$ is a Kummer group. \end{mthm}

When $\Gamma$ is non-elementary  there is a maximal $\Gamma$-invariant curve $D_\Gamma$; more precisely,
either $\Gamma$ does not preserve any curve, or there exists a unique, 
maximal, $\Gamma$-invariant  Zariski closed subset of pure dimension $1$. 
This curve $D_\Gamma$ can be contracted to yield a (singular) complex analytic surface $X_0$ and a $\Gamma$-equivariant 
birational morphism 
\begin{equation}
\pi_0\colon X\to X_0.
\end{equation} Thus if $(X,\Gamma)$ is not a Kummer group,   property~(b) holds on $X_0$.  
It turns out  that
{when $\Gamma$ contains a parabolic automorphism, \emph{$X_0$ is projective}} (see Proposition~\ref{pro:contraction}). Another result, which plays an important role in the proof of Theorem~\ref{mthm:main} is the following Theorem~\ref{mthm:invariant_curve_loxodromic}:
 {\emph{any non-elementary subgroup 
$\Gamma\subset \Aut(X)$ contains a loxodromic element whose maximal periodic curve is equal to $D_\Gamma$}} (see Section~\ref{sec:invariant_curves} for the precise statement).

Let us stress that even if $X$ and $\Gamma$ are defined over $\bfk$, 
Theorem~\ref{mthm:main} concerns orbits of $\Gamma$ in $X(\C)$. In this respect it is very different in spirit from the results 
of \cite{Silverman:1991}   or \cite{Kawaguchi:AJM}, in which  finiteness results are obtained  for the number of periodic orbits 
of elementary groups acting on    $X(\bfk')$, where $\bfk'$ is a fixed finite extension of $\bfk$,  
 which ultimately  rely  on Northcott's theorem.

Under the assumptions of Theorem~\ref{mthm:main}, we obtain the following corollaries (see Corollaries~\ref{cor:finite_orbits_finite}, and~\ref{cor:DMM} and Proposition~\ref{pro:Kummer_finite_orbits} for details):
\emph{ 
\begin{enumerate}[--]
\item If $\Gamma$ does not preserve any algebraic curve and $X$ is not an abelian surface,  then $\Gamma$ admits at most finitely many finite orbits.  
\item If $C$ is an irreducible  curve  containing  infinitely many $\Gamma$-periodic points, then 
either $C$ is $\Gamma$-periodic or  
 $(X, \Gamma)$ is a Kummer group and $C$ comes from a translate of  an abelian subvariety. 
In particular if $C$ has genus   $\geq 2$, it  contains at most finitely many $\Gamma$-periodic points.
 \item If $\Gamma$ has a Zariski dense set of finite orbits, then its finite orbits are dense in $X(\C)$ for the Euclidean topology; furthermore 
if $f_1$ and $f_2$ are two loxodromic automorphisms in $\Gamma$,
 their periodic points coincide, except for at most finitely many of them which are located on $\Gamma$-invariant curves.
\end{enumerate}
}
As we shall see in~Remark~\ref{rem:questions_Kawaguchi}, 
the last statement provides 
a partial answer to a  question of Kawaguchi.

\subsection{Proof strategy and extension to complex coefficients}  
Let us say a few words about the proof of Theorem \ref{mthm:main} (a more detailed outline is given in \S \ref{par:strategy_5}). 
Given  two ``typical''  loxodromic elements $f,g$ in $\Gamma$,  
intuition suggests  that $\mathrm{Per}(f)\cap \mathrm{Per}(g)$ cannot be Zariski dense 
unless some ``special'' phenomenon  happens. This situation has been referred to as an \emph{unlikely intersection} 
problem in the algebraic dynamics  literature.  Previous work on this topic 
suggests to handle this problem using  methods from arithmetic geometry (see e.g. \cite{baker-demarco, dujardin-favre}).  
In this respect a key idea would be  to use arithmetic equidistribution (see~\cite{Yuan:Inventiones, Berman-Boucksom:Inventiones})
to derive an equality $\mu_f=\mu_g$ between the measures  of 
maximal entropy of $f$ and~$g$.   
Unfortunately we do not know how to infer rigidity results directly from this equality, so the proof of Theorem~\ref{mthm:main} 
is not based on this sole argument. 
To reach a concluson, we make use of the dynamics of the whole group $\Gamma$,
in particular of the classification of $\Gamma$-invariant measures (see \cite{cantat_groupes, invariant}), together with the classification 
of loxodromic automorphisms $f$ whose  measure of maximal entropy $\mu_f$ is absolutely continuous with respect to the Lebesgue measure (see \cite{cantat-dupont,Filip-Tosatti}). 
The existence of parabolic elements in $\Gamma$ is required at three important stages, 
including the arithmetic step; in particular we are not able to prove Theorem~\ref{mthm:main} 
without assuming that $\Gamma$ contains parabolic elements (see \S \ref{subs:open} for a more precise discussion).  

Even if arithmetic methods lie at the core of the proof of Theorem \ref{mthm:main}, 
it is  natural to expect that the assumption  that  
 $X$ and $\Gamma$ be defined over a number field is superfluous. We are indeed able  to get rid of it 
when $\Gamma$ has no invariant curve.

\begin{mthm}\label{mthm:KtoC}
Let $X$ be a compact K\"ahler surface which is not a torus. Let 
$\Gamma$ be a   subgroup of $\Aut(X)$ which contains a parabolic element and     
does not preserve any algebraic curve.    Then $\Gamma$  admits only finitely many periodic points. 
\end{mthm}

The proof of Theorem~\ref{mthm:KtoC} is based on specialization arguments, 
inspired notably  by the approach of \cite{dujardin-favre} (see Section \ref{sec:KtoC}). 
It applies, for instance, to the action of $\Gamma=\Aut(X)$ on any unnodal Enriques surface $X$, and
to the foldings of euclidean pentagons with generic side lengths (see  \cite[\S 3]{stiffness} for details on these examples). 
  
\smallskip

We  conclude this introduction by explaining  two further  applications of Theorem~\ref{mthm:main}. 

\subsection{Canonical vector heights}\label{par:canonical_vector_height}  
Theorem~\ref{mthm:main}   will be applied to answer a question of Baragar on the existence of certain 
canonical heights (see~\cite{Baragar:2004, Baragar_vanLuijk:2007, Kawaguchi:2013}). 

Let $X$ be a projective surface, defined over a number field $\bfk$. Denote by $\Pic(X)$ the Picard group 
of $X_{\overline{\Q}}$. The Weil height machine provides, for every line bundle $L$ on $X$, a height function $h_L\colon X(\Qbar)\to \R$, 
defined up to a bounded error $O(1)$. This construction is additive, $h_{aL+bL'}=ah_L+bh_{L'}+O(1)$ for all pairs 
$(L,L')\in \Pic(X )^2$ and all coefficients $(a,b)\in \Z^2$. When $L={\mathcal{O}}_X(1)$ for some embedding $X\subset \P^N_\bfk$, then 
$h_L$ coincides with the usual logarithmic Weil height.

If $f$ is a regular endomorphism of $X$ defined over $\bfk$ 
and $L$ is an ample line bundle such that $f^*L=L^{\otimes d}$ for some integer $d>1$, then $h_L\circ f = dh_L+O(1)$.   Tate's renormalization trick
\begin{equation}
\hat{h}_L(x):=\lim_{n\to +\infty} \frac{1}{d^n} h_L(f^n(x))
\end{equation}
provides a canonical height for $f$ and $L$ that is, a function $\hat{h}_L\colon X(\Qbar)\to \R_+$ such that $\hat{h}_L=h_L+O(1)$ and $\hat{h}_L\circ f= d\hat{h}_L$ exactly, with 
no error term. This construction was extended to loxodromic automorphisms of projective surfaces by 
Silverman, Call, and Kawaguchi
(see~\cite{Silverman:1991, Call-Silverman:1993, Kawaguchi:AJM}): in this case one obtains a pair of canonical  
heights $\hat h_f^\pm$ satisfying $\hat h_f^\pm\circ f^{\pm 1} = \lambda(f)^\pm  \hat h_f^\pm$. 
(Note that here $\hat h_f^+$ and $\hat h_f^-$ are Weil heights associated to  $\R$-divisors.)

If  $\Gamma$ is an infinite subgroup of $\Aut(X)$, also defined over $\bfk$, it is natural to ask
whether a   $\Gamma$-equivariant family of heights can be constructed.  Specifically, one 
looks for a family of representatives $\hat{h}_L$ of the Weil height functions, i.e. $\hat{h}_L=h_L+O(1)$ for every $L$
in $\Pic(X;\R):=\Pic(X)\otimes_\Z\R$, 
depending  linearly on  $L$, and satisfying the exact relation 
 \begin{equation}\label{eq:canonical_vector_height}
\hat{h}_L(f(x))=\hat{h}_{f^*L}(x) \quad \quad (\forall x\in X(\overline{\Q}))
\end{equation} 
for every pair $(f,L)\in \Gamma\times \Pic(X;\R)$ (see \S \ref{sec:canonical_vector_height} for details).  
 A prototypical example is given by the Néron-Tate  height, when $\Gamma$ is the group of automorphisms  of an abelian surface
preserving the origin. 
Such objects were named {\bf{canonical vector heights}}~(\footnote{The name ``vector height'' comes from the following viewpoint. Assume $\Pic(X;\R)=\NS(X;\R)$, and fix a basis 
$L_i$ of  $\Pic(X;\R)$. For  a point $x$ in $X(\overline{\Q})$, consider the vector $(h_{L_i}(x))\in \R^\rho$, where $\rho=\dim_\R\Pic(X;\R)$.
Then, the equivariance property \eqref{eq:canonical_vector_height} can be phrased in terms of this vector,  hence the terminology.}) 
by Baragar in~\cite{Baragar:Canad2003}. He proved their existence when $X$ is a K3 surface with Picard number $2$,  in 
which case $\Aut(X)$ is virtually cyclic. He also  gave evidence for their  non-existence  on certain 
Wehler surfaces (see \cite{Baragar_vanLuijk:2007}). In~\cite{Kawaguchi:2013} Kawaguchi obtained a complete proof of
this non-existence for an explicit family of Wehler surfaces; his argument relies on the study of $\Gamma$-periodic orbits.     
  
Extending Kawaguchi's methods and using Theorem \ref{mthm:main},  we completely solve    the 
existence problem  of canonical vector  heights for non-elementary groups with parabolic elements:
\emph{let   $X$ be  a smooth projective surface and   $\Gamma$ be a non-elementary subgroup of $\Aut(X)$ containing parabolic elements, 
both defined over  a  number field~$\bfk$;
if $(X,\Gamma)$ possess a canonical vector height, then  $X$ is an abelian surface and $\Gamma$ has a finite orbit} (see Theorem~\ref{mthm:canonical_vector_height} in Section~\ref{sec:canonical_vector_height}).  The second
assertion implies that, after conjugation by a translation, a finite index subgroup of $\Gamma$ preserves the neutral element of the 
abelian surface $X$, in particular  the Néron-Tate height provides a canonical vector height, and we explain how all possible canonical 
vector heights are derived from the Néron-Tate height (see Theorems~\ref{mthm:prime} and~\ref{mthm:second}  below  for
    precise statements). 

\subsection{Stationary measures}  
Another  application, which was 
our primary source of motivation for this work,  concerns the classification of invariant and 
stationary measures.  

Assume that $X$ and $\Gamma$ are defined over $\R$; in particular, 
$\Gamma$ acts on the real part $X(\R)$ of $X$, which we assume here to be non-empty.
The group $\Gamma$ permutes the connected components 
of $X(\R)$ and, choosing one component $X(\R)_i$ in each $\Gamma$-orbit, the surface 
$X(\R)$ splits into a finite, disjoint union of  real analytic and $\Gamma$-invariant surfaces 
\begin{equation}
X(\R) = \bigsqcup_{i\in I} \Gamma(X(\R)_i).
\end{equation}
For simplicity, as in Theorem~\ref{mthm:KtoC},  suppose that 
$X$ is not abelian,  
$\Gamma$ contains both loxodromic and parabolic elements, and $\Gamma$ does not preserve any algebraic curve $D\subset X$. 
Let $\nu$ be a probability measure on $\Aut(X)$, whose support is a finite set generating~$\Gamma$. 
Using the results of \cite{cantat_groupes, invariant, stiffness} we infer  that: {\emph{the set of  $\nu$-stationary probability measures on $X(\R)$
coincides with that of $\Gamma$-invariant probability measures 
and is a finite dimensional simplex, whose extremal points are given by: 
\begin{itemize}
\item finitely many real analytic area forms $\omega_i$, one for each orbit 
$\Gamma(X(\R)_i)$, with  support equal to $\Gamma(X(\R)_i)$;
\item the uniform counting measures on the (finitely many) finite orbits of $\Gamma$.
\end{itemize}}}
 
 \begin{eg*}  Suppose $X$ is a real K3 surface, with $X(\R)\neq \emptyset$. Then, 
$X(\R)$ is orientable, and there is a non-vanishing section  $\Omega$ of the canonical bundle $K_X$ which induces a positive area form $\Omega_\R$ on 
$X(\R)$ (see~\cite{invariant} for instance). The area forms mentioned in the first item are the restrictions of ${\Omega_{\R}}$
to the surfaces $\Gamma(X(\R)_i)$, up to some normalization factors.  
So, if $X$ is a very general real Wehler surface with $X(\R)\neq \emptyset$, and if $\nu$ generates $\Aut(X)$,
Theorem~\ref{mthm:wehler}  implies that the only $\nu$-stationary measures are convex combinations of restrictions of the natural area
 form $\Omega_\R$ to the components of $X(\R)$ (this result was announced in \cite{stiffness}). \end{eg*}

\subsection{Organization of the paper}  
We start  by proving Theorem~\ref{mthm:wehler} in Section~\ref{sec:wehler}, which  is  
  independent of  the rest of the paper. In Section~\ref{sec:invariant_curves} we study invariant curves for loxodromic automorphisms and non-elementary groups. In particular we obtain an effective bound 
for the degree of a  curve  invariant under a loxodromic automorphism (see Proposition~\ref{pro:degree_invariant_curve}) and prove Theorem~\ref{mthm:invariant_curve_loxodromic}. 
In Section~\ref{sec:tori} we  briefly discuss the case of  tori and  review the Kummer construction. 
The core of the paper is Section~\ref{par:finite_orbits}, in which we develop the  
arithmetic method outlined above  and establish Theorem \ref{mthm:main}.
Section~\ref{sec:corollaries} is devoted to consequences of Theorem \ref{mthm:main}, and related comments.
We prove Theorem~\ref{mthm:KtoC} in Section~\ref{sec:KtoC}.   
 Finally, canonical vector heights are discussed in Section~\ref{sec:canonical_vector_height}, where we solve 
 Baragar's problem. Some open problems and 
 possible extensions of our results are discussed in \S\S \ref{subs:open} and~\ref{subs:discussion}. 
 
\subsection*{Acknowledgement} We warmly thank Pascal Autissier, Antoine Chambert-Loir, Marc Hindry, Mattias Jonsson, 
and Junyi Xie  for useful discussions or comments. 

%%%%%%%%%%%%%%%%%%%%%%%%%%%%%%%%%%%%%%%%%%%%%%
%%%%%%%%%%%%%%%%%%%%%%%%%%%%%%%%%%%%%%%%%%%%%%
\section{Very general Wehler surfaces}\label{sec:wehler}
%%%%%%%%%%%%%%%%%%%%%%%%%%%%%%%%%%%%%%%%%%%%%%
%%%%%%%%%%%%%%%%%%%%%%%%%%%%%%%%%%%%%%%%%%%%%%

Consider the family of Wehler surfaces
described in Section~3.1 of \cite{stiffness} (or in~\cite{Baragar:2004, Cantat:Acta, Kawaguchi:2013, McMullen:Crelle}). In this section  we prove Theorem \ref{mthm:wehler}. Recall the statement:
\begin{thm}\label{thm:no_finite_orbit_wehler}
If $X\subset \P^1\times \P^1\times \P^1$ is a very general Wehler surface, 
then $\Aut(X)$ does not preserve any non-empty, 
proper, and Zariski closed subset of $X$. 
\end{thm}

Here, by very general, we mean that this property holds in the 
complement of a set of countably many hypersurfaces in the space of 
surfaces of degree $(2,2,2)$ in $\P^1\times \P^1\times \P^1$.
The proof follows from an elementary but rather tedious  parameter counting argument. 
As we shall see in \S \ref{par:example_thin_wehler}, such a statement does not hold 
if we replace $\Aut(X)$ by a thin non-elementary subgroup.

%%%%%%%%%%%%%%%%%%%%%%%%%%%%%%%
\subsection{Notation and preliminaries}\label{par:notation_wehler}
%%%%%%%%%%%%%%%%%%%%%%%%%%%%%%%

We   use the notation of   \cite[\S 3.1]{stiffness}:
$M=\P^1\times \P^1\times \P^1$,  with affine coordinates $(x,y,z)$ (denoted $(x_1,x_2,x_3)$ in~\cite{stiffness}), 
$\pi_1$, $\pi_2$, and $\pi_3$ are the projections on the first, second, and third  factors, and 
$\pi_{ij}$ is the projection $(\pi_i,\pi_j)$ onto $\P^1\times \P^1$. 
Then $L_i=\pi_i^*({\mathcal O}(1))$, $L=L_1^2\otimes L_2^2\otimes L_3^2$, and $X\subset M$ is a member of  the linear system $\vert L \vert$. 
In the affine coordinates $(x,y,z)$, $X$ is defined by a polynomial  equation of 
degree $(2,2,2)$, which we write
\begin{equation}\label{eq:general_X}
P(x,y,z) = A_{222} x^2y^2z^2+ A_{221} x^2y^2z + \cdots + A_{100}x + A_{010}y   +  A_{001}z + A_{000}. 
\end{equation}
We thus  see that $ H^0(M,L) $ is of dimension 27 
and since the  equation $\set{P=0}$ is 
defined up to multiplication by a complex scalar, the family of Wehler surfaces $X$ is $26$-dimensional. 
Modulo  the action of $G  = \PGL(2, \C)^3$ they form an irreducible family   of dimension $17$.

It was shown in \cite[Prop. 3.1]{stiffness} that there exists a Zariski open set 
 $W_0\subset \vert L\vert$   of surfaces $X\in \vert L\vert$ such that 
\begin{enumerate}
\item[(i)] $X$ is a smooth K3 surface;
\item[(ii)] each of the three projections $(\pi_{ij})_X\colon X\to \P^1\times \P^1$ is a finite map, that is, 
$X$ does not contain any 
fiber of $\pi_{ij}\colon M\to \P^1\times \P^1$.
\end{enumerate} 
From now on, we suppose that $X$ belongs to$W_0$. 
 Let $i$, $j$, $k$ be three indices with $\{i,j,k\}=\{1,2,3\}$. Denote by $\sigma_i\colon X\to X$ the involutive automorphism of $X$
that permutes the points in the fibers of the $2$-to-$1$ branched covering  
$(\pi_{jk})_X\colon X\to \P^1\times \P^1$. 
By~\cite[Lem. 3.2]{stiffness}, the three involutions $\sigma_i$ generate a 
non-elementary subgroup of $\Aut(X)$.  
This subgroup is isomorphic to $\Z/2\Z\star \Z/2\Z\star \Z/2\Z$, it preserves  the subspace
of $\NS(X;\Z)$ generated by the Chern classes of the $L_i$, and its action on this subspace is 
given by the matrices in Equation~(3.4) of~\cite{stiffness}. Then  $f_{ij}=\sigma_i\circ \sigma_j$ is a parabolic automorphism
of $X$, preserving the genus $1$ fibration $\pi_k:X\to \pu$.  Moreover, if $X$ is very general
the   $L_i$ generate  $\NS(X;\Z)$ (see \cite[Prop. 3.3]{stiffness}).

%%%%%%%%%%%%%%%%%%%%%%%%%%%%%%%
\subsection{Invariant curves}\label{par:invariant_curves_wehler}
%%%%%%%%%%%%%%%%%%%%%%%%%%%%%%%

\begin{pro}
If $X\in W_0$, $\Aut(X)$   does not preserve any algebraic curve. 
\end{pro}

This is a direct consequence of the  considerations of the previous paragraph, together with the following 
more precise result. 

\begin{lem}\label{lem:wehler_invariant_curves} 
Let $X$ be a smooth Wehler surface. Assume that 
the three involutions $\sigma_i$  induce a faithful action of the group $\Z/2\Z\star \Z/ 2\Z \star \Z/2\Z$. 
Then the group generated by the $\sigma_i$ does not preserve any curve. 
  \end{lem}

\begin{proof} Assume that  $C$ is an invariant curve. Since no curve can be contained simultaneously in 
 fibers of $\pi_1$,   $\pi_2$ and $\pi_3$,
without loss of generality, we may suppose that 
 $\pi_1\colon C\to \P^1(\C)$ is dominant. Then the  
automorphism $f_{23}=\sigma_2\circ \sigma_3$ has finite order: indeed, on a general fiber 
$F$ of $\pi_1$, it acts as a translation that preserves the non-empty  finite set $F\cap C$. 
This contradicts the fact that $f_{23}$ is parabolic and finishes the proof. 
\end{proof}

Thus, to prove Theorem~\ref{thm:no_finite_orbit_wehler}, 
we are left to prove the non-existence of  periodic orbits, which is the purpose of  the following paragraphs.

%%%%%%%%%%%%%%%%%%%%%%%%%%%%%%%%%%%
\subsection{Elliptic curves}\label{par:wehler_curves}
%%%%%%%%%%%%%%%%%%%%%%%%%%%%%%%%%%%
Here we study (2,2) curves in dimension 2. We keep notation as in \S \ref{par:notation_wehler}. Let us 
consider the line bundles $L_i=\pi_i^*({\mathcal O}(1))$ on $ \P^1\times \P^1$
and set $L=L_1^2\otimes L_2^2$.  Fix (affine) coordinates $(x,y)$ on $\P^1\times \P^1$, with  $x$ and $y$ in $\C \cup \{\infty\}$.
A curve $C\subset \P^1\times \P^1$ in the linear system $\vert L\vert$ 
is given by an equation of degree (2,2) in $(x,y)$. Assume that $C$ contains 
the points $(0,0)$, $(\infty, 0)$, and $(0,\infty)$ and that it is smooth at the origin,  with a 
tangent line given by  $x+y=0$. 
Then its equation reduces to the form
\begin{equation}\label{eq:wehler_curves}
\alpha x^2y^2 + \beta x^2y+ \gamma xy^2+\delta xy + \varepsilon (x+y)=0
\end{equation}
for some complex numbers $\alpha$, $\beta$, $\gamma$, $\delta$, and $\varepsilon $, with $\e\neq 0$. Denote this curve by $C_{(\alpha,\beta,\gamma,\delta,\e)}$. 
For a general choice of these parameters, $C$ is a smooth curve of genus $1$. 
We will need the following more precise result. 
  
\begin{lem}\label{lem:singular}
Fix  $(\beta, \gamma, \delta, \e)$ with $\e\neq 0$. Then for general $\alpha$, $C_{(\alpha,\beta,\gamma,\delta,\e)}$ is smooth.
\end{lem}

\begin{proof}
An easy explicit calculation shows that 
the points of $C$ on $\set{\infty}\times \P^1$ and $\P^1\times \set{\infty}$ are
smooth unless $\alpha = \beta =\gamma = 0$. So for $\alpha\neq 0$, $C$ has no singular point
at infinity. Now, viewing the equation \eqref{eq:wehler_curves} as a quadratic equation in 
$x$ depending on the variable $y$, we can consider its discriminant 
$\Delta_x = \Delta_x(y)$, 
which is a polynomial of degree 4 in $y$ that detects fibers $\C\times \set{y}$ 
intersecting $C$ at a single point (for those values $y$ for which $C\cap \C\times \set{y}$
 is contained in $\C^2$, that is the polynomial in $x$ is of degree 2). 
 It is easy to check that if $(x,y)$ is a singular point of $C$, then  $y$ must be a multiple root 
 of $\Delta_x$.
 Hence     if $y\mapsto \Delta_x(y)$  only has simple roots, 
$C$ is smooth in $\C^2$. Thus  it is enough to check that if $(\beta, \gamma, \delta, \e)$ is an 
arbitrary 4-tuple such that $\e\neq 0$, $\Delta_x$ has only simple roots for general $\alpha$. 

A simple calculation shows that  $\Delta_x(y)   = ay^4 +by^3+cy^2+dy+e$, where  only $b$ depends on 
$\alpha$, with  $b(\alpha) = 2\gamma\delta - 4\alpha\e$, and $e=\e^2\neq 0$. Now 
the discriminant of $\Delta_x$, 
as a degree $4$ polynomial in $y$, is a polynomial expression  in $(a,b,c,d,e)$, and as a polynomial 
in $b$ it has a unique leading term $27b^4e^2$. So, $(\beta, \gamma, \delta, \e)$ being fixed, 
with $\e\neq 0$, this discriminant depends non-trivially on $\alpha$; for a general $\alpha$, 
this discriminant is not zero
thus $\Delta_x$  has four distinct roots, so that $C$ is smooth, as was to be proved.  
\end{proof}

As for Wehler surfaces, there are two involutions $\sigma_1$ and $\sigma_2$ on $C$, respectively
permuting the points in the fibers of the projections $(\pi_2)_{\vert C}\colon C\to \P^1$ and $(\pi_1)_{\vert C}\colon C\to \P^1$,
that is, $\sigma_i$ changes the $i$-th coordinate, while keeping the other ones  unchanged.
The composition $f=\sigma_1\circ \sigma_2$ is a translation on $C$ mapping  $(0,\infty)$ to $(\infty, 0)$;
in particular, $f$ is not the identity. 

\begin{lem}\label{lem:transversality_wehler_curve}
Fix  $(\beta, \gamma, \delta, \e)$ with $\e\neq 0$ and assume 
 that the curve  $C_{(0,\beta,\gamma,\delta,\e)}$ is 
smooth. Then the dynamics of the translation $f$ on $C_{(\alpha,\beta,\gamma,\delta,\e)}$  varies non-trivially 
with $\alpha$:  it is periodic for a countable dense set of $\alpha$'s, 
and non-periodic for the other parameters. 
\end{lem}  

\begin{proof}
For $\alpha$ in the complement of a finite set, $C_\alpha:=C_{(\alpha,\beta,\gamma,\delta,\e)}$ is a smooth 
curve of genus~$1$, and $f$ acts as a translation on $C_\alpha$. Let us analyze the orbit of $(0, \infty)$.
Denote by $u, v\in \C\cup \{\infty\}$ the complex numbers such that $\sigma_2(\infty,0)=(\infty,v)$ and  $(u,v)=\sigma_1(\infty,v)=f^2(0,\infty)$. 
The translation $f$ is periodic of period $2$ if and only if $(u,v)=(0,\infty)$, if and only if
 $(\infty,\infty)$ is a point of $C_\alpha$, if and only if $\alpha=0$. 
 Hence,
   $f$ is periodic of period $2$ on 
$C_0$ but after perturbation it is not of period $2$ anymore.  For small 
$\alpha$ we can write $C=\C/\Lambda_\alpha$
for some lattice $\Lambda_\alpha=\Z+\Z\tau(C_\alpha)$ and $f(z)=z+t(C_\alpha)$, 
with $t(C_\alpha)$ and $\tau(C_\alpha)$  depending holomorphically on 
the parameters $\alpha$. If we further decompose $t(C_\alpha)=a(C_\alpha)+b(C_\alpha)\tau(C_\alpha)$, 
 where $a$ and $b$ are two real analytic 
functions with values in~$\R$, then both $a$ and $b$ must be non constant. Indeed if one of them  were constant, then 
the other one would be a non-constant real holomorphic function, which is impossible (see \cite[Prop. 2.2]{cantat_groupes} for a similar argument). The result follows. 
\end{proof}

%%%%%%%%%%%%%%%%%%%%%%%%%%%%%%%%%%%
\subsection{Proof of Theorem~\ref{thm:no_finite_orbit_wehler}}  
%%%%%%%%%%%%%%%%%%%%%%%%%%%%%%%%%%%

\subsubsection{From finite orbits to fixed points}\label{par:Zd}
Let us  form the universal family ${\mathcal X}\subset W_0\times M$, where $W_0\subset \vert L\vert$
is the open set defined in Section~\ref{par:notation_wehler}: the fiber of the projection 
${\mathcal X}\to W_0$ above  $X\in W_0$ is precisely the surface $X\subset M$.

The group $\Z/2\Z\star \Z/2\Z\star\Z/2\Z$ acts by automorphisms on ${\mathcal X}$, preserving each fiber of ${\mathcal X}\to W_0$: 
the generators of the first, second and third $\Z/2\Z$ factors  give rise to  involutions
${\widetilde{\sigma}}_1$, ${\widetilde{\sigma}}_2$ and ${\widetilde{\sigma}}_3$ which, when restricted to a fiber $X$,  
correspond to the automorphisms $\sigma_i\in \Aut(X)$. 
These involutions ${\widetilde{\sigma}}_i$  
extend to birational 
involutions of the Zariski closure ${\overline{\mathcal X}}\subset \vert L \vert \times M$.

\begin{rem} If $X\in \vert L\vert$ is smooth and contains
a  fiber $V=\{(x_0,y_0)\}\times \P^1\subset X$ of $\pi_{12}$, the curve $V$ is contained in the indeterminacy locus of ${\tilde{\sigma}}_3$  
(one may consult~\cite{cantat-oguiso} for further results: see Theorem 3.3 and the proof of its third and fourth assertions).
\end{rem}

Consider the group $ \Z/2\Z\star \Z/2\Z\star\Z/2\Z$ acting on  ${\mathcal X}$. Its   
restriction     to the fiber $X$ gives a subgroup $\Gamma$ of $\Aut(X)$. Let $d$ be a positive integer.
There are only finitely many homomorphisms from 
$ \Z/2\Z\star \Z/2\Z\star\Z/2\Z$ to  groups of order $\leq d!$, and  the intersection 
 of the kernels of these homomorphisms is a normal subgroup of     finite index. Denote by $\Gamma_d$ the corresponding subgroup of $\Aut(X)$. 
If  $\Gamma $ has an orbit of cardinality  $\leq d$ on some surface $X$, 
then this orbit is fixed pointwise by 
$\Gamma_{ d}$. Let us introduce the subvariety
\begin{equation}
{\mathcal{Z}}_d  = \set{(X,x)\; ; \; x\in X \; {\text{ and }} \forall f \in \Gamma_{d}\ ,  f(x)=x   }\subset  {\mathcal X}.
\end{equation}
Since ${\mathcal X}\to W_0$ is proper, from  this discussion we get:

\begin{lem}\label{lem:transversality_Zd}
The following properties are equivalent: 
\begin{itemize}
\item[(1)] for a very general surface $X\in \vert L\vert$, every orbit of $\Gamma $  in $X$ is infinite. 
\item[(2)] for every $d\geq 1$, the projection ${\mathcal{Z}}_d \to W_0$ is not surjective; 
\end{itemize}
\end{lem}

\subsubsection{Preparation}
According to Lemma~\ref{lem:transversality_Zd}, to prove Theorem~\ref{thm:no_finite_orbit_wehler} 
it suffices to show that the projection of ${\mathcal{Z}}_d\subset {\mathcal{X}}$ onto $W_0$ is a proper subset for every $d\geq 1$. 
So, let us 
assume that there is an integer $d$ for which ${\mathcal{Z}}_d$ surjects onto $W_0$ and seek for a contradiction.
Pick a small open subset $U\subset W_0$ for the Euclidean topology, over 
 which one can choose  a holomorphic 
 section $s\colon X\mapsto s_X$ of ${\mathcal X}\to W_0$ such that $s_X$ is fixed by $\Gamma_{d}$; equivalently, 
the image of $s$ is contained in ${\mathcal{Z}}_d$.

The group $G= \PGL_2(\C)\times \PGL_2(\C)\times \PGL_2(\C)$ acts on $M$ and on $\vert L \vert$, preserving $W_0$. 
Recall that modulo the action of this group, the space of Wehler surfaces is irreducible and of dimension $17$. 

%%%
\subsubsection{Case 1}\label{par:no_finite_orbit_wehler_step_1} 
%%%
Let us first  assume that we can find $U$ such that $s_X$ is   fixed neither by ${\tilde{\sigma}}_1$,    ${\tilde{\sigma}}_2$, nor  
${\tilde{\sigma}}_3$. As in Lemma \ref{lem:singular} this implies that for each pair of indices $i\neq j$, the fiber $C$ of $(\pi_i)_X\colon X\to \P^1$ through $s_X$ is 
smooth near $s_X$ and $s_X\in C$ is not a ramification point of the projection $(\pi_j)_{\vert C} \colon C\to\P^1$. 

As in Section~\ref{par:notation_wehler}, fix coordinates $(x,y,z)$ on $M=(\P^1)^3$ with $x$, $y$, and $z$ in $\C\cup \{\infty\}$. 
Modulo the action of $G$, we may assume that for every $X$ in $U$,
\begin{itemize}
\item[(a)] the point $s_X$ is the point $(0,0,0)$ in $(\P^1)^3$;
\item[(b)] $X$ contains $(\infty, 0, 0)$, $(0,\infty, 0)$, and $(0,0,\infty)$; 
\item[(c)] the tangent plane to $X$ at the origin is given by the equation $x+y+z=0$;
\item[(d)] the coefficients of $x^2y^2z^2$ and $x$, $y$ and $z$ in the equation of $X$ are all  
 equal to the same complex number.  
\end{itemize}
Note that (a) can be achieved by a single translation, (b) can be obtained by transformations 
of the form $(x,y, z) \mapsto (\frac{x}{x-\alpha}, \frac{y}{y-\beta}, \frac{z}{z - \gamma}) $, 
(c)  is achieved by the action of diagonal maps (note that by our assumption, the tangent plane to $X$ at the origin $s_X=(0,0,0)$ 
cannot be one of the coordinate planes), and then we obtain (d) by the action of homotheties. 
After such a conjugation, the equation of $X$ is of the form 
\begin{align} \label{eq:19}
A x^2y^2z^2 &+ B x^2 y^2 z+ B' x^2 y z^2+ B'' x y^2z^2 + C x^2y z  +   C' x y^2 z+ C'' xyz^2 \\
\notag &+   Dx^2y^2+D'x^2z^2+D''y^2z^2+   E xyz    \\
\notag &+ F x^2 y + F' x^2 z + F'' x y^2  + F'''  y^2 z + F^{iv} x z^2 + F^{v} yz^2   \\
 \notag &+ G xy + G' xz + G'' yz+    A(x+y+z)=0.
\end{align}
Since this equation is defined up to multiplication by an element  of $\C^*$, we are left with 19 parameters. 

The automorphism $f_{12}=\sigma_1\circ \sigma_2$ preserves the genus 
$1$ fibration $(\pi_3)_{\vert X}\colon X\to \P^1$. 
The fiber of $(\pi_3)_{\vert X}$ through $(0,0,0)$
is a curve $C\subset \P^1\times \P^1$ given by the equation 
\begin{equation}\label{eq:5}
Dx^2 y^2+ F x^2 y +  F'' x y^2 + G xy + A(x+y)=0.
\end{equation}
Two cases need to be considered, depending on the smoothness of this curve.  
\begin{itemize}
\item if  this curve is singular, by Lemma~\ref{lem:singular} the coefficients in Equation~\eqref{eq:5} satisfy a non-trivial relation of the form
$P_3(D, F, F'', G, A) = 0$;
\item if it is smooth, consider an iterate $f_{12}^m$ of $f_{12}$ in $\Gamma_d$, with $1\leq m\leq d!$; then $f_{12}^m$ is 
a translation of the genus $1$ curve $C$ that fixes $s_X$, so that it fixes $C$ pointwise. 
From Lemma~\ref{lem:transversality_wehler_curve}, the coefficients  in  Equation~\eqref{eq:5}  satisfy a relation of the form 
$Q_3(D, F, F'', G, A) = 0$. 
\end{itemize}
In both cases we get a relation of the form 
$R_3(D, F, F'', G, A) = 0$  (with $R_3 = P_3$ or $Q_3$) 
that depends non-trivially on the first factor. 
Similarly, 
 looking at the dynamics of $f_{23}=\sigma_2\circ \sigma_3$ and $f_{31}=\sigma_3\circ \sigma_1$, 
 we obtain two further relations of the form $R_1(D'', F'', F^v, G'', A) = 0$ and 
 $R_2(D', F',F^{iv}, G', A) = 0$.

We claim that the subset defined by these 3 constraints is of codimension 3: indeed if we look at the 
subvariety cut out by the equations $R_i = 0$, $i=1, 2, 3$ and slice it by a 3-plane corresponding to 
the coordinates $D$, $D'$ and $D''$, then by Lemmas~\ref{lem:singular} and~\ref{lem:transversality_wehler_curve} and the independence of variables, 
this slice is reduced to  a point. 
This shows that the image of the section $X\mapsto s_X$ is at most $16$-dimensional, which 
contradicts the fact that $W_0/G$ is of pure dimension $17$.   Thus 
 our hypothesis on ${\mathcal{Z}}_d$ cannot be true and Case~1 does not hold.  

%%% 
\subsubsection{Case 2} 
%%% 

If Case~1 does not hold,    
every point $(X,(x,y,z))$ of ${\mathcal{Z}}_d$ has the property: {\emph{$(x,y,z)\in X$ is a  ramification point for at least one of the three projections $(\pi_i)_{\vert X}$}}.
Equivalently, every point of the finite orbit $F=\Gamma_d(s_X)\subset X$  is fixed 
by at least one of the three 
involutions $\sigma_i$. This case is simpler, since a direct count of parameters will lead to a contradiction.

\smallskip

$\bullet$ If a point of $F$ were a ramification point of each $(\pi_i)_{\vert X}$, this point would be a singularity of $X$, 
and $X$ would not be in $W_0$. So, each point of $F$ is a ramification point for at least one and at most $2$ of the projections. 

\smallskip

$\bullet$ Now, assume that {\emph{every point of $F$ is a ramification point for exactly $2$ of the projections}}. Choose a local section 
$s_X$ of $ {\mathcal{Z}}_d$ above a small open set $U\subset W_0$ (for the Euclidean topology), 
as in \S \ref{par:no_finite_orbit_wehler_step_1}. 
Permuting the coordinates and using a translation in $G$, we assume that 
$s_X=(0,0,0)$ and  $s_X$ is fixed by $\sigma_2$ and $\sigma_3$. 
After this normalization, with notation as in Equation~\eqref{eq:general_X},
 we have $A_{010}= A_{001} = A_{000}=0$. 
Let $s_X'=\sigma_1(s_X)$; 
this point is not equal to $s_X$ because 
otherwise $X$ would be singular at~$s_X$. So, we may use a transformation 
of the form $x\mapsto \frac{x}{x-\alpha}$ in $G$
 to assume that $s_X'=(\infty, 0,0)$ (i.e. $A_{200} = 0$). Now by our assumption, this second point 
must be fixed by $\sigma_2$ and $\sigma_3$, which imposes two more constraints ($A_{201} =  A_{210} =0$). 
Now, consider the curve $C_1\subset X$ defined by the equation $x=0$. Using
elements of $G$ acting on $y$ and $z$ by $y\mapsto \frac{y}{y-\beta}$ and $z\mapsto \frac{z}{z-\gamma}$, we may assume that
$(0,\infty,\infty)$ is on $C_1$ and is a ramification point for $(\pi_2)_{\vert C_1}$. 
With such a choice, the coefficients of 
$y^2z^2$ and $y^2z$ vanish. 
At this stage we did not use the diagonal action of  $(\C^*)^3$, which  
stabilizes
$(0,0,0)$, $(\infty, 0,0)$, and $(0, \infty, \infty)$. With this  
we can  impose for instance 
the same non-zero  coefficients for the terms $xy$, $yz$, and $zx$, 
 so we end up with $17$ coefficients,  hence at most $16$ free parameters.  
Again this  contradicts the fact that $\dim(W_0)=17$.

\smallskip

$\bullet$ Now, assume that {\emph{one of the points of the finite orbit $F$ is fixed by $\sigma_3$ but not by $\sigma_1$ and $\sigma_2$}}. The analysis is similar to that of the previous case.  
We may choose this point to be $s_X$, and using the group $G$, 
we can arrange  that  $s_X=(0,0,0)$, $\sigma_1(s_X)=(\infty, 0,0)$, and $\sigma_2(s_X)=(0,\infty,0)$; with the notation from Equation~\eqref{eq:general_X}, this means $A_{000} = A_{200}  = A_{020} = 0$. In addition  $A_{001} = 0$  because $(0,0,0)$ is fixed by $\sigma_3$. 
By our hypothesis, $(\infty, 0,0)$ is fixed by $\sigma_2$ or $\sigma_3$ (or both). 
This implies that at least one of $A_{210}$ or $A_{201}$ vanishes.
 Likewise $A_{120}A_{021} = 0$.  
Now consider the curve $C_2\subset X$ given by $y=0$.  Given the 
constraints  already listed, the equation of $C_2$ 
can be written as 
\begin{equation}
\alpha x^2z^2 + \beta x^2 z+\gamma xz^2+ \delta xz + \e z^2 + \iota x = 0.
\end{equation}
There are $4$ ramification points for $(\pi_1)_{\vert C_2}$, counting with 
multiplicities, and none of them satisifies $z=0$. So using $z\mapsto \frac{z}{z-\gamma}$ and $x\to \lambda x$ 
we may put one of them at $(1,0,\infty)$. This imposes $\alpha + \gamma + \e=0$ and $\beta + \delta=0$.
 
Finally, we may still use the subgroup 
$\{\mathrm{Id}\}\times \C^*\times \C^*\subset G$, which fixes the four points $(0,0,0)$, $(\infty, 0, 0)$, 
$(0, \infty, 0)$ and $(1, 0,\infty)$, to assume that the non-zero coefficients 
in front of $yz$, $xz$, and $z^2$ are equal. In conclusion, under our assumption 
we have found at least  $10$  independent linear constraints on the 
coefficients of the Wehler surface so again 
at most $16$ free parameters remain.

\smallskip

So, in all cases we get a contradiction, and   the proof of Theorem~\ref{thm:no_finite_orbit_wehler} is complete.

%%%%%%%%%%%%%%%%%%%%%%%%%%%%%%%
\subsection{Examples} \label{par:example_thin_wehler} 
%%%%%%%%%%%%%%%%%%%%%%%%%%%%%%%

\subsubsection{} Consider the subgroup $H$ of $\Z/2\Z\star \Z/2\Z\star\Z/2\Z$ generated by $f_{23}^m$ and $f_{31}^m$, 
for some large positive integer $m$ (as above, $f_{23}=\sigma_2\circ \sigma_3$, $f_{31}=\sigma_3\circ \sigma_1$). 
The automorphism $f_{23}$ preserves the fibers of the projection $(\pi_1)_{\vert X}$ and its
periodic points form a dense set of fibers (see~\cite{cantat_groupes, invariant} or~\S \ref{par:parabolic} below). The intersection number between 
a fiber of $(\pi_1)_{\vert X}$ and a fiber of $(\pi_2)_{\vert X}$ is equal to $2$. So, if $m$ is big enough, $f_{23}^m$ and $f_{31}^m$
share a common fixed point (in fact $\simeq c m^4$ common fixed points, for some $c>0$ as $m$ goes to $+\infty$). 
If $X\in W_0$, $\langle f_{23}^m, f_{31}^m\rangle$ is non-elementary because the class $c_1\in \NS(X;\Z)$ of the invariant fibration of $f_{23}$ 
is not fixed by $f_{31}$, and vice-versa (see also Lemma~\ref{lem:ping-pong_parabolic} below). 
Taking a surface $X\in W_0$ that is defined over $\Q$, we get in particular:

\begin{pro}\label{pro:number_fixed_points}
For every integer $N\geq 0$, there is a smooth Wehler surface $X$ defined over $\Q$ and a non-elementary subgroup $\Gamma$ of $\Aut(X_\Q)$ with 
at least $N$ fixed points. 
\end{pro}
 
\begin{rem} 
If $X\in W_0$ and $m\geq 1$,  {\emph{the group $\langle f_{23}^m, f_{31}^m\rangle$ has infinite index in $\Aut(X)$}}. Indeed,  
the index of $\langle (\sigma_2\circ\sigma_3)^m, 
(\sigma_3\circ\sigma_1)^m\rangle$  in $\Z/2\Z\star \Z/2\Z\star\Z/2\Z$ is infinite. 
\end{rem}

\subsubsection{}  Let us construct smooth Wehler surfaces for which the subgroup $\Gamma$
of $\Aut(X)$ generated by the three involutions $\sigma_i$ has a finite orbit. Using affine coordinates $(x,y,z)$ 
for $(\P^1)^3$, set  $V=\{(\varepsilon_1, \varepsilon_2,\varepsilon_3)\; ; \; \varepsilon_i  \in \set{0, \infty} \forall i=1,2,3\}$. 
 Observe that if $V$ is contained in a Wehler surface $X$, then $V$ is an orbit of $\Gamma$ of size $8$. Now, writing  the 
equation of $X$ as in~\eqref{eq:general_X},   $V\subset X$ if and only if $A_{ijk}=0$ for all triples $(i, j, k)\in \{0,2\}^3$. The corresponding linear system has no base points, except from 
the points of $V$. If $(A_{100}, A_{010},A_{001})\neq (0,0,0)$, the surface is smooth at the base point $(0,0,0)$. The smoothness 
at the other 7  points of $V$  is determined by   similar constraints on the coefficients $A_{ijk}$: for instance   
the smoothness  at the point $(\infty, 0,0)$  is equivalent to $(A_{210}, A_{201}, A_{100})\neq (0,0,0)$. 
 Thus, the theorem of Bertini shows that a general member of the family of Wehler surfaces containing $V$ is smooth, and 
 the result follows.

%%%%%%%%%%%%%%%%%%%%%%%%%%%%%%%%%%%%%%
%%%%%%%%%%%%%%%%%%%%%%%%%%%%%%%%%%%%%%
\section{Non-elementary groups and invariant curves}\label{sec:invariant_curves}
%%%%%%%%%%%%%%%%%%%%%%%%%%%%%%%%%%%%%%
%%%%%%%%%%%%%%%%%%%%%%%%%%%%%%%%%%%%%%

The main purpose of this section is to establish the following:

\begin{mthm}\label{mthm:invariant_curve_loxodromic}
Let $X$ be a compact K\"ahler surface  and let $\Gamma$ be a subgroup of $\Aut(X)$ containing a loxodromic element. Then there exists a loxodromic element 
$f$ in $\Gamma$ such that every $f$-periodic curve is $\Gamma$-periodic. 
\end{mthm}

 Along the way, some results of independent interest will be obtained: Proposition \ref{pro:degree_invariant_curve}, 
 which will be used in \S \ref{sec:KtoC}, gives an effective bound for  the   degree of a periodic curve under a loxodromic automorphism; Proposition \ref{pro:contraction}
provides a  singular  model of $(X, \Gamma)$ without $\Gamma$-periodic curves, and discusses ampleness properties
of some line bundles: this will be crucial for the  study the dynamical  heights    in \S \ref{par:finite_orbits}.

%%%%%%%%%%%%%%%%%%%%%%%%%%%%%%%%%%%%%%
\subsection{Preliminaries}\label{par:basics_on_par_and_lox}
%%%%%%%%%%%%%%%%%%%%%%%%%%%%%%%%%%%%%%

Let $X$ be a compact K\"ahler surface. By the Hodge index theorem, the intersection form 
$\langle \cdot \vert \cdot \rangle$ is non-degenerate and of 
signature $(1, h^{1,1}(X)-1)$ on $H^{1,1}(X;\R)$. Fix a K\"ahler form  $\kappa_0$ on $X$, with $\int_X \kappa_0\wedge \kappa_0=1$, denote its class by $[\kappa_0]$, and define the positive cone 
in $H^{1,1}(X;\R)$ to be the set 
\begin{equation}
{\mathrm{Pos}}(X)=\{ u\in H^{1,1}(X;\R)\; ; \; \langle u\vert u\rangle >0 \, \; {\text{and}} \, \; \langle [\kappa_0] \vert u\rangle >0\}.
\end{equation} 
Equivalently, ${\mathrm{Pos}}(X)$ is the connected component of $\{ u\in H^{1,1}(X;\R)\; ; \; \langle u\vert u\rangle >0\}$ 
containing K\"ahler forms; in particular, its definition does not depend on $\kappa_0$. This cone ${\mathrm{Pos}}(X)$
contains one of the two connected components, denoted $\Hyp_X$, of the hyperboloid $\{ u\in H^{1,1}(X;\R)\; ; \; \langle u\vert u\rangle =1 \}$;
we can identify $\Hyp_X$ with its projection $\P(\Hyp_X)$ in the projective space $\P(H^{1,1}(X;\R))$, and 
in doing so 
we get $\Hyp_X\simeq \P(\Hyp_X)= \P({\mathrm{Pos}}(X))$. Via this identification, the Hilbert metric on $\Hyp_X$ coincides with the 
hyperbolic metric induced by the intersection form (see~\cite[\S 2]{stiffness}), and the boundary $\partial \Hyp_X$ is identified to the 
projection of the isotropic cone in $\P(H^{1,1}(X;\R))$.

An automorphism of $X$ has a {\bf{type}} (elliptic, parabolic or loxodromic) according to the type 
of its induced action on $\Hyp_X$. 
Given a subgroup $\Gamma\leq \Aut(X)$, 
 we denote by $\Gamma_{\mathrm{par}}$ (resp. $\Gamma_{\mathrm{lox}}$) 
the set of parabolic (resp. loxodromic) automorphisms in $\Gamma$.   

\subsubsection{Parabolic automorphisms (see \cite{cantat_groupes, Cantat:Milnor, invariant})}\label{par:parabolic}
If  $g$ is parabolic, it permutes the fibers of a genus $1$ fibration $\pi_g\colon X\to B_g$, and induces an automorphism $\overline{g}$ of the 
curve $B_g$. The induced automorphism $\overline{g}$ has finite order, 
except maybe when $X$ is a torus $\C^2/\Lambda$ (see \cite[Prop. 3.6]{Cantat-Favre}).

If $\overline{g}$ is the identity, 
then $g$ preserves each fiber of $\pi_g$, acting as a translation on each smooth fiber. If $U_0$ is a disk in $B_g$ that does not contain 
any critical value of $\pi_g$, the universal cover of $\pi_g^{-1}(U_0)$ is holomorphically equivalent to $U_0\times \C$, with its 
fundamental group $\Z^2$ acting by $(x,y)\in U_0\times \C\mapsto (x,y+a+b\tau(x))$ for every $(a,b)  \in \Z^2$, where $\tau\colon U_0\to \C$ is a  holomorphic function 
 taking its values in the upper half plane. 
In these coordinates $g$ lifts to a diffeomorphism $\tilde{g}(x,y)=(x,y+t(x))$ for some  holomorphic function $t\colon U_0\to \C$. 
The $m$-th iterate $g^m$ fixes pointwise a fiber $\{x\}\times \C/(\Z\oplus \Z\tau(x))$ if and only if $mt(x)\in \Z\oplus \Z \tau(x)$.
The union of such fibers, for all $m\geq 1$, form a dense subset of $X$.  This comes from the fact that ``$t$ varies independently from
$\tau$'', a property which  implies also that the differential of $g^m$ at a fixed point is, except for finitely many fibers, a $2\times 2$ upper triangular
matrix with $1$'s on the diagonal and a non-trivial lower left coefficient. We refer to \cite{cantat_groupes, Cantat:Milnor, invariant}, and to the proof of Theorem~\ref{thm:dense_active_saddles} for
a slightly different viewpoint on this property, using real-analytic coordinates.

The induced action $g^*$ on $H^{1,1}(X;\R)$ admits a simple description:
if $F$ is any fiber of $\pi_g$, its class $[F]\in H^{1,1}(X;\R)$ is fixed by $g^*$, the ray $\R_+[F]$ is contained in the isotropic cone, and $\frac{1}{n^2}(g^n)^*w$ 
converges towards a positive multiple of $[F]$ for every $w\in {\mathrm{Pos}}(X)$. In particular the class 
$[F]$ is nef. Regarding the induced action on $\Hyp_X$, 
$\P([F])$ is the unique fixed point of the parabolic map
$g^*$ on $\Hyp_X\cup \partial \Hyp_X$ (see~\cite{cantat_groupes, Cantat:Milnor, invariant}).
 
 Recall that a linear endomorphism of a vector space is unipotent if  all its eigenvalues $\alpha\in \C$
are equal to $1$; it is {\bf{virtually unipotent}} if some of is positive iterates is unipotent. 
Since the topological entropy of $g$ vanishes, Lemma 2.6 of \cite{Cantat-Romaskevich} gives the following(\footnote{This can 
also be proved directly. For instance, on $\NS(X;\R)$ this follows from the fact that the intersection form is negative definite on $[F]^\perp/\R[F]$ and the lattice
$\NS(;\Z)$ is $g^*$-invariant.}).  

\begin{lem}\label{lem:parabolic_unipotent} Let $X$ be a complex projective surface. If $g\in \Aut(X)$ is parabolic, then $g^*$ is virtually unipotent, both on $\NS(X;\R)$ and on $H^2(X;\R)$. \end{lem}

\subsubsection{Loxodromic automorphisms (see~\cite{Cantat:Milnor})} \label{subs:loxodromic}
The dynamics of a loxodromic automorphism $f$ is much richer. 
The isolated periodic points of $f$ of period $m$ equidistribute towards
a probability measure $\mu_f$ as $m$ goes to $+\infty$, the topological entropy of $f$ is 
positive, and $\mu_f$ is the unique ergodic, 
$f$-invariant probability measure of maximal entropy.  

We denote by $\lambda(f)$ the spectral radius of the  induced  automorphism $f^*$ on $H^{1,1}(X)$, which is larger than 1. 
Then   $\lambda(f)$ and $1/\lambda(f)$ are eigenvalues of $f^*$ 
 with  multiplicity~$1$, with respective nef eigenvectors $\theta_f^+$ and 
 $\theta_f^-$  which are isotropic and 
 generate an $f^*$-invariant plane $\Pi_f\subset H^{1,1}(X;\R)$. 
 Their projectivizations are the two fixed points on $\fr \Hyp_X$ of the 
 induced loxodromic isometry    of $\Hyp_X$. The remaining eigenvalues have modulus~$1$. 
We normalize the eigenvectors $\theta^{\pm}_f$ by imposing
\begin{equation}
\langle \theta_f^+\vert [\kappa_0] \rangle= \langle \theta_f^-\vert [\kappa_0] \rangle= 1
\end{equation}
where $\kappa_0$ is the K\"ahler form introduced 
at the beginning of \S \ref{par:basics_on_par_and_lox} (recall that  $\langle [\kappa_0]\vert [\kappa_0]\rangle=1$).
We set $m_f=\unsur{2} ({\theta_f^++\theta_f^-})$. With such a choice,  $\langle{m_f\vert m_f}\rangle  = 
\unsur{2} \langle\theta_f^+\vert \theta_f^-\rangle >0$.

\begin{figure}[h]
%\centering\epsfig{figure=hyperbolic327.pdf}
\begin{center}
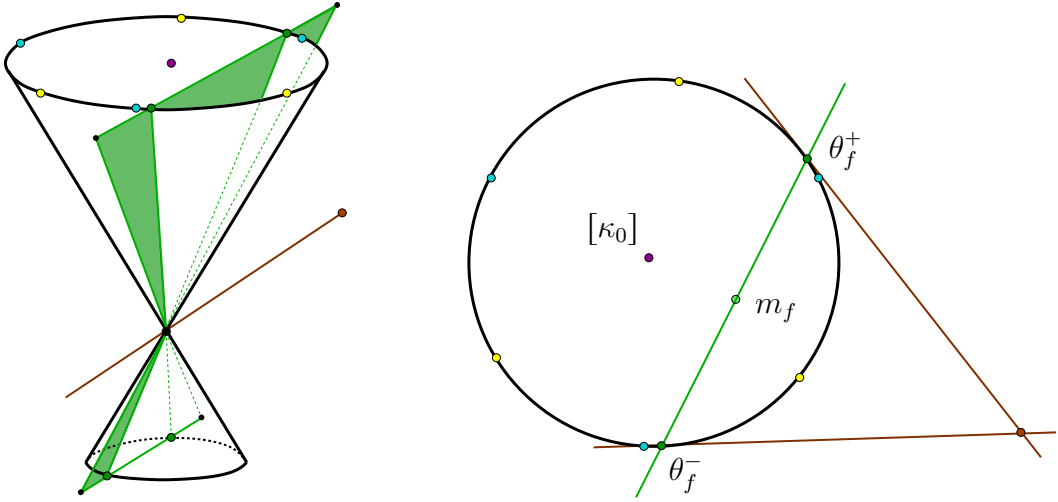
\end{center}
\caption{{\small{On the left is a picture of the Néron-Severi group of $X$ in case $\rho(X)=3$. The green plane is $\Pi_f$, it intersects 
the isotropic cone along the two lines $\R\theta_f^+$ and $\R\theta_f^-$; the brown line is its orthogonal complement $\Pi_f^\perp$, the magenta point is
$[\kappa_0]$. If $f$ preserves a curve $E$, its class is on $\Pi_f^\perp$. On the right is a projective view of the same picture, but now the two brown lines
are the projectivization of the planes $(\theta_f^+)^\perp$ and $(\theta_f^-)^\perp$.}}}
\end{figure}

\begin{rem}\label{rem:visual_angle}
Denote by ${\mathrm{Ang}}_{\kappa_0}(\theta^+_f, \theta^-_f)$ the visual angle between the boundary points $\P(\theta^+_f)$ and $\P(\theta^-_f)$, 
as seen from $[\kappa_0]$ (or $\P([\kappa_0])$). Then 
\begin{equation}\label{eq:sinus}
\langle{m_f\vert m_f}\rangle =\left(\sin\left(\frac{1}{2}{\mathrm{Ang}}_{\kappa_0}(\theta^+_f, \theta^-_f)\right)\right)^2
=\frac{1}{2}\left(1-\cos({\mathrm{Ang}}_{\kappa_0}(\theta^+_f, \theta^-_f)) \right),
\end{equation}
so in particular, $0 < \langle m_f\vert m_f\rangle  \leq 1$, and the right hand inequality is an equality if and only if $m_f= [\kappa_0]$
(\footnote{This can be obtained 
 from   elementary Euclidean geometry in the hyperplane $\langle\cdot \vert [\kappa_0]\rangle =1$ by fixing coordinates in which 
the quadratic form associated to the intersection product expresses as 
$x_0^2- x_1^2 -\ldots - x_n^2$ and $[\kappa_0]  = (1, 0, \ldots 0)$.}). 
 In $\Hyp_X$, the geodesic joining 
 $\P(\theta^-_f)$ and $\P(\theta^+_f)$ is the 
curve $\mathsf{Ax}(f)$ parametrized by $s\theta^+_f+t\theta^-_f$ with $s\in \R_+^*$ and $st=\langle \theta^+_f\vert \theta^-_f\rangle^{-1}$. 
The projection of $[\kappa_0]$ on $\mathsf{Ax}(f)$  is $ \frac{\sqrt{2}}{\langle \theta^+_f\vert \theta^-_f\rangle^{1/2}}m_f$ %= \P(m_f)$ 
and 
\begin{equation}\label{eq:cosh}
\cosh(\dist([\kappa_0], \mathsf{Ax}(f)))=\frac{\sqrt{2}}{\langle \theta^+_f\vert \theta^-_f\rangle^{1/2}}
\end{equation} (see \cite[Lem. 6.3]{blanc-cantat}). \end{rem}

\subsubsection{Non-elementary subgroups of $\Aut(X)$}\label{subsub:non_elementary}
In this paragraph we collect a few   facts 
on non-elementary  groups of automorphisms, and refer the reader to~\cite[\S 2.3]{stiffness} for details. 
By definition a subgroup $\Gamma\subset \Aut(X)$ is   non-elementary if it acts on 
$\Hyp_X$ as a non-elementary group of  isometries or, equivalently, 
if it contains a non-abelian free group, all of whose elements $f\neq \id$ are loxodromic.
Such a group $\Gamma\subset \Aut(X)$ 
preserves a unique subspace $\Pi_\Gamma
\subset H^{1,1}(X;\R)$ on which: (i) $\Gamma$ acts strongly irreducibly and (ii) 
the intersection form induces a Minkowski form. Moreover, $\Pi_\Gamma=\Pi_{\Gamma_0}$
for any finite index subgroup of $\Gamma$.

Various sufficient conditions on a subgroup $\Gamma$ imply  that it is non-elementary:
\begin{itemize}
\item $\Gamma$ contains a pair of loxodromic elements $(f,g)$ with $\set{\theta_f^+, \theta_f^-}\cap
\set{\theta_g^+, \theta_g^-} = \emptyset$;
\item $\Gamma$ contains  two parabolic elements associated to different fibrations;
\item $\Gamma$ contains a parabolic and a loxodromic element. 
\end{itemize}
 
If $\Aut(X)$ contains a non-elementary group $\Gamma$, then
$X$ is automatically projective and $\Pi_\Gamma$
is contained in the N\'eron-Severi group $\NS(X;\R)$ (see~\cite[\S 3.6]{stiffness}). 
If in addition $\Gamma$ contains a parabolic element, then
$\Pi_\Gamma$ is defined over $\Q$ with respect to the lattice 
$\NS(X;\Z)$ (see~\cite[Lem.~2.9]{stiffness}). 

The limit set $\Lim(\Gamma)\subset \fr\Hyp_X$ 
is the closure of the set of fixed points of loxodromic elements in 
$\P(\Pi_\Gamma)$, or equivalently the smallest closed invariant subset in $\fr\Hyp_X$. 
The following lemma is well-known (see \cite[Lem. 3.24]{kapovich}): 

\begin{lem} \label{lem:density_limit_set}
If $\Gamma$ is non-elementary, $\set{(\P(\theta^+_f), \P(\theta^-_f))\; ; \;  f\in \Gamma_{\mathrm{lox}}}$   is dense in $\Lim(\Gamma)^2$. 
\end{lem}

\subsection{Periodic curves of loxodromic automorphisms}\label{par:periodic_curves_of_loxodromic}
Our purpose in this paragraph is to bound the degrees of the periodic curves of a
loxodromic automorphism.

\begin{lem} \label{lem:lorenzian} Let $e$ be an element of $H^{1,1}(X;\R)$ 
such that $e$ is orthogonal to $m_f$  and $\langle [\kappa_0]\vert e\rangle = 1$. Then $\langle e \vert e \rangle < 0$ and
\[
-\langle e \vert e \rangle \geq  \frac{\langle{m_f\vert m_f}\rangle}{1- \langle{m_f\vert m_f}\rangle}=\left(\tan\left(\frac{1}{2}{\mathrm{Ang}}_{\kappa_0}(\theta^+_f,\theta^-_f)\right)\right)^2. 
\]
\end{lem}
Note that under the assumption of the lemma 
$m_f$ cannot be equal to $[\kappa_0]$,  so $0 < \langle{m_f\vert m_f}\rangle <1$ by 
  Remark~\ref{rem:visual_angle}.  
  
\begin{proof}
Write $m_f=[\kappa_0]+v$ and $e=[\kappa_0]+w$ where $v$ and $w$ are in the orthogonal complement $[\kappa_0]^\perp$. 
Then, $\langle e \vert m_f\rangle=0$, so $\langle v\vert w\rangle=-1$, and the Cauchy-Schwarz inequality gives 
$ 1 \leq (-\langle v\vert v\rangle)(-\langle w\vert w\rangle)$ because the intersection form is negative definite on $[\kappa_0]^\perp$. 
This inequality is equivalent to $1\leq (1-\langle m_f\vert m_f\rangle)(1-\langle e\vert e\rangle)$ and the result follows. 
\end{proof}

If $C\subset X$ is a curve,   define its {\bf{degree}} (with respect to $\kappa_0$) to be:  
\begin{equation}
\deg (C) = \int_C \kappa_0  =
\langle [C] \vert [\kappa_0] \rangle,
\end{equation}
and similarly define the degree of an automorphism $g\in \Aut(X)$ by:
\begin{equation}
\deg (g) = \int_X \kappa_0\wedge g^*\kappa_0  =\langle [\kappa_0]\vert  g^*[\kappa_0]\rangle.
\end{equation}
In the following lemma, $K_X$ denotes the canonical bundle of $X$:
\begin{lem}\label{lem:control_periodic_curves_strong}
Let $c_X\geq 0$ be a constant such that $\langle K_X\vert \cdot \rangle \leq c_X \langle [ \kappa_0]\vert \cdot \rangle $ on the effective cone.
If   $f\in \Aut(X)$ is  loxodromic  and 
$E$ is a reduced, connected, and  $f$-periodic curve, then      
$$\langle{\theta_f^+\vert \theta_f^-}\rangle \deg(E)\leq  2(1+ c_X).$$
If $E$ is not  connected, then $E$ has at most $\rho(X)-2$ connected components, thus 
$$\langle{\theta_f^+\vert \theta_f^-}\rangle \deg(E)\leq  2(\rho(X)-2)(1+c_X)\leq 2(b_2(X)-2)(1+c_X)$$
where $\rho(X)$ is the Picard number of $X$ and $b_2(X)$ is its second Betti number.
\end{lem}
If $E$ is $f$-invariant, then $[E]$ is orthogonal to $\Pi_f$ for the intersection form,
so the Hodge index theorem implies that  $[E]^2<0$. Thus, if $E$ is irreducible, it is 
determined by its class $[E]$, and Lemma~\ref{lem:control_periodic_curves_strong} shows that 
$f$ has only finitely many irreducible periodic curves; this finiteness result strengthens 
\cite[Prop.~B]{Kawaguchi:AJM} (see also~\cite{Cantat:Acta} and\cite[Prop. 4.1]{Cantat:Milnor}). We shall denote by $D_f$ the union 
of these irreducible $f$-periodic curves.

\begin{eg} We can take $c_X=0$ when $X$ is a K3, Enriques, or abelian surface.  
\end{eg}

\begin{proof}[Proof of Lemma~\ref{lem:control_periodic_curves_strong}] Assume first  that $E$ is connected.
Set $e= \frac{[E]}{\deg(E)}$ so that 
$\langle e\vert [\kappa_0] \rangle= 1$.   Since 
$E$ is reduced and connected, its  arithmetic genus  $\langle K_X+ E\vert E\rangle +2$ is non-negative
 (see~\cite[\S II.11]{BHPVDV}), so
\begin{equation}\label{eq:inters1}
- \langle{E\vert E}\rangle \leq  2+ \langle K_X \vert E\rangle \leq 2+ c_X\deg(E).
\end{equation}
On the other hand  Lemma \ref{lem:lorenzian} implies
\begin{equation}\label{eq:inters2}
- \langle{E\vert E}  \rangle = -\deg(E) ^2\langle{e\vert e} \rangle  \geq 
 \frac{\langle m_f\vert m_f\rangle}{1-\langle m_f\vert m_f\rangle}\deg(E)^2.
\end{equation}
Putting these two inequalities together we get
\begin{equation}\label{eq:inters3}
\deg(E)^2 \leq 
\frac{1- \langle m_f\vert m_f\rangle}{\langle m_f\vert m_f\rangle} \lrpar{ 2+ c_X \deg(E)}.
\end{equation} 
Solving for the corresponding quadratic equation in $\deg(E)$, and applying the inequality 
 $  t(1-t)\leq 1/4$ with $t=  \langle m_f\vert m_f\rangle$ finally gives 
\begin{equation}\label{eq:degaE}
 \langle m_f\vert m_f\rangle \deg(E)   \leq (1-\langle m_f\vert m_f\rangle)c_X + 1/ \sqrt{2} \leq c_X+1. 
\end{equation}
 
For the second assertion, write $E$ as a union of disjoint connected components $E_i$. The classes
$[E_i]$ are pairwise orthogonal, and are contained in  $(\theta^+_f)^\perp\cap (\theta^-_f)^\perp$, 
a subspace of codimension $2$ in the Néron-Severi group of $X$. This implies that there are
at most $\rho(X)-2$ connected components. 
\end{proof}

\begin{pro}\label{pro:degree_invariant_curve}  
Let $X$ be a compact K\"ahler surface endowed with a reference K\"ahler form $\kappa_0$ such that $\int \kappa_0^2=1$. 
If $f\in \Aut(X)$ is  loxodromic and $E$ is an $f$-invariant curve, then  
$$\deg(E)\leq 2^{54}(\rho(X)-2)(1+c_X) \deg(f)^{56}, $$ 
where the degrees are relative to $\kappa_0$ and  $c_X$ is as in Lemma \ref{lem:control_periodic_curves_strong}. 
\end{pro}

\begin{proof}
As in Remark~\ref{rem:visual_angle}, denote by $\dist$ the hyperbolic distance on $\Hyp_X$  and let $\mathsf{Ax}(f)$ be the axis of the 
loxodromic isometry  $f^*$. Lemma 4.8 in \cite{blanc-cantat} implies 
 that(\footnote{It was stated for birational transformations of $\P^2$ in \cite{blanc-cantat} but the estimate holds 
in our setting with the same proof (actually an easier one since here we work in a finite dimensional hyperbolic space).}) $$\dist([\kappa_0], \mathsf{Ax}(f))\leq 28 \dist ([\kappa_0], f^*[\kappa_0]) = 28 \cosh\inv(\deg(f)).$$
Then, using the  formula~\eqref{eq:cosh} for the   
 distance $\dist([\kappa_0], \mathsf{Ax}(f))$ together with the elementary  inequality $\cosh(kx)\leq 2^{k-1} \cosh (x)^k$, we obtain
\begin{equation}
\frac2{\langle \theta_f^+\vert  \theta_f^-\rangle}   
=  \cosh (\dist([\kappa_0], \mathsf{Ax}(f))) ^2 
\leq 2^{54} (\deg(f))^{56}.
\end{equation} 
The result now follows  from Lemma~\ref{lem:control_periodic_curves_strong}.
\end{proof}

%%%%%%%%%%%%%%%%%%%%%%%%%%%%%%%%%
\subsection{$\Gamma$-periodic curves, singular models and ampleness} \label{par:periodic_curves_ampleness}
%%%%%%%%%%%%%%%%%%%%%%%%%%%%%%%%%
Denote by $\Pi_\Gamma^\perp$ the orthogonal complement of $\Pi_\Gamma$ 
with respect to the intersection form.

\begin{lem}\label{lem:carac_invariant_curves}
 Let $\Gamma\subset \Aut(X)$ be a non-elementary subgroup. 
 \begin{enumerate}[{\em (i)}]
 \item A  curve $C\subset X$ is $\Gamma$-periodic if and only if $[C]\in \Pi_\Gamma^\perp$.  
 \item If   $\Gamma$ contains a parabolic element,  and $C$ is irreducible, then $C$ is $\Gamma$-periodic  if and only 
if $C$ is contained in a fiber of $\pi_g$ for every $g\in \Gamma_{\mathrm{par}}$.
 \end{enumerate}
\end{lem}
 
 \begin{proof} For (i), we   note that since the intersection form is negative definite on   
 $\Pi_\Gamma^\perp$, $\Gamma$ acts on this space as a group of Euclidean isometries. Thus, 
 if $c\in \Pi_\Gamma^\perp$ is
an integral class, then $\Gamma^*(c)$ is a finite set. Since $\Pi_\Gamma$ is generated by nef classes, 
$[C]$ belongs to $\Pi_\Gamma^\perp$ if and only if each of its irreducible components does, so it
  is enough to prove the result for an irreducible curve. 
 Now an  irreducible curve $C$ with 
negative self-intersection  is uniquely determined by its class $[C]$; so if $[C]$ is 
contained in $\Pi_\Gamma^\perp$,  we conclude that $C$ is $\Gamma$-periodic. 
Conversely, if $C$ is $\Gamma$-periodic, a finite index subgroup $\Gamma'\subset \Gamma$ preserves~$C$. 
If $f\in \Gamma'_{\mathrm{lox}}$, then $\langle  \theta^+_f \,\vert\, [C]\rangle =0$ because $f$ preserves the intersection form. 
But ${\mathrm{Vect}}( \theta^+_f , f\in \Gamma'_{\mathrm{lox}})$ is 
a $\Gamma'$-invariant subspace of $\Pi_\Gamma$, hence by strong irreducibility it coincides with $\Pi_\Gamma$ (see~\S \ref{subsub:non_elementary}). 
So, $[C]\in \Pi_\Gamma^\perp$, and in particular $[C]^2 <0$.

Let us prove the second assertion. 
If $[C]\in \Pi_\Gamma^\perp$ and $g\in \Gamma_{\mathrm{par}}$, $[C]$ intersects trivially the class $[F]$ of the general fiber of $\pi_g$; this implies
that $C$ is contained in a fiber of $\pi_g$, and is a component of a singular fiber since $[C]^2<0$. Now, denote by $S$ the set of irreducible curves which are a 
component of $\pi_g$ for all $g\in \Gamma_{\mathrm{par}}$; it remains to prove that each $C\in S$ is $\Gamma$-periodic. 
Since $\Gamma$ is non-elementary, $\Gamma_{\mathrm{par}}$ contains two elements $g_1$ and $g_2$ with distinct fixed points
on the boundary of $\Hyp_X$; these fixed points are respectively given by the classes $[F_1]$ and $[F_2]$ of any smooth fiber of $\pi_{g_1}$ and $\pi_{g_2}$; hence, 
$\pi_{g_1}$ and $\pi_{g_2}$ can not share any smooth fiber. This shows that elements of $S$ are contained in singular fibers of $\pi_{g_1}$, and in particular $S$ is finite.
Moreover, $S$ is $\Gamma$-invariant, because $\Gamma_{\mathrm{par}}$ is invariant under conjugacy, thus every $C\in S$ is a $\Gamma$-periodic curve.
\end{proof}

The following proposition shows  that examples as in 
 \cite[\S 11]{cantat-dupont} do not appear for non-elementary groups {{containing parabolic automorphisms}}.

\begin{pro}\label{pro:contraction} Let $\Gamma\subset \Aut(X)$ be a non-elementary subgroup containing parabolic automorphisms. 
There is a birational morphism $\pi_0\colon X\to X_0$ onto a normal projective surface $X_0$ and a homomorphism
$\tau\colon \Gamma\to \Aut(X_0)$
 such that 
\begin{enumerate}[\em (1)]
\item $\pi_0$ contracts all $\Gamma$-periodic curves and only them;
\item $\pi_0$ is equivariant: $\pi_0\circ f= \tau(f)\circ \pi_0$ for every $f\in \Gamma$;
\item there is an ample line bundle $A$ on $X_0$ such that $\pi_0^*A$ is a big and nef line bundle, whose
 class  belongs to $\Pi_\Gamma$.
 \end{enumerate}
\end{pro}

Before starting the proof, recall that a line bundle $M$ on $X$ is semi-ample if and only if 
$M^{\otimes m}$ is globally generated (or equivalently base-point free) 
for some $m>0$ (see~\cite[\S 2.1.B]{Lazarsfeld:Book1}). 
Set 
\[
\Free(X;M)=\{ m \in \N \; \vert\; mM \; \; {\text{is base point free}}\}.
\]
(here we use the additive notation    $mM$ for  the line bundle $M^{\otimes m}$.)
This defines  a semi-group in $\N$, and we denote by $\Fr(M)$ the largest natural number such that every 
element of $\Free(X;M)$ is a multiple of $\Fr(M)$. Given $k$ in $\Free(X;M)$, the line bundle $kM$ 
determines a  morphism 
\[
\Phi_{kM}\colon X\to X_{kM} \subset \P(H^0(X,kM)^{\vee}),
\]
onto a projective (possibly singular) normal variety $X_{kM}$. According to Theorem~2.1.27 in \cite{Lazarsfeld:Book1}, there is an algebraic 
fibre space $\Phi\colon X\to Y$
such that 
\begin{enumerate}
\item $Y$ is a normal projective variety (see Example~2.1.15 in~\cite{Lazarsfeld:Book1});
\item $X_{kM}=Y$ and $\Phi_{kM}=\Phi$ for sufficiently large elements $k$ of $\Free(X;M)$;
\item there is an ample line bundle $A$ on $Y$ such that $\Phi^*A=\Fr(M)M$.
\end{enumerate}
Note that conversely the pull-back of a base point free line bundle by a morphism is base point free. 

\begin{eg}\label{eg:parabolic_second_example}
To each $g\in \Aut(X)_{\mathrm{par}}$ corresponds a semi-ample line bundle $L_g$ such that (i) the members of $\vert L_g\vert$ are given by 
the fibers of $\pi_g$ and (ii)  $\pi_g\colon X\to B_g$ 
coincides with the fibration $\Phi\colon X\to Y$ determined by $L_g$. The ray $\R_+[L_g]\subset H^{1,1}(X;\R)$ determines the unique 
fixed point of $g^*$ in $\partial \Hyp_X$, and $L_g$ is nef (see Section~\ref{par:basics_on_par_and_lox}).
\end{eg}

\begin{proof}[Proof of Proposition~\ref{pro:contraction}] 
By Lemma \ref{lem:carac_invariant_curves} we can 
fix  a finite number of parabolic elements $g_i\in \Gamma$, $1\leq i\leq k$,
such that the set of irreducible and $\Gamma$-periodic curves $C\subset X$ is 
exactly the set of irreducible curves 
which are contained in fibers of $\pi_{g_i}$ for $i= 1, \ldots, k$. The line bundle $M=\sum_i L_{g_i}$ 
is semi-ample,  it is nef because the 
$L_{g_i}$ are  and  it is big because $M^2 >0$, finally  its class belongs to 
$\Pi_\Gamma$ because the classes $[L_{g_i}]$ belong to the limit set of $\Gamma$ 
(see~\cite[\S 2.3.6]{stiffness}). Since $M$ is big, 
the fibration $\Phi = \Phi_{kM}\colon X\to Y$ defined, as above, by the sufficiently large multiples of $M$  is
a birational morphism (a generically finite fibration is a birational morphism since its fibers are, by definition, connected). 
By construction $\Phi$ contracts exactly the periodic curves of $\Gamma$. So, setting $\pi_0=\Phi$ and $X_0=Y$,
we obtain a birational morphism that contracts all periodic curves, and only them. Since $\Gamma$ permutes these curves, 
it induces a group of automorphisms on $X_0$. Moreover, we know that there is an ample line bundle $A$ on $X_0$ such that $\pi_0^*A=\Fr(M)M$; this
proves the third assertion.
\end{proof}

\begin{rem}\label{rem:contraction}
In the proof of Proposition~\ref{pro:contraction}, one may add extra parabolic automorphisms $g_j\in \Gamma_{\mathrm{par}}$, say with $k+1\leq j\leq \ell$, and 
replace $M$ by $\sum_{i=1}^\ell m_i L_i$ for any choice of integers $m_i>0$, while getting the same 
conclusion. After multiplication by $\Q_+^*$, the classes constructed in this way
form a dense subset   
of the convex cone  
\begin{equation}
\left\{\sum_{i=1}^\ell \alpha_i c_1(L_{g_i})\; ; \; \ell \geq 1, \; g_i\in \Gamma_{\mathrm{par}}, \; {\text{and }} \; \alpha_i \in \R_+^* \; {\text{ for all }} \; i\right\}.
\end{equation}
This cone is $\Gamma$-invariant, its closure is the smallest convex cone whose projectivization 
contains the limit set $\Lim(\Gamma)$, and it spans $\Pi_\Gamma$ because 
$\Pi_\Gamma$ is the smallest vector space containing $\Lim(\Gamma)$. 
Thus, the classes of the form  $\alpha c_1(\pi_0^*A)$, where $A$ runs over the set of ample line bundles on 
$X_0$ and $\alpha$ runs over  $\Q_+^*$, is a dense subset of  this cone.  
 \end{rem}

\subsection{Proof of Theorem~\ref{mthm:invariant_curve_loxodromic}} 

Let us first deal with the case where   
$\Gamma$ is elementary. By~\cite[Theorem 3.2]{Cantat:Milnor} there is a loxodromic element $f\in \Gamma$ 
such that $(f^*)^\Z$ has finite index in $\Gamma^*$. If $\Aut(X)^0$ is non-trivial, 
then $X$ is a torus and then $f$ has no invariant curve (see~\cite[Remark 3.3]{Cantat:Milnor} and \cite{Cantat-Favre}). Otherwise, the kernel of the 
homomorphism $\Gamma\to \Gamma^*$ is finite, $f^\Z$ has finite index in $\Gamma$, and therefore  a curve is $\Gamma$-periodic if and only if
it is $f$-periodic, so we are done when $\Gamma$ is elementary.
 
When $\Gamma$ is non-elementary, Theorem~\ref{mthm:invariant_curve_loxodromic}
 is covered by the following 
more precise statement (recall that $X$ is automatically projective in this case \cite[Thm. 3.17]{stiffness}).

\begin{pro} \label{pro:invariant_curve_loxodromic_precised}
Let $X$ be a complex projective surface and $\Gamma$ be a  non-elementary 
subgroup of $\Aut(X)$. Then there exists a loxodromic element 
$f$ in $\Gamma$ such that every $f$-periodic curve is $\Gamma$-periodic.
 If in addition $\Gamma$ contains a 
parabolic element, $f$ can be chosen of the form $h \circ g$, where $g$ and $h$ are   parabolic and unipotent. 
\end{pro}

\begin{proof}
Consider a subset $S\subset \Gamma_{\mathrm{lox}}$ such that $\set{(\P(\theta^+_f), \P(\theta^-_f))\; ; \;  f\in S}$
is dense in $\Lim(\Gamma)^2$, as in Lemma~\ref{lem:density_limit_set}. Let us exhibit an 
$f\in S$ such that every $\Gamma$-periodic curve is $f$-periodic. By contradiction, we 
assume    that every  $f\in S$ admits  at least one   irreducible periodic curve $C(f)$ which is not $\Gamma$-periodic, and 
we set  $c(f)=[C(f)]$. 
By Lemma \ref{lem:carac_invariant_curves}, $c(f)$ does not belong to  $\Pi_\Gamma^\perp$, thus $u\mapsto \langle c(f)\vert u\rangle$ is 
a non-trivial linear form on $\Pi_\Gamma$.
Since the class of any periodic curve is orthogonal to $\Pi_f$, $\langle c(f)\vert \theta^+_f\rangle   = \langle c(f)\vert \theta^-_f \rangle = 0$.

Let $U$ and $U'$ be  open subsets of 
$\partial \Hyp_X$ intersecting $\Lim(\Gamma)$ non trivially, and 
such that $\overline U\cap \overline U' = \emptyset$; let $x$ be an element of $U\cap \Lim(\Gamma)$.
Define 
\[
A(U,U')=\{ f\in \Aut(X) \; ; \; f \; {\text{is loxodromic}},  \; \P(\theta^+_f)\in U \; {\text{and}} \; \P(\theta^-_f)\in U'\} .
\]
and 
\begin{equation}
D(U,U')=\{ c(f)\; ; \; f \in A(U,U')\cap S\}. 
\end{equation}
By Lemma~\ref{lem:control_periodic_curves_strong}, $D(U,U')$ is a finite set. 
From our assumption on $S$, there is a sequence $(f_n)$ of elements in 
$A(U,U')\cap S$ such that  $x=\lim_n(\P(\theta^+_{f_n}))$. Extracting a subsequence 
if necessary we may assume that $c(f_n)$ is constant, equal to some $c_\infty\in D(U,U')$, and we infer that $x$ is contained in  
$c_\infty^\perp$. As a consequence, the limit set  $\Lim(\Gamma)\subset \P(\Pi_\Gamma)$
is locally contained in the finite union of hyperplanes $\P(c^\perp\cap \Pi_\Gamma)$, for $c\in D(U,U')$. By compactness,  $\Lim(\Gamma)$ 
is contained in a finite  union of hyperplanes, which contradicts the strong irreducibility of the action of $\Gamma$ on $\Pi_\Gamma$. 

Now, to prove the first assertion of the proposition, we simply put $S  = \Gamma_{\mathrm{lox}}$, 
which   satisfies the desired density property by Lemma \ref{lem:density_limit_set}. For the second assertion, 
we let $S$ be the set of loxodromic elements of the form $h \circ g$, 
where $g$ and $h$ are   parabolic and unipotent. To check the density property, we first observe that 
the set of fixed points of parabolic elements is dense in $\Lim(\Gamma)$: indeed it is enough to consider the conjugates of a single parabolic transformation. Then, applying  the next lemma together with 
Lemma~\ref{lem:parabolic_unipotent}   finishes the proof.  
\end{proof}

\begin{lem}\label{lem:ping-pong_parabolic}
Let $h$ and $h'$ be two parabolic elements of $\Aut(X)$ with distinct fixed points 
$u$ and $u'$ in $\partial \Hyp_X$. 
Let $U$ and $U'$ be small, disjoint neighborhoods of $u$ and $u'$, respectively, in $\P(\Pi_\Gamma)$. 
Then if $N\in \Z$ is large enough,   $f_N:=h^N\circ (h')^N$ is a loxodromic automorphism 
such that $\P(\theta^+_{f_N})\in U$ and  $\P(\theta^-_{f_N})\in U'$.
\end{lem}
\begin{proof} Let  us denote by $\P(h^*)$ the linear projective transformation induced by $h^*$ on $\P(\NS(X;\R))$.
Since $U$ does not contain $u'$, $\P (h'^*)^N(U)\subset U'$ if  $\vert N\vert$ is large enough; similarly $\P (h^*)^N(U')\subset U$. So for  
$f_N=h^N\circ (h')^N$, $\P (f_N^*)$ maps $U'$  stricly inside itself and likewise 
 $\P ((f_N\inv)^*)$  maps $U$ strictly inside itself. This implies 
that $f_N$ is loxodromic, with its $\alpha$-limit and $\omega$-limit points in $U'$ and $U$ respectively. 
\end{proof}

%%%%%%%%%%%%%%%%%%%%%%%%%%%%%%%%%%%%%%%%%%%%%%
%%%%%%%%%%%%%%%%%%%%%%%%%%%%%%%%%%%%%%%%%%%%%%
\section{Complex tori and Kummer examples}\label{sec:tori}
%%%%%%%%%%%%%%%%%%%%%%%%%%%%%%%%%%%%%%%%%%%%%%
%%%%%%%%%%%%%%%%%%%%%%%%%%%%%%%%%%%%%%%%%%%%%%

This section gathers some facts on automorphism groups of complex tori. We also   introduce and study the 
notion of   Kummer group. Part of this  material is well-known, we provide the details for completeness. 

\subsection{Finite orbits on tori} \label{par:abelian}
Consider a compact complex torus $A=\C^k/\Lambda$. Each automorphism $f$ of $A$ is an affine transformation $f(z)=L_f(z)+t_f$, where 
$z\mapsto z+t_f$ is the translation part and $L_f$ is a linear automorphism, induced by a linear transformation of 
$\C^k$ that preserves~$\Lambda$. Let $\Gamma$ be a  subgroup of $\Aut(A)$.

\smallskip

{\noindent \bf{Warning.}}  By definition, 
compact tori and abelian varieties come equipped with their group structure,
in particular with their neutral  element, or ``origin''. On the other hand, an automorphism $f$ with a non-trivial translation part $t_f$
does not preserve this group  structure. If $x\in A$
is fixed by $\Gamma$, conjugating  $\Gamma$ by the translation $z\mapsto z+x$ we may assume
 that $\Gamma$ fixes the neutral 
element of $A$ and acts by linear isomorphisms. Alternatively, we can  transport the group structure 
by using  this translation  and put the neutral 
element at $x$: this   changes the group structure without changing the underlying complex manifold. We will frequently  do this 
operation in the following,  without always specifying the change in the group structure of $A$. 

\smallskip

Suppose that the orbit $\Gamma(x)\subset  A$ is finite, of cardinality $m$,  and 
consider the stabilizer $\Gamma_0\subset \Gamma$ of $x$; its index divides $m !$. Conjugating $\Gamma$ by $z\mapsto z+x$, 
as explained above, all elements $f\in \Gamma_0$ are linear. 
In that case, every torsion point has a finite $\Gamma_0$-orbit, hence also a finite $\Gamma$-orbit; in particular, 
finite orbits of $\Gamma$ form a  dense subset of $A$ for the Euclidean topology.  
The next proposition summarizes this discussion. 

\begin{pro}\label{pro:finite_orbits_torus}
Let $A$ be a compact complex torus, and let $\Gamma$ be a subgroup of $\Aut(A)$. If $\Gamma$ has a finite 
orbit, then its finite orbits  form a dense subset of $A$. More precisely   if   a periodic 
point of $\Gamma$ is chosen as the origin of $A$ for its group law, then all torsion points of 
$A$ are periodic points of $\Gamma$. \end{pro}

\begin{rem}\label{rmk:torsion}
If in Proposition~\ref{pro:finite_orbits_torus} we moreover assume  that $\dim_\C A=2$ and $\Gamma$ contains a loxodromic element, then conversely  \emph{all periodic points of $\Gamma$ are torsion points}. This follows from Lemma~\ref{lem:lox_on_tori} below.  
\end{rem}

%%%%%%%%%%%%%%%%%%%%%%
\subsection{Dimension $2$ (see~\cite{Cantat:Milnor, McMullen:Crelle})}\label{par:lox_on_tori}
%%%%%%%%%%%%%%%%%%%%%%
Let $A=\C^2/\Lambda$ be a compact complex torus of dimension $2$, and let $f(z)=L_f(z)+t_f$ be a 
loxodromic element of $\Aut(A)$. 
The loxodromy means exactly that the eigenvalues $\alpha$ and $\beta$  of $L_f$   satisfy $\abs{\alpha}<1<\abs{\beta}$.
Pick a basis of $\Lambda$, and use it to identify $\Lambda$ with $\Z^4$ and $\C^2$ with $\R^4$, as real vector spaces. Then, $L_f$ corresponds to 
an element $M_f\in \GL_4(\Z)$. 

\begin{lem}\label{lem:lox_on_tori} Let $f$ be a loxodromic automorphism of a compact complex torus $A$ of dimension $2$. Then:
\begin{enumerate}[{\em(1)}]
\item  $f$  has a fixed point, and
after translation $z\mapsto z+x$ by such a fixed point, its periodic points  are exactly the torsion points of $A$; 
\item there is no $f$-invariant curve: the orbit $f^\Z(C)$ of any curve is dense in $A$.  
\end{enumerate}
\end{lem}

\begin{proof}
For (1), using  the above notation, 
fixed points of $f$ are determined by the equation $(L_f-\id)(z)\in  \Lambda -t_f$, or equivalently $(M_f-\id)(z)\in \Z^4-t_f$. 
Since the complex eigenvalues of $L_f$ are distinct from $1$, there is at least one fixed point. So, after conjugation by a translation, 
we may assume that $t_f=0$.  Then, periodic points of $f$ correspond to points $x\in \R^4$ such 
that  $M_f^n (x)-x \in  \Z^4$: any solution to such an equation is rational, which means that it corresponds to a torsion point in $A$.
 
To prove (2) without loss of generality we may assume that 
$t_f=0$. There are two linear forms $\xi_f^+$, $\xi_f^-\colon \C^2\to \C$ such that $\xi_f^+\circ L_f=\alpha \xi_f^+$
and $\xi_f^-\circ L_f=\beta \xi_f^-$. They determine two holomorphic $1$-forms on $A$, the kernels of which define 
two linear foliations: the stable and unstable foliations of $f$. Both have  dense leaves. The class $\theta^+_f$ 
is represented by $\xi_f^+\wedge{\overline{\xi_f^+}}$, up to some positive multiplicative factor. If $C\subset A$ 
is a complex analytic curve, there is an open subset ${\mathcal{U}}$ of 
$A$ in which $C$ intersects both foliations transversely. If $y$ is a torsion point, $y$ is $f$-periodic, and its stable manifold being dense, it intersects $C$. 
Thus, $f^n(C)$ accumulates every torsion point, and is
dense in $A$. In particular, $C$ is not invariant, as claimed. 
\end{proof}

%%%%%%%%%%%%%%%%%%%%%%
\subsection{Kummer structures}\label{par:Kummer_structures}
%%%%%%%%%%%%%%%%%%%%%%
\subsubsection{Kummer pairs}\label{par:Kummer_pairs} Let $X$ be a compact complex surface, and let
 $\Gamma$ be a subgroup of $\Aut(X)$. By definition  $(X, \Gamma)$ is a \textbf{Kummer group} if 
there is an abelian surface $A$, a finite subgroup $G$ of $\Aut(A)$, a subgroup $\Gamma_A$ of $\Aut(A)$ containing $G$,
and a birational morphism $q_X\colon X\to A/G$ such that:
\begin{enumerate}[(a)]
\item $\Gamma_A$ normalizes $G$. Thus, if $q_A\colon A\to A/G$ is the quotient map, there is 
a homomorphism $h\in  \Gamma_A \mapsto \overline{h} \in  \Aut(A/G)$ such that $q_A\circ h=\overline{h}\circ q_A$ for every $h\in \Gamma_A$;
we shall denote by $\overline \Gamma_A$ the image of this homomorphism.
\item the birational map $q_X$ is $\Gamma$-equivariant: there is a homomorphism $\Gamma\ni f \mapsto \overline f \in \Aut(A/G)$, whose image is denoted by $\overline \Gamma$, 
 such that
$q_X\circ f=\overline f \circ q_X$ for every $f\in \Gamma$.
\item the subgroups $\overline \Gamma$ and ${\overline{\Gamma}}_A$ of $\Aut(A/G)$ coincide. 
\end{enumerate}
To each $f\in \Gamma$ corresponds an element $f_A$ of $\Gamma_A$, unique up to composition with elements of $G$; the type
of $f$ as an automorphism of $\Aut(X)$ coincides with the type of $f_A$ as an automorphism of $A$, and 
$\lambda(f)=\lambda(f_A)$.  

\begin{rem} Consider a section $\Omega_A$ of the canonical bundle $K_A$ such that $\int_A\Omega_A\wedge {\overline{\Omega_A}}=1$; 
it is unique up to multiplication by a complex number of modulus one. In particular, the volume 
form $\vol_A=\Omega_A\wedge {\overline{\Omega_A}}$ is invariant under $\Aut(A)$. 
The quotient of $\vol_A$ by 
the action of $G$ determines a probability measure on $A/G$, and then on $X$. 
This probability measure coincides with the measure of maximal entropy $\mu_f$ for every $f\in \Gamma_{\mathrm{lox}}$.
\end{rem}

From the definition of a Kummer group,   Proposition~\ref{pro:finite_orbits_torus} and Remark~\ref{rmk:torsion}
we get:

\begin{pro}\label{pro:Kummer_finite_orbits}
If $(X, \Gamma)$ is a  
Kummer group with at least one finite orbit, then its finite orbits are dense in $X$ for the euclidean topology. 
Furthermore there exists a dense, $\Gamma$-invariant, Zariski open subset in which all 
periodic points of loxodromic elements of $\Gamma$ coincide. 
\end{pro}

\subsubsection{Classification of Kummer examples}\label{par:Kummer_classification} 
Let us consider, firstly, the case of an infinite cyclic group generated by a loxodromic Kummer example $f$.
From the classification given in~\cite{Cantat-Favre, cantat-favre-corrigendum},
we may assume that  the finite group $G$ is  
a cyclic group fixing the origin of $A$; in other words, $G$ is 
induced by a cyclic subgroup of $\GL_2(\C)$ preserving the lattice $\Lambda$ such that $A=\C^2/\Lambda$. 
And there are only seven possibilities:
\begin{enumerate}
\item $G=\{ \id\}$ and $X$ is a blow-up of an abelian surface.
\item $G=\{\id, -\id\}$ and $A/G$ is a Kummer surface, in the classical sense; in particular $X$ is a blow-up of a K3 surface.
\item $A$ is the torus $(\C/\Z[\ii])^2$ and $G$ is the group of order $4$ generated by $\ii \id$ ; in this case  $X$ is a rational surface.
\item $A$ is the torus $(\C/\Z[\exp(2\ii \pi/3)])^2$ and $G$ is the group of order $3$ generated by $\exp(2\ii \pi/3) \id$; in this case  $X$ is a rational surface.
\item $A$ is the torus $(\C/\Z[\exp(2\ii \pi/3)])^2$ and $G$ is the group of order $6$ generated by $\exp(\ii \pi/3)  \id$; in this case  $X$ is a rational surface.
\item Let $\zeta_5$ be a primitive fifth root of unity. The cyclotomic field $\Q[\zeta_5]$ has two distinct non-conjugate embeddings in $\C$,
$\sigma_1$ and $\sigma_2$ determined by $\sigma_1(\zeta_5)=\zeta_5$ and $\sigma_2(\zeta_5)=\zeta_5^2$. The ring of integers 
coincides with $\Z[\zeta_5]$ and its image by $\sigma=(\sigma_1,\sigma_2)$ is a lattice $\Lambda_5\subset \C\oplus \C$. The abelian
surface $A$ is the quotient $\C^2/\Lambda_5$. The group $G$ is generated by the diagonal linear map 
\begin{equation}
(x,y)\mapsto (\zeta_5 x, \zeta_5^2 y)
\end{equation}
and has order $5$.  
Here, $X$ is rational too.
\item As in the previous example, $A=\C^2/\Lambda_5$, but now $G$ has order $10$ and is generated by 
$(x,y)\mapsto (-\zeta_5 x, \zeta_5 y)$, and again $X$ is  rational.  
\end{enumerate}
%So, except in the first two cases, $X$ is rational; in particular, $X$ is never an Enriques surface. 

These constraints on $(A,G)$ apply to non-elementary Kummer groups;
in particular we shall always assume that $G$ is cyclic and fixes the neutral element of $A$. 

In  Cases~(1) to (5) of the above list, the abelian surface is $\C^2/(\Lambda_0\times\Lambda_0)$ for some lattice $\Lambda_0$ in $\C$. The natural action of $\GL(2, \Z)$ on 
$\C^2$ preserves $\Lambda_0\times \Lambda_0$, and induces a non-elementary subgroup of $\Aut(A)$, which commutes to $G$; as a consequence, it  determines  also a non-elementary subgroup of $\Aut(A/G)$. 
On the other hand, cases~(6) and~(7) do not appear: 

\begin{lem}\label{lem:67} 
If $(X, \Gamma)$ is a non-elementary Kummer group, then $G$ is generated by a homothety and the   quotient $A/G$ is not of type (6) or (7) in the classification above. 
\end{lem}

\begin{proof} The group $\Gamma_A$ permutes the fixed points of $G$. So, the stabilizer $\Gamma_A^\circ={\mathrm{Stab}}_{\Gamma_A}(0)$ 
of the neutral element is a finite index, non-elementary subgroup of $\Gamma_A$. Pick any loxodromic element $f$ in $\Gamma_A^\circ$; 
it acts by conjugacy on $G$, which is finite, so there is  a positive iterate such that 
$f^n\circ g=g\circ f^n$ for all $g\in G$. Near the orgin, 
$f^n$ and $g$ are two commuting linear transformations, $f^n$ has two eigenvalues, of modulus $<1$ and $>1$ respectively, and
$g$ must preserve the corresponding  stable and unstable directions of $f$. Since $\Gamma_A^\circ$ is non-elementary, these tangent directions form an infinite set as $f$ varies
in the set of loxodromic elements of $\Gamma_A^\circ$, so $g$ is a homothety, and we are done. \end{proof}

\subsubsection{Invariant curves} 
We keep the  notation from 
the previous paragraphs and consider a non-elementary Kummer group $(X, \Gamma)$. 
The singularities of $A/G$ are cyclic quotient singularities, and $X$ dominates a minimal resolution of 
$A/G$. 

Let us examine Case (2), when $G=\{ \id, -\id\}$. Then $G$ has $16$ fixed points, and to resolve the $16$ singularities 
of $A/G$ one can proceed as follows. First, one blows up the fixed points, creating $16$ rational curves. Then one lifts the action of 
$G$ to the blow-up $\hat{A}$. If $E$ is one of the exceptional divisors, then $G$ fixes $E$ pointwise and acts as $w\mapsto -w$ 
transversally, so locally the quotient map can be written $(w,z)\mapsto (w^2,z)$, with $E=\{ w=0\}$ giving rise to a smooth rational 
curve of self-intersection $-2$ on $\hat{A}/G$. This construction provides the minimal resolution $X_{\mathrm{min}}=\hat{A}/G$ of $A/G$, the singularities being 
replaced by disjoint $(-2)$-curves. 
 Cases (3), (4), (5) can be handled with a similar process because if $x\in A$ is stabilized 
by a subgroup $H$ of $G$, then $H$ is locally given around $x$ as a cyclic group of homotheties; so, in the minimal resolution 
of $A/G$ the singularities are replaced by disjoint rational curves $E_i$ of negative self-intersection $E_i^2\in \{-2, \ldots, -6\}$. 
Cases (6) and (7) are more delicate 
 however, by Lemma~\ref{lem:67}, we don't need to deal with them  (see Appendix~\ref{par:Kummer6} for Case (6)).

\begin{lem}\label{lem:invariant_curve_kummer} 
Let $(X, \Gamma)$ be a non-elementary Kummer group on a smooth projective surface. 
Then: \begin{enumerate}[{\em (1)}]
\item $X$ is abelian if and only if $\Gamma$ admits no invariant curve; 
\item any connected $\Gamma$-periodic curve $D$  is a smooth rational curve, and the induced dynamics 
of  $\mathrm{Stab}_D(\Gamma)$ on $D$ has no periodic orbit.\end{enumerate}
Moreover, $D_f=D_\Gamma$ for every $f\in \Gamma_{\mathrm{lox}}$. 
\end{lem}

\begin{proof}
The minimal resolution $X_{\min}$ of $A/G$ is unique, up to isomorphism (see~\cite[\S III.6]{BHPVDV}, Theorems~(6.1) and (6.2), and their proofs). Thus,  $X$ dominates $X_{\min}$ and every $f\in \Gamma$ preserves the exceptional divisor of the morphism $X\to X_{\min}$ and induces
an automorphism $f_{\min}$ of $X_{\min}$. 
%(\footnote{To see this, denote by $E$ the exceptional divisor of the morphism $X_{\min} \to A/G$ then $f_{\mathrm{min}}$ induces a birational transformation of $X_{\mathrm{min}}$ which is regular on the complement of $E$. Assume that there is an indeterminacy point $q$ on $E$, and resolve it, writing $f=p\circ q^{-1}$ for some sequence of blow-ups $q\colon Y \to X_{min}$ and a birational morphism $p:Y\to X_{min}$. Since every irreducible component of $E$ has self-intersection $<-1$, there is no curve of self-intersection $-1$ on the total transform $q^*E$, hence $p$ cannot contract any curve of $q^*E$, a contradiction. One can also argue as follows. By a theorem of Zariski ( O. Zariski, The reduction of the singularities of an algebraic surface. Ann. Math. (2) 40,639-689 (1939).), a minimal resolution of singularities can be achieved via a sequence of normalizations and blowups of points i.e., of maximal ideals. Since both operations are equivariant, this yields an equivariant resolution}).
In particular $\Gamma$ admits  an invariant curve, unless $G=\set{\mathrm{id}}$ and $X=X_{\min}  = A$. 
Conversely in that case $\Gamma$ has no invariant curve, by Lemma~\ref{lem:lox_on_tori}.  This proves the first assertion.

Let us prove the second assertion for the induced group
 $\Gamma_{\min}\subset \Aut(X_{\mathrm{min}})$.  Let $E$ be a connected periodic curve for $\Gamma_{\min}$.
If $f_{\min}\in \Gamma_{\min}$ is loxodromic, it  comes from an Anosov map $f_A\colon A\to A$, as in Lemma~\ref{lem:lox_on_tori}, and $f_A$ does not 
have any periodic curve. Since $E$ is $f_{\min}$-periodic, it  is contained in the exceptional divisor of the resolution $X_{\min}\to A/G$; 
as explained before the lemma,   this divisor is a disjoint union of rational curves, so $E$ is one of these rational curves $E_x = q_{X_{\mathrm{min}}}\inv(q_A(x))$, where $x\in A$ has
a nontrivial stabilizer $G_x\subset G$. In particular $x$ is fixed by a  
finite index subgroup $\Gamma_{A,x}$ of~$\Gamma_A$. 
Now since $\Gamma$ is non-elementary, $\Gamma_A$ and $\Gamma_{A,x}$ are  non-elementary as well, and 
since the action of $\Gamma_{A,x}$ on $A$ is by affine transformations, its action on the exceptional divisor $E_x$  
is that of  a non-elementary subgroup of $\PGL_2(\C)=\Aut(E_x)$.  
In particular, $\Gamma_{A,x}$  does not admit  any finite orbit in  $E_x$. 

The birational morphism $\pi: X\to X_{\mathrm{min}}$ is equivariant with respect to $\Gamma$ and its image
$\Gamma_{\mathrm{min}}$ in $\Aut(X_{\mathrm{min}})$.  So, $\pi^{-1}$ blows up periodic orbits of 
$\Gamma_{\mathrm{min}}$. 
The last few lines show that, when such a periodic point $y\in X_{\mathrm{min}}$ is blown up, 
firstly  $y$ does  not lie on the exceptional locus of $X_{\mathrm{min}}\to A/G$, and secondly the
exceptional divisor $E_y$ does not contain  any finite orbit. So, $X$ is obtained by simple blowups centered on a finite set 
 of distinct  periodic points of $\Gamma_{\min}$,  every connected 
component of the exceptional locus of $q_X$ is a smooth rational curve, and there is no  $\Gamma$-periodic point in these curves.
\end{proof}

%%%%%%%%%%%%%%%%%%%%%%%%%%%%%%%%%%%%%%
%%%%%%%%%%%%%%%%%%%%%%%%%%%%%%%%%%%%%%
\section{Unlikely intersections for non-elementary groups}\label{par:finite_orbits}
%%%%%%%%%%%%%%%%%%%%%%%%%%%%%%%%%%%%%%
%%%%%%%%%%%%%%%%%%%%%%%%%%%%%%%%%%%%%%
This section is devoted to the proof of Theorem \ref{mthm:main}.
Since the proof comprises many steps,   let us  start with a rough outline 
of the argument.  

\subsection{Strategy of the proof}\label{par:strategy_5} If $\Gamma$ has a Zariski dense set $F$ of finite orbits, a standard argument shows that there
exists a sequence $(x_n)\in F^\N$  which is generic in the sense that it escapes any fixed proper subvariety 
$Y\subset X$.   
Let $m_{x_n}$ be the probability measure  
equidistributed over the Galois orbit of $x_n$.
We want to use arithmetic equidistribution to show that the  sequence  of measures   $(m_{x_n})$ converges when $n\to \infty$. 
For this, we need  a height function $h_L$ associated to some line bundle $L$ on $X$ satisfying appropriate   
positivity properties, and such    that $h_L$ vanishes on $F$. In \S \S\ref{subs:kawaguchi_currents}--\ref{subs:arithmetic_equidistribution} 
we construct such height functions: they are associated to  the choice of 
certain finitely supported probability measures $\nu$  on $\Gamma$. Indeed 
 to such a measure, we associate the linear 
endomorphism $P_\nu^*=\sum \nu(f) f^*$ on the N\'eron-Severi group of $X$, and we construct a big and nef line bundle $L$ 
such that $P_\nu^*[L]=\alpha(\nu)[L]$ for some $\alpha(\nu)>1$; then, $h_{L}$ will be a Weil height that satisfies the invariance 
$\sum \nu(f)h_L\circ f = \alpha(\nu) h_L$.
The arithmetic equidistribution theorem of Yuan shows that  the measures $m_{x_n}$ converge to 
a measure $\mu_\nu=S_\nu\wedge S_\nu$, where $S_\nu$ is a dynamically defined 
 closed positive current with cohomology class equal to $[L]$. On the other 
hand, the measures $m_{x_n}$, hence  their limit $\mu$, do  not depend on $\nu$. 
 As we vary the choice of $\nu$, there is enough freedom in the construction 
 to show that for every 
$f\in \Gamma_{\mathrm{lox}}$, $\mu_\nu$ can be made arbitrary close to the maximal entropy measure 
$\mu_f$. It follows that $\mu_f=\mu$ is independent of $f$ and is $\Gamma$-invariant. In 
\S \ref{subs:active}, using the dynamics of parabolic elements of $\Gamma$ 
 we deduce that $\supp(\mu) = X$. Then the classification of $\Gamma$-invariant measures 
 from \cite{cantat_groupes, invariant} implies that $\mu$ has a smooth density, and   the main result 
 of \cite{cantat-dupont} shows that every $f\in \Gamma_{\mathrm{lox}}$ is a Kummer example 
 (Theorem \ref{thm:finite_orbits_zariski} in \S \ref{par:finite_orbits_zariski}). 
 At this point the Kummer structure may a priori depend on $f$, as in Example~\ref{eg:kummer_non_kummer} below. 
This issue is solved in  \S \ref{par:Kummer} by adding  an  argument based on Theorem \ref{mthm:invariant_curve_loxodromic},  
which finally shows that  $(X, \Gamma)$   is a Kummer group.

\begin{eg}\label{eg:kummer_non_kummer}
Let $X$ be a Kummer surface possessing both a Kummer automorphism $f$ and a non-Kummer one $h$, as in \cite{Keum-Kondo}. 
Then, $f$ and $h\circ f \circ h^{-1}$ are  two Kummer automorphisms which are not associated  with the same Kummer structure; the pair
$(X, \langle f, h\circ f \circ h^{-1}\rangle)$ is not a Kummer group.
\end{eg}

%%%%%%%%%%%%%%%%%%%%%%%%%%%%%%%%%%%%%%
\subsection{Kawaguchi's currents} \label{subs:kawaguchi_currents}
%%%%%%%%%%%%%%%%%%%%%%%%%%%%%%%%%%%%%%
%%%
\subsubsection{Action on $H^{1,1}$}  
%%%
Let $X$ be a Kähler surface and let $\nu$ be a probability measure on $\Aut(X)$ 
satisfying the (exponential) moment assumption  
\begin{equation}\label{eq:C1_moment}
\int  \lrpar{\norm{f}_{C^1} +  \norm{f\inv}_{C^1}}^2 d\nu(f) <+\infty. 
\end{equation}
By \cite[Lem. 5.1]{stiffness} this implies   the  following moment condition on the 
cohomological  action  of $\Gamma$:
\begin{equation}\label{eq:first_cohomological_moment}
\int  \norm{ f^*} d\nu(f) <+\infty ,
\end{equation}
where $f^*$ is the endomorphism of $H^{2}(X;\R)$ determined by $f$ and $\norm{ \cdot }$ is any operator norm. (For the proof of Theorem \ref{thm:finite_orbits_zariski}, we will only consider
 finitely supported measures so the moment conditions will be trivially satisfied.) 
Let $\Gamma_\nu$ be the
subgroup of $\Aut(X)$  generated by the support of $\nu$.
We define a linear endomorphism $P_\nu$ of $H^{2}(X;\C)$ by  
\begin{equation}
P_\nu(u)=\int  f^*(u) \; d\nu(f). 
\end{equation}

The following lemma is a strong version of the Perron-Frobenius theorem (see~\cite{Birkhoff}). 
\begin{lem}\label{lem:unique_w}
Assume that $\Gamma_{\nu}$ does not fix any isotropic line in $H^{1,1}(X;\R)$. 
Then, $P_\nu$ has a unique eigenvector $w_\nu\in H^{1,1}(X;\R)$ such that $w_\nu^2=1$ and $\langle w_\nu\vert [\kappa_0]\rangle >0$. 
This eigenvector is big and nef. The eigenvalue  $\alpha(\nu)$ such that
\[
P_\nu(w_\nu)= \alpha(\nu) w_\nu
\]
is larger than $1$ and coincides with the spectral 
radius of $P_\nu$; the multiplicity of $\alpha(\nu)$ is equal to $1$, and all other eigenvalues $\beta \in \C$ of $P_\nu$ satisfy $\vert \beta \vert < \alpha(\nu)$.
\end{lem}

\begin{proof} 
Let $u$ be an isotropic vector, contained in the closure of the positive cone ${\mathrm{Pos}}(X)$.
Then, $P_\nu(u)$ has positive self-intersection, except if $f^*(u)\in \R u$ for $\nu$-almost every $f$, and then
for all $f$ in the support of $\nu$ by continuity. Hence, the hypothesis implies 
that $P_\nu$ maps the positive cone strictly inside itself. From the Perron-Frobenius theorem
(\cite{Birkhoff}), $P_\nu$ has an eigenvector $w_\nu$ in the interior of this cone such that 
\begin{enumerate}
\item[(1)] $w_\nu$ is a dominant eigenvector: the eigenvalue $\alpha(\nu)$ such that 
$P_\nu(w_\nu)=\alpha w_\nu$ is 
the spectral radius of $P_\nu$. 
\end{enumerate}
Since $w_\nu$ is in the interior of ${\mathrm{Pos}}(X)$, we may assume that $w_\nu$ is in $\Hyp_X$. 
Projectively,  $\P(P_\nu)$ contracts strictly the Hilbert metric of the convex set $\P({\mathrm{Pos}}(X))=\P(\Hyp_X)$; so, 
 $P_\nu$ does not have any eigenvector in $\Hyp_X$ besides $w_\nu$ itself, and the $\P(P_\nu)$-orbit of
any point of $\P(\Hyp_X)$ converges towards $\P(w_\nu)$.
The K\"ahler cone is also invariant: $P_\nu(\Kah(X))\subset \Kah(X)$.  
Hence, $w_\nu$ is nef (it is in the closure of $\Kah(X)$). 
It is big because it is nef and has positive self-intersection (see e.g. \cite[Thm 2.2.16]{Lazarsfeld:Book1}).
Thus, 
\begin{enumerate}
\item[(2)] this vector $w_\nu$ is the unique eigenvector of $P_\nu$ in ${\mathrm{Pos}}(X)$ up to a positive scalar multiple; it is nef and big.
\end{enumerate}
If $w'$ is another eigenvector with eigenvalue $\alpha(\nu)$ and $w'\notin \R w_\nu$, the plane $\Vect(w',w_\nu)$ intersects $\Hyp_X$
along a geodesic, all of whose elements satisfy $P_\nu(v)=\alpha(\nu) v$. From assertion (2), we get a contradiction. 
Now, if the multiplicity of $\alpha(\nu)$ in the characteristic polynomial of $P_\nu$ is 
larger than 1, there exists a vector $u\notin \R w_\nu$
such that $P_\nu(u)=\alpha(\nu) u + \alpha(\nu) w_\nu$. In the plane $\Vect(w_\nu, u)$ 
the cone ${\mathrm{Pos}}(X)$ is bounded 
by two rays;  changing $u$ into $u+cw_\nu$ if necessary (for some $c\in \R$), these two rays are 
$\R_+(w_\nu + a u)$ and $\R_+(w_\nu - bu)$ for some positive real numbers $a$ and $b$.
But then, the image of $\R_+(w_\nu - bu)$ by $P_\nu$ is 
 not contained in ${\mathrm{Pos}}(X)$, a contradiction. It follows that
\begin{enumerate}
\item[(3)] $\alpha(\nu)$ is a simple root of the characteristic polynomial of $P_\nu$.
\end{enumerate}
Finally, 
\begin{enumerate}
\item[(4)] if $P_\nu (v)=\beta v$ is any complex eigenvector with $v\notin \C w_\nu$, 
then $\vert \beta \vert < \alpha(\nu)$.
\end{enumerate}
Indeed, if $\beta$ is another eigenvalue with $\vert \beta \vert=\alpha(\nu)$, then $\beta=\alpha(\nu) e^{2{\mathsf{i}} \pi \theta}$ for some $\theta\in \R$, 
and there is a plane $V\subset H^{1,1}(X;\R)$ on which $P_\nu$ acts as a similitude of strength $\alpha(\nu)$ and angle $2\pi\theta$; pick $v$ in this 
plane, then if $\e\in \R_+^*$ is small, the vector $w_\nu+ \e v$ is in ${\mathrm{Pos}}(X)$ and the $\P(P_\nu)$-orbit of $\P(w_\nu+ \e v)$ in $\P(\Hyp_X)$
stays   at constant  distance from
 $\P(w_\nu)$, contradicting the contraction property of $P_\nu$. 
This concludes the proof. \end{proof}

\begin{eg}[See~\cite{Cantat:Milnor}, \S 2, and \cite{cantat-dupont}, \S 2.2]\label{eg:loxodromic_first_part}
Let $f$ be a loxodromic automorphism, and take $\nu$  to 
be the probability measure $p\delta_f+q\delta_{f^{-1}}$
with $p$, $q \geq 0$ and $p+q=1$. Note that $\Gamma_\nu=f^{\Z}$ does not satisfy the assumption of Lemma \ref{lem:unique_w}. 

Then $P_{\nu}=p f^*+q(f^{-1})^*$ preserves the 
$f^*$-invariant plane $\Pi_f\subset H^{1,1}(X;\R)$. 
If $p> q$, the spectral radius of $P_{\nu}$ is equal to $p\lambda(f) + q/\lambda(f)$, and
the corresponding eigenspace is the isotropic line of $\Pi_f$ corresponding to the eigenvalue $\lambda(f)$ of $f$; if $p< q$, the spectral radius  is equal to $p/\lambda(f) + q\lambda(f)$
and the eigenspace is the other isotropic line in $\Pi_f$. If $p=q=1/2$, then $(P_\nu)_{\vert \Pi_f}$ is the scalar multiplication by 
\begin{equation}
\alpha(f)=\frac{1}{2}(\lambda(f)+1/\lambda(f))
\end{equation}
and all vectors $u\in \Pi_f$ satisfy $P_\nu(u)=\alpha(f) u$. 
This example shows that  the previous lemma fails for $\nu$, whatever the  values of $p$ and $q$ are: the dominant eigenvector
is at the boundary of the hyperbolic space, or is not unique. 
\end{eg}

\subsubsection{Stationary currents}
Let us borrow some notation from~\cite[\S 6.1]{stiffness}: we fix K\"ahler forms $\kappa_i$, 
the cohomology classes of which 
provide a basis $([\kappa_i])$ of $H^{1,1}(X;\R)$. Then, if $a$ is any element of $H^{1,1}(X;\R)$, there is a unique 
$(1,1)$-form $\Theta(a)=\sum_i a_i \kappa_i$ in $\Vect(\kappa_i, 1\leq i\leq h^{1,1}(X))$ whose class $[\Theta(a)]$ is equal to $a$. 
If $S$ is any closed positive current of bidegree $(1,1)$, then 
$S=\Theta([S])+dd^cu_S$ for some  upper semi-continuous function
$u_S\colon X\to \R$: this function is locally the difference of a plurisubharmonic
 function and a smooth function, and it is unique up to an additive constant.

The following proposition is essentially due to Kawaguchi, who proved it in~\cite{Kawaguchi:Crelle} under slightly more restrictive
assumptions. 

\begin{pro}\label{pro:kawaguchi} 
Let $X$ be a compact Kähler surface, and let $\vol$ be a smooth volume form on $X$. Let
$\nu$ be a probability measure on $\Aut(X)$ satisfying the moment condition~\eqref{eq:C1_moment}.
Assume that there exists  $w\in H^{1,1}(X;\R)$ and $\alpha >1$ satisfying
\begin{itemize}
\item[(i)] $P_\nu (w) = \alpha w$;
\item[(ii)] $w$ is big and nef and $w^2=1$.
\end{itemize}
Then, there is a unique closed positive current $S_\nu$  such that 
\begin{equation}\label{eq:kawaguchi}
\int f^*(S_\nu) \; d\nu(f)=\alpha S_\nu \quad{\text{and}} \quad [S_\nu]=w.
\end{equation}
This current has continuous potentials: 
$S_\nu=\Theta(w)+dd^c(u)$ for a unique continuous function $u$ such that $\int_X u \, \vol=0$.
In particular, the product $S_\nu\wedge S_\nu$ is a well-defined  probability measure on $X$.
\end{pro}

Note that here $\Gamma_\nu$ is not assumed to be non-elementary. Actually this proposition will be applied in two situations: 
\begin{itemize}
\item when $\Gamma_\nu$ does not fix any boundary point of $\Hyp_X$,
$w=w_\nu$, and $\alpha=\alpha(\nu)$, as in Lemma~\ref{lem:unique_w}; 
\item when $\nu=\frac{1}{2}(\delta_f+\delta_{f^{-1}})$ for some loxodromic automorphism and $w$ is a fixed
point of $\frac{1}{\alpha(f)}P_\nu$ in $\Hyp_X$, as in Example~\ref{eg:loxodromic_first_part}.
\end{itemize}

\begin{proof}[Proof of Proposition~\ref{pro:kawaguchi}]
Fix a K\"ahler form $\kappa_0$ on $X$, as in~\S \ref{par:basics_on_par_and_lox}.
Let $\beta$ be a smooth form with $[\beta]=w$. For simplicity, we denote by the same letter 
$P_\nu$ the operator $\int f^* ( \cdot ) d\nu(f)$ 
acting on the cohomology, on differential forms, or on currents.
Write $(\alpha^{-1}P_\nu)\beta=\beta+dd^c(h)$
for some smooth function $h$. Then, 
\begin{equation}\label{eq:geometric_iteration}
\left(\frac{1}{\alpha}P_\nu\right)^n \beta=\beta + dd^c\left( \sum_{j=0}^{n-1} \frac{1}{\alpha^j}(P_\nu)^j(h)\right),
\end{equation}
where $(P_\nu)^j(h)(x)=\int  h\circ f(x) d\nu^{\star j}(f)$ is the average of $h$ with  
respect to the probability measure $\int \delta_{f(x)} d\nu^{\star j}(f) $ and   $\nu^{\star j}$ denotes the $j$-th convolution of $\nu$. 
The supremum of $\vert (P_\nu)^j(h)\vert$ on $X$ satisfies $\norm{(P_\nu)^j(h)}_\infty\leq \norm{h}_\infty$ for all $j\geq 1$.
From this we deduce that the series on the right 
hand side of Equation~\eqref{eq:geometric_iteration} converges geometrically: 
\begin{equation}\label{eq:kawa_geom_conv}
\sum_{j=k}^{\ell} \norm{\frac{1}{\alpha^j}(P_\nu)^j(h)}_\infty \leq  \frac{\alpha-\alpha^{k-\ell}}{\alpha-1}\frac{\norm{h}_\infty}{\alpha^k}\leq  \frac{\alpha}{\alpha-1}\frac{\norm{h}_\infty}{\alpha^k}.
\end{equation}
Thus, if we set $h_\nu=\sum_{j\geq 0}\frac{1}{\alpha^j}(P_\nu^*)^j(h)$ and  $S_\nu=\beta + dd^c(h_\nu)$ we see that 
$S_\nu$ is a closed current 
which  satisfies  $P_\nu(S_\nu)=\alpha S_\nu$. 
Furthermore, since $\beta$ can be written as a difference of positive closed currents (e.g. by writing $\beta = (\beta+ C\kappa_0) - C\kappa_0$ for a suitably large positive  $C$), we infer that $S_\nu = \lim_{n\to\infty} (\alpha\inv P_\nu)^n \beta$
 is also a difference of positive closed currents. 
Let us prove that $S_\nu$ is positive and unique.

By (ii),  $w$ is  nef, which implies that the set $\Cur(w_\nu)$ of closed positive currents with class $w_\nu$ is not empty. 
The linear operator $\alpha^{-1}P_\nu$  
preserves this compact convex set; as a consequence, $\Cur(w_\nu)$ contains a fixed point $T$ of  $\alpha^{-1}P_\nu$.  
So we can fix a non-zero    closed  positive current $T$ satisfying $P_\nu(T)=\alpha T$ and $[T]= [S_\nu]$. 
Since $T-S_\nu$ is cohomologous to zero and expresses as  a difference $T_1-T_2$ of positive closed currents, 
according to~\cite[Lemma~6.1]{stiffness}, we can write $T-S_\nu=dd^c(h)$ where $h=u_1-u_2$, and  $u_i$ is 
an    $(A\kappa_0)$-psh function   ($A$ depends only of the mass of the $T_i$).  Changing 
$u_i$ into $u_i-\int_Xu_i\vol$ we may assume that $\int_Xu_i \vol=0$ for $i=1,2$. 
From the invariance of $T-S_\nu$ under $\alpha\inv P_\nu$ we obtain that 
\begin{equation}
\frac{1}{\alpha}P_\nu(h)=h+c
\end{equation}
for some constant $c\in \R$; thus, $\alpha^{-n}P_\nu^n(h)=h+c_n$ where $c_n$ converges 
geometrically towards some real number $c_\infty$. From \cite[Lemma~6.6]{stiffness}, there is a constant $C>1$ such that 
\begin{equation}\label{eq:jac}
\int_X \frac{1}{\alpha^n}\vert P_\nu^n(h)\vert \vol\leq \frac{C}{\alpha^n}\int_X 
\log(C\norm{ \Jac(f^{-1})}_\infty) \, d\nu^{\star n}(f)
\end{equation}
for all $n\geq 1$.
Thanks to the moment condition~\eqref{eq:first_cohomological_moment} and the subadditivity property $$\log(\norm{ \Jac((f\circ g)^{-1})}_\infty)\leq \log(\norm{ \Jac(f^{-1})}_\infty)+\log(\norm{ \Jac(g^{-1})}_\infty),$$ we see that 
\begin{align*}
\int_X \log(\norm{ \Jac(f^{-1})}_\infty) \, d\nu^{\star n}(f)
 &= \int_X \log(\norm{ \Jac(f_n^{-1}\circ \cdots \circ f_1\inv)}_\infty) \, d\nu (f_1)\cdots d\nu (f_n)\\
 & \leq \sum_{j=1}^n  
  \int_X \log(\norm{ \Jac(f_j^{-1} )}_\infty) \, d\nu (f_1)\cdots d\nu (f_n)  \\& = O(n),
  \end{align*}
so the right-hand side of  the inequality~\eqref{eq:jac} tends to $0$ as $n$ goes to $+\infty$.  
Hence  $\alpha^{-n}P_\nu^n(h)$ converges towards $0$ in $L^1(X;\vol)$.
 Since $\alpha^{-n}P_\nu^n(h)=h+c_n$ also converges  towards $h+c_\infty$, we deduce that $h$ is a constant, namely $h=-c_\infty$, and that $T=S_\nu$. In particular 
$S_\nu$ is   positive, and is the unique positive closed current satisfying~\eqref{eq:kawaguchi}.

Finally, since $S_\nu$ has continuous potentials, 
$S_\nu\wedge S_\nu$ is a well-defined positive measure; its total 
mass is $1$ because $1=w^2=[S_\nu]^2=\int_X S_\nu\wedge S_\nu$.\end{proof}

\begin{rem}\label{rem:kawaguchi} 
Here is another setting, closer to \cite{Kawaguchi:Crelle}, in which the same result holds. Let $X$ be a (singular)
complex projective surface and $L$ be an ample line bundle on $X$. Pick any integer $m\geq 1$ such that $mL$ is very ample, 
and consider the Kodaira-Iitaka embedding 
$\Phi_{mL}\colon X\to \P^N$, with $\P^N:=\P(H^0(X,mL)^\vee)$. Then, $L$ coincides with $\Phi_{mL}^*({\mathcal{O}}(1)_{\vert X})$ and any Fubini-Study form 
$\omega$ determines a smooth $(1,1)$-form $\kappa_{mL}:=\Phi_{mL}^*(\omega)$ on $X$ (see \cite[\S 1]{Demailly:MemSMF} for  forms, currents, and potentials on singular complex spaces). The form $\kappa_L:=\frac{1}{m}\kappa_{mL}$ is locally equal to $dd^c(v)$ for some 
smooth function, namely $v=\frac{1}{m}u\circ \Phi_{mL}$ where $u$ is a local potential of $\omega$ in $\P^N$.
Now consider a probability measure $\nu$ on $\Aut(X)$ such that $P_{\nu}^*L=\alpha L$, with $\alpha >1$. Then, the proof of Proposition~\ref{pro:kawaguchi}
applies, and provides a closed positive current $S_\nu$ on $X$ with continuous potentials such that $P_\nu^*S_\nu=\alpha S_\nu$; this is proven 
in \cite[Thm. 3.2.1]{Kawaguchi:Crelle} (Kawaguchi assumes $X$ to be smooth, but this is only used in the first line of the proof to introduce the smooth form $\kappa_L$).
\end{rem}

\begin{eg}\label{eg:loxodromic_second_part} As in Example~\ref{eg:loxodromic_first_part}, consider the case 
$\nu=\frac{1}{2}(\delta_f+\delta_{f^{-1}})$, where $f$ is a loxodromic automorphism. 
There are two closed positive currents $T^+_f$ and $T^-_f$, with continuous potentials, such that $f^*(T^\pm_f)=\lambda(f)^{\pm 1} T^\pm_f$;
they are unique up to a positive scalar factor and their classes generate the isotropic lines 
$\R \theta^\pm_f$  (see~\cite[\S 5]{Cantat:Milnor}). 
By convention, we choose them so that $\langle [T^+_f]\vert   [T^-_f]\rangle =1$, 
or equivalently, $T^+\wedge T^-$ is a probability measure; to determine them uniquely we further 
require   $\langle [T^+_f]\vert   [\kappa_0]\rangle =\langle [T^-_f]\vert   [\kappa_0]\rangle$.  
 Beware that   this normalization is different from 
that of $\theta_f^\pm$ so a priori $[ T^\pm_f]\neq \theta_f^\pm$. 
Pick a class $w=a[T^+_f]+b[T^-_f]$ with $a$, $b>0$ such that $w^2=1$ (equivalently $2ab=1$).
By uniqueness, 
the current $S_\nu$ provided by Proposition~\ref{pro:kawaguchi}
is equal to $aT^+_f+bT^-_f$; in particular, the measure $S_\nu\wedge S_\nu$ is equal to $T^+_f\wedge T^-_f$ which, in turn, is
the measure of maximal entropy $\mu_f$ (see \S \ref{par:basics_on_par_and_lox} and~\cite[\S \S 5.2, 8.2]{Cantat:Milnor}). \qed
\end{eg}

\subsubsection{Continuity properties  of stationary currents}\label{subset:continuity}
Now, consider a sequence of probability measures $(\nu_n)$ such that the support of each $\nu_n$ is
contained in a fixed finite set $\{ f_1, \ldots, f_m\}$: 
\begin{equation}
\nu_n=\sum_i \nu_n(f_i) \delta_{f_i},
\end{equation}
with coefficients  in the simplex of dimension $m-1$ determined by the constraints $\nu_n(f_i)\geq 0$ and $\sum_i \nu_n(f_i) = 1$.
Assume that $(\nu_n)$ converges towards $\nu_\infty=\sum_i \nu_\infty(f_i)\delta_{f_i}$, and that none of the
groups  $\Gamma_{\nu_n}$   fixes a point in $\partial \Hyp_X$. 
For $n\in \N$, we denote by $w_{\nu_n}\in \Hyp_X$ the eigenvector of the operator $P_{\nu_n}$ 
given by Lemma~\ref{lem:unique_w}, 
and by $S_{\nu_n}$ the current given by Proposition~\ref{pro:kawaguchi} for the class $w_{\nu_n}$; we shall 
write $S_{\nu_n}=\Theta(w_{\nu_n})+dd^c u_{\nu_n}$, as in Proposition~\ref{pro:kawaguchi}.
For the measure $\nu_\infty$, we make one of the following two assumptions: 
\begin{itemize}
\item[(a)] $\Gamma_{\nu_\infty}$ does not fix any point of $\partial \Hyp_X\subset \P(H^{1,1}(X;\R))$; or
\item[(b)] $\nu_\infty=\frac{1}{2} (\delta_f+\delta_{f^{-1}})$ for some loxodromic automorphism $f$, and $w_{\nu_n}$ 
converges to $w_{\nu_\infty}:=\frac{1}{\sqrt{2}} ([T^+_f]+[T^-_f])$, with notation as in Example~\ref{eg:loxodromic_second_part}. 
 \end{itemize}
In both cases, Proposition~\ref{pro:kawaguchi} provides a unique closed positive current $S_{\nu_\infty}$ 
with class $[S_{\nu_\infty}]=w_{\nu_\infty}$ such that $P_{\nu_\infty}=\alpha(\nu_\infty) S_{\nu_\infty}$; it coincides with $\frac{1}{\sqrt 2} (T^+_f+T^-_f)$ in case (b).

In case (a), by the uniqueness in~Lemma~\ref{lem:unique_w},
the classes $w_{\nu_n}$ converge towards $w_{\nu_\infty}$; in case (b) this convergence holds by assumption.
Note that  the corresponding constants  $\alpha(\nu_n)$ converge as well.  

\begin{lem}\label{lem:convergence}
Under the above assumptions,
\begin{enumerate}[\em (1)]
\item  the sequence of  closed positive currents $(S_{\nu_n})$ converges towards $S_{\nu_\infty}$;
\item the canonical (continuous) potentials $ u_{\nu_n}$ converge uniformly to that  of $S_{\nu_\infty}$;
\item the sequence of measures $\mu_n:=S_{\nu_n}\wedge S_{\nu_n}$ converges towards $S_{\nu_\infty}\wedge S_{\nu_\infty}$.
\end{enumerate} 
In case {\em{(b)}},  $S_{\nu_\infty}\wedge S_{\nu_\infty}$ is the unique measure of maximal entropy $\mu_f$ of  $f$.
\end{lem}

\begin{proof} 
The first assertion follows from the uniqueness of the current $S_{\nu}$, obtained in Proposition~\ref{pro:kawaguchi}, and 
the compactness of the space of currents of mass 1. 
The geometric convergence obtained from Equation~\eqref{eq:kawa_geom_conv} shows that the sequence of potentials $ u_{\nu_n}$ is 
equicontinuous, and by the uniqueness of the normalized potentials, it follows that $ (u_{\nu_n})$ converges 
uniformly to $ u_{\nu_\infty}$. Then the convergence of the sequence of measures
$(\mu_n)$ follows from the continuity properties of wedge products of currents (see \cite[III.3.6]{demailly_agbook}). 
Finally, the characterization of  $S_{\nu_\infty}\wedge S_{\nu_\infty}$
in case (b) was explained in Example~\ref{eg:loxodromic_second_part}.
\end{proof}

%%%%%%%%%%%%%%%%%%%%%%%%%%%%%%%%%%%%%%
\subsection{Rational invariant classes}\label{par:rational_invariant_classes}
%%%%%%%%%%%%%%%%%%%%%%%%%%%%%%%%%%%%%%
We now construct sequences of probability measures for which 
the fixed classes $w_{\nu_n}$ have good positivity and integrality properties; the last 
assertion makes use of the contraction $\pi_0\colon X\to X_0$ constructed in 
Proposition~\ref{pro:contraction}. 

\begin{pro}\label{pro:kawa_sequence_nun}
Let $\Gamma$ be a non-elementary subgroup of $\Aut(X)$ such that $\Pi_\Gamma$ 
is defined over $\Q$.
Let $f$ be a loxodromic element of $\Gamma$. There is a sequence $(\nu_n)$ of probability measures  on $\Aut(X)$ 
such that 
\begin{enumerate}[\em (1)]
\item The support ${\mathrm{Supp}}(\nu_n)$ is a finite subset $F$ of $\Gamma$ that does not depend on $n$ and 
generates a non-elementary subgroup of $\Gamma$ containing $f$;
\item $\nu_n(g)$ is a positive  rational number for all $g\in F$;
\item the unique  eigenvector $w_{\nu_n}$  of $P_{\nu_n}$  in $\Hyp_X$ is an element of $\R_+\NS(X;\Z) $;
\item the corresponding eigenvalue $\alpha(\nu_n)$ belongs to $\Q_+\cap ]1,+\infty[$;
\item $\nu_n$ converges to the measure $\frac{1}{2}(\delta_{f}+ \delta_{f^{-1}})$ and 
$  w_{\nu_n}$ converges to 
 $\frac1{\sqrt2}({[T^+_f]+ [T^-_f]})$.  \end{enumerate}
If $\Gamma$ contains a parabolic element $g$, one can furthermore assume 
 that $g$ belongs to $ F$ and that $w_{\nu_n}\in \R_+\pi_0^*[A_n]$ for some  ample line bundle
 %\serge{il se pourrait que pour la dernière assertion, on ait juste besoin de $\Gamma_\pi$ défini sur $\Q$... A garder en tête} 
$A_n$ on $X_0$. 
\end{pro}

\begin{proof} 
For  the proof we use the  conventions  of  \S \ref{subs:loxodromic}, in particular the classes 
$\theta_f^\pm$, which can be  defined by $\theta^\pm_f= \langle T^\pm_f\vert [\kappa_0] \rangle^{-1} [T^\pm_f]$.

\smallskip

{\bf{Step 1}}.-- Since the representation of $\Gamma$ on $\Pi_\Gamma$ is irreducible, it is also 
irreducible over~$\C$. Indeed, if $W$ is a proper, $\Gamma$-invariant, complex subspace of $\Pi_\Gamma\otimes_\R\C$,
then $W$ does not contain any non-zero real vector $u\in H^{1,1}(X;\R)$; in particular, it does not contain any isotropic 
eigenvector of any loxodromic element of $\Gamma$. This implies that $W$ is contained in the orthogonal complement
$(\theta_h^+)^\perp$ for all   $h\in \Gamma_{\mathrm{lox}}$. But in $\Pi_\Gamma$ the intersection 
$\bigcap_{h\in \Gamma_{\mathrm{lox}}} (\theta_h^+)^\perp$ is defined over $\R$ and is $\Gamma$-invariant, 
so it is trivial. 

Thus, according to Burnside's theorem (see~\cite{Halperin-Rosenthal}), $\Gamma$ contains a basis of the real vector space $\End(\Pi_\Gamma)$. More precisely, one can 
find a basis  $(f_1^*, f_2^*, \ldots, f_N^*)$  with $N=(\dim \Pi_\Gamma)^2$ such that $f_i\in \Gamma$ for all $i$, $f_1=f$ and $f_2=f^{-1}$ (indeed, $f$ and $f^{-1}$ are linearly independent endomorphisms).  In particular, the set of linear combinations $\sum_i \alpha_i f_i^*$ with $\alpha_i\geq 0$
contains a non-empty, open, and  convex cone of $\End(\Pi_\Gamma)$. 

If $\Gamma$ contains a parabolic element $g$ we can further require that $g$ belongs to the basis, 
 because $f^*$, $g^*$, and $(f^{-1})^*$ are linearly independent, as can be seen by diagonalizing~$f^*$.

\smallskip

{\bf{Step 2}}.-- Set $F= \{ f_1, f_2, \cdots, f_N\}$ and $\Delta_N=\{(\nu_i)\in \R_+^N\; ;\;  \sum_i \nu_i=1\}$. Let $\Delta_N^\circ$ be the interior
of this simplex. Points in $\Delta_N^\circ$ correspond to probability measures $\nu=\sum_i \nu_i \delta_{f_i}$ whose support is equal to $F$.
When $\nu\in \Delta_N^\circ$, the group $\Gamma_\nu$ is non-elementary by Step 1, so
by Lemma \ref{lem:unique_w} $P_\nu$ has a unique 
  fixed point $w_\nu$ in $\Hyp_X$. As a consequence, the map $\nu\in \Delta^\circ\mapsto w_\nu$ is continuous. 

Now, consider a sequence $(\nu_n)$ of elements of $\Delta_N^\circ$ converging to $a \delta_f+(1-a)\delta_{f^{-1}}$, with $0<a<1$.
Normalize the fixed point $w_{\nu_n}$ by setting $\overline w_n:=\langle w_{\nu_n}\vert [\kappa_0]\rangle^{-1} w_{\nu_n}$, so that $\langle \overline w_{n} \vert [\kappa_0]\rangle=1 $ and $\overline w_n$ 
stays in a compact subset of $H^{1,1}(X;\R)$. If $\overline w_{n_j}$ 
converges to $\overline w$ along a subsequence $(n_j)$ the limit is a nef eigenvector of the operator $af^*+(1-a)(f^{-1})^*$ associated to  an eigenvalue 
$\geq 1$. Thus, if $a$ is small the limit must be equal to $ \theta^-_f$ 
and the sequence 
$(\overline w_n)$ converges towards this eigenvector (see Example~\ref{eg:loxodromic_first_part}). 
Conversely, if $1-a$ is small, then the limit is $ \theta^+_f$.

The subset 
\begin{equation}
\Delta_N^\circ(\e)=\{(\nu_i)\in \Delta_N^\circ\; ;\; \nu_i\leq \e, \forall i\geq 3\}
\end{equation}
is connected. So, the closure of its image by the continuous 
 map $\nu\mapsto \langle w_{\nu}\vert [\kappa_0]\rangle^{-1}  w_\nu$ is a 
compact and connected subset of $\Pi_\Gamma$, and the intersection of these compact sets is also  connected. 
This set is contained in the segment $[\theta_f^-, \theta_f^+]$ because it is contained in 
${\overline{{\mathrm{Pos}}(X)}}$ and in the union of eigenvectors of $a f^* + (1-a)(f^{-1})^*$, for $a\in [0,1]$.  Since it contains the
endpoints of this segment,  it actually coincides with it. 

From this we deduce  that there exists  a sequence of probability measures $\nu_n\in \Delta_N^\circ$ such that 
$\langle w_{\nu_n}\vert [\kappa_0]\rangle^{-1} w_{\nu_n}$
 converges to the class $ \frac12(\theta^+_f+\theta^-_f) $, hence 
\begin{itemize}
\item $  w_{\nu_n}$
 converges to the class $ \frac1{\sqrt2}([T^+_f]+[T^-_f]) $. 
\end{itemize}

Then the arguments given in Example~\ref{eg:loxodromic_first_part}, Example~\ref{eg:loxodromic_second_part},  and Lemma~\ref{lem:convergence} show that:
\begin{itemize}
\item $P_{\nu_n}$ converges towards $\frac{1}{2}(f^*+(f^{-1})^*)$;
\item $\alpha(\nu_n)$ converges towards $\alpha(f)=\frac{1}{2}(\lambda(f)+1/\lambda(f))$.
\end{itemize}

\smallskip

{\bf{Step 3}}.-- At this stage the coefficients 
$\nu_n(f_i)$ and the eigenvalues $\alpha(\nu_n)$ 
are positive real numbers.  Let $U_n$ be a small open 
neighborhood of $\nu_n=(\nu_n(f_i))$ in  $\Delta_N^\circ$. As a consequence of the first step, the map $\nu'\in U_n\mapsto w_{\nu'}$ 
contains a small neighborhood of $w_{\nu_n}$ in $\Pi_\Gamma\subset \NS(X;\R)$. Thus, after a small perturbation of $\nu_n$ we may assume that
$w_{\nu_n}$ is an element of $\R_+\NS(X;\Z)$.
According to Proposition~\ref{pro:contraction} and 
Remark~\ref{rem:contraction}, when $\Gamma$ contains parabolic elements, 
we may further  choose $w_{\nu_n}$ to be 
proportional to the pullback $[\pi_0^*A_n]$ of an ample class.

The equation satisfied by $w_{\nu_n}$ is  $\alpha(\nu_n) w_{\nu_n} = \sum_i \nu_n(f_i) f_i^*(w_{\nu_n})$. Write $w_{\nu_n}=\eta_n {\tilde{w}}_n$ for some ${\tilde{w}}_n$
in $\NS(X;\Q)$ and $\eta_n$ in $\R_+$; the equation becomes 
\begin{equation}
{\alpha(\nu_n)} {\tilde{w}}_n = \sum_{i=1}^N  {\nu_n(f_i)}f_i^*({\tilde{w}}_n).
\end{equation}
This is  a linear relation  of the form $\beta_0 {\tilde{w}}_n = \sum_i  \beta_i f_i^*({\tilde{w}}_n)$, 
where ${\tilde{w}}_n$ and the $f_i^*({\tilde{w}}_n)$ belong to $\NS(X;\Q)$
 and the $\beta_i$ are positive real numbers (with $\beta_0>1$). Thus, given any $\e >0$, 
there is a relation of the form 
 ${\tilde{\beta_0}}{\tilde{w}}_n = \sum_i \tilde \beta_i f_i^*({\tilde{w}}_n)$   where the  coefficients 
$\tilde \beta_i$ are   rational numbers  which are $\e$-close
to the original  $\beta_i $. 
This proves that we can perturb $\nu_n$ one more time to add the assumption that
 the $\nu_n(f_i) $ and $\alpha(\nu_n)$ are positive rational numbers. 
\end{proof}

%%%%%%%%%%%%%%%%%%%%%%%%%%%%%%%%%
\subsection{Arithmetic equidistribution} \label{subs:arithmetic_equidistribution}
%%%%%%%%%%%%%%%%%%%%%%%%%%%%%%%%%
Assume that the normal projective surface $X$ and the subgroup $\Gamma$ of $\Aut(X)$ are defined over some number field  $\bfk\subset \Qbar$.
($X$ is not assumed to be smooth here.)
For $y$ in $X(\Qbar)$, let $m_y$ denote the uniform probability measure supported on the Galois orbit of $y$,
\begin{equation}
m_y=\frac{1}{\deg(y)} \sum_{y'\in \Gal(\Qbar : \bfk)(y)}\delta_{y'};
\end{equation}
here, $\deg(y)$ is the degree of the closed point defined by $y$, 
or equivalently the cardinality of the orbit of $y$ under 
the action of the Galois group $\Gal(\Qbar : \bfk)$, and the sum ranges over all points $y'$ in this orbit.
A sequence $(x_j)$ of points of $X(\Qbar)$ is {\bf{generic}} if the only Zariski closed subset of $X$ 
containing
infinitely many of the $x_j$'s is $X$. Equivalently, $(x_j)$ converges to the generic point of $X$ for the Zariski 
topology.    

In the following theorem, $\Gamma_\nu$ is the group generated by the support of $\nu$, and $S_\nu$ is the current
given by Proposition~\ref{pro:kawaguchi} and Remark~\ref{rem:kawaguchi} (associated to
 the normalized class $\langle w \vert w \rangle^{-1/2}w$).

\begin{thm}\label{thm:berman-boucksom_periodic}
Let $X$ be a normal projective surface defined over a number field $\bfk$.
Let $\nu$ be a probability measure on $\Aut(X_\bfk)$ with finite support $F$ and rational weights $\nu(f)\in \Q_+$, for $f$ in $F$. Assume that
\begin{itemize}
\item[(i)] $P_\nu^*w=\alpha w$
for some ample class $w$ in $\NS(X;\Q)$ and some eigenvalue $\alpha>1$; 
\item[(ii)] $(x_j)\in X(\Qbar)^\N$ is a generic sequence such that each $x_j$ is a periodic point of $\Gamma_\nu$.
\end{itemize}
Then, the sequence of probability measures $(m_{x_j})$ 
converges   towards the measure $S_\nu\wedge S_\nu$, and this measure is $\Gamma_\nu$-invariant.
\end{thm}

It is important here that $w$ is a rational class, that is $w\in \NS(X;\Q)$ instead of
just $\NS(X;\R)$, 
since we rely on results of Kawaguchi, Yuan and Zhang that require this assumption. It is 
also crucial that $X$ is not supposed to be smooth because this result will be applied 
to the model $X_0$ constructed in \S \ref{par:periodic_curves_ampleness}.
When $\Gamma_\nu$ is non-elementary, the eigenvector $w$ must be proportional to $w_\nu$  and $\alpha=\alpha(\nu)$ (as in Lemma~\ref{lem:unique_w}).

As explained in Section~\ref{par:strategy_5}, 
a consequence of the theorem is that the limit $S_\nu\wedge S_\nu$ 
of the sequence $(m_{x_j})$, depends only on $\Gamma_\nu$ (but not on the weights $\nu(f)$); 
this will ultimately imply that  Assumption~(ii)  happens only in very rare situations.   

\begin{eg}\label{eg:suz}
Under the assumption of Theorem~\ref{thm:berman-boucksom_periodic}, assume furthermore that $X$ is an abelian surface. 
Since $\Gamma_\nu$ has a periodic point $x_1$, the stabilizer $\Gamma_{x_1}={\mathrm{Stab}}_{\Gamma_\nu}(x_1)$ has finite index in $\Gamma_\nu$; conjugating by a translation we can take  $x_1$ as
the neutral  element for the group law on $X\simeq \C^2/\Lambda$. Then, the periodic points of $\Gamma_{x_1}$ (and of $\Gamma_\nu$) are exactly the 
torsion points of $X$ (see~\S \ref{par:abelian}). By the equidistribution theorem 
of Szpiro, Ullmo, and Zhang, the measures $m_{x_j}$ converge towards the Haar measure of $X$ (see~\cite{Szpiro-Ullmo-Zhang}). 
Also, $\Gamma_{x_1}$ is induced by a 
subgroup $\tilde\Gamma_{x_1}$ of $\GL_2(\C)$ preserving the lattice $\Lambda$, and $\Gamma_\nu$ is a group of affine automorphisms
with linear part given by $\tilde\Gamma_{x_1}$ and translation part given by the finite subset $\Gamma_\nu(0)\subset X$. Every cohomology
class $u$ in $H^{1,1}(X;\R)$ has a distinguished representative, given by the unique translation invariant 
$(1,1)$-form $\omega_u$ on $X$ such that $[\omega_u]=u$. Since $\Gamma_\nu$ acts by affine automorphisms, the operator
$\alpha^{-1}P_\nu$ preserves $\omega_{w_\nu}$, and the current $S_\nu$ is given by $\langle w_\nu \vert w_\nu\rangle^{-1/2}\omega_{w_\nu}$.  Thus, for abelian 
surfaces, Theorem~\ref{thm:berman-boucksom_periodic} corresponds to the theorem of Szpiro, Ullmo, and Zhang together with the fact that
$\omega_{w_\nu}\wedge \omega_{w_\nu}$ is proportional to the volume form inducing 
 the Haar measure on $X$. \end{eg}

\begin{proof}[Preliminary remarks for the proof of Theorem~\ref{thm:berman-boucksom_periodic}]
Set  $\Gamma=\Gamma_\nu$.
Let $\pi\colon Y\to X$ be a minimal resolution of $X$;  it is unique, up to isomorphism (see~\cite[\S III.6]{BHPVDV}, Theorems~(6.1) and (6.2)), and $\Gamma$
lifts to a group of automorphisms $\Gamma_Y$ of $Y$; we shall also consider $\nu$ as a measure on $\Gamma_Y$. The pull-back $\pi^*\colon \Pic(X)\to \Pic(Y)$ is an embedding,
%\serge{Si $D$ est un diviseur dans $X$ et $\pi^*D=0$ dans $Pic(Y)$, il existe une fonction rationnelle 
%$\Phi$ telle que $\pi^*(D)=div(\Phi)$, mais comme $\pi$ est un morphisme birationnel $\Phi$ est l'image 
%réciproque $\Phi_X\circ \pi$ d'une fonction sur $X$, et $D=div(\Phi_X)$. Cela montre l'injection.}
and   an isometry for the intersection form. Since $\pi^*\NS(X;\R)$ is $\Gamma_Y$-invariant and contains classes with positive self-intersection, we deduce that $\Gamma$ is elementary if and only if $\Gamma_Y$ is, and $\Pi_{\Gamma_Y}=\pi^*\Pi_{\Gamma}\subset \pi^*\NS(X;\R)$ if $\Gamma$  is non-elementary. Also, $\pi^*w$ satisfies $P_\nu(\pi^*w)=\alpha \pi^*w$ in $\NS(Y;\R)$.

If $\Gamma$ is elementary, we apply \cite[Theorem 3.2]{Cantat:Milnor} on $Y$. Two cases may occur. In the first case, 
$\Gamma_Y$ fixes a class $u\neq 0$ in the closure of the positive cone $\overline{\mathrm{Pos}}(Y)$; but then $\langle \pi^*w\vert u\rangle>0$ because $\langle \pi^*w\vert \pi^*w\rangle >0$, and $\langle P_\nu(\pi^*w)\vert u\rangle= \langle \pi^*w\vert u\rangle$ 
because $u$ is invariant: this contradicts $\alpha >1$. So, we are in fact in the second case: $\Gamma_Y$ (resp. $\Gamma$) contains a loxodromic element $f$ and preserves the pair of lines $\R[T^+_f]\cup \R[T^{-}_f]$. So, even when
 $\Gamma$ is elementary, we know that it contains a loxodromic element.

Now, assume that $\Pic^0(X)$ is non-trivial; equivalently, $\Pic^0(Y)$ is non-trivial. 
Since $\Gamma_Y$ contains a loxodromic element, we deduce from \cite[Theorem 10.1]{Cantat:Milnor}  
that $Y$ is a blow-up of an abelian surface
(for $\Pic^0(Y)$ is trivial  when $Y$ is birationally equivalent to a rational, K3, or Enriques surfaces).
But then, $X$ is smooth and is also a blow-up of an abelian surface.  If $X$ itself is not  an abelian  surface,
the exceptional divisor $E$ of the blow up is $\Gamma$-invariant; as above since $w$ is ample 
and $\alpha>1$ we obtain a contradiction. 
So $X$ is abelian, and Theorem~\ref{thm:berman-boucksom_periodic}  follows from the discussion in Example~\ref{eg:suz}. 
Thus in what follows, we may   assume $\Pic^0(X)=0$ to simplify the exposition. 
\end{proof}

\begin{proof}[Proof of Theorem~\ref{thm:berman-boucksom_periodic}] We now assume that $\Pic^0(X)=0$ and that $\Gamma_\nu$
contains a loxodromic element. 

The proof is based on standard ideas from arithmetic equidistribution theory. For the reader's 
convenience we provide background and details (see also \cite{Lee:2012} for the applicability of arithmetic equidistribution in this
context). 
Changing $w$ into a multiple, we may assume $w\in \NS(X;\Z)$. Multiplying the equation $\sum_f \nu(f) f^*(w)=\alpha w$
by the least common multiple $b$ of the denominators, we obtain  the  linear relation  
\begin{equation}\label{eq:Integral_invariance_of_w}
\sum_{f\in F} n(f) f^*(w)=dw;
\end{equation}
in which $d=b\alpha$ and the coefficients $n(f)=b\nu(f)$ are positive integers  
such that 
\begin{equation}
 \sum_{f\in F} n(f) = b < d = b \alpha
\end{equation}
because $\alpha>1$. 
Denote by $D$ a divisor with class $w$, and by $L$ the line bundle ${\mathcal{O}}_X(D)$. Since
$\Pic^0(X)$ is trivial, the Equality~\eqref{eq:Integral_invariance_of_w} implies  
\begin{equation}
\bigotimes_{f\in F} (f^*L)^{\otimes n(f)} = L^{\otimes d}
\end{equation}
up to an isomorphism of line bundles that we do not specify. 
From this identity, Kawa\-guchi constructs in \cite[\S 1]{Kawaguchi:Crelle} 
a function $\hat{h}_L\colon X(\Qbar)\to \R_+$ which satisfies the relation 
$\sum_f n(f) {\hat{h}}_L\circ f = d{\hat{h}}_L$  
and differs from the naive Weil height function associated to $L$ only by 
a bounded error. It will be referred to as  as the {\bf{canonical stationary height}} (associated to $\nu$ and $L$). 
This height function may be decomposed as a sum of continuous local height functions, see~\cite[\S 4]{Kawaguchi:Crelle}.
 Arakelov theory also provides a canonical adelic metric on $(X,L)$; in particular, for each place $v$
of $\bfk$, there is a metric $\vert \cdot \vert_v$ on the line bundle $(X_{\bfk_v}, L_{\bfk_v})$, where $\bfk_v$ is 
an algebraic closure of the $v$-adic completion of $\bfk$, such that 
\begin{equation}\label{eq:height_place}
\prod_{f\in F} \vert s(f(x)) \vert_v^{n(f)}= \vert s(x)\vert_v^d 
\end{equation}
for every local section $s$ of $L$ defined over $\bfk$. 
In our setting, an embedding $\bfk\subset \C$ is 
fixed; it corresponds to one of the places of $\bfk$. The adelic metric corresponding 
to that place gives a continuous metric on $L$, and  from  the relation \eqref{eq:height_place} and 
 the uniqueness of the current $S_\nu$ 
we see that the  curvature current of the metric is precisely the current $S_\nu$ from 
 Proposition~\ref{pro:kawaguchi} (see also Remark~\ref{rem:kawaguchi}).% (see~\cite{Kawaguchi:Crelle, Lee:2012}). 

\begin{lem}\label{lem:kawaguchi_northcott_periodic}
A point $x\in X({\overline{\Q}})$ satisfies ${\hat{h}}_L(x)=0$ if and only if its  $\Gamma_\nu$-orbit is finite.
\end{lem}
\begin{proof}[Proof (see~\cite{Kawaguchi:Crelle}, Prop. 1.3.1)]
Let $\bfk'$ be any number field containing $\bfk$. 
The set $\{ x\in X(\bfk')\;  ; \;  {\hat{h}}_L(x)=0\}$ is $\Gamma_\nu$-invariant and by Northcott's theorem it is finite, 
so every element of that set has a finite orbit. Let us prove the other implication. Iterating  the relation $\sum_f n(f) {\hat{h}}_L\circ f = d{\hat{h}}_L$ and evaluating  it on a periodic point $x$ yields
$\alpha^n {\hat{h}}_L(x)=\sum_{g\in\Gamma}\nu^{\star n}(g){\hat{h}}_L(g(x))$ where $\nu^{\star n}$ is the $n$-th convolution of $\nu$. The right hand side is bounded because 
${\hat{h}}_L(g(x))$ takes only finitely many values, 
and on the left hand side the term $\alpha^n$ goes to $+\infty$; thus ${\hat{h}}_L(x)=0$, as asserted.
\end{proof} 

Let  $\bfA_{\bfk}$ denote the ring of ad\`eles of the number field $\bfk$.  The sections of $L$ defined over $\bfk$
determine a lattice $H^0(X,L)$ in $H^0(X,L)\otimes \bfA_\bfk$, and the quotient $(H^0(X,L)\otimes \bfA_\bfk)/H^0(X,L)$ is therefore compact. Denote by $\overline{L}$ the line bundle $L$ endowed with its canonical adelic metric. For each place $v$, denote by $B_v\subset H^0(X,L)\otimes \bfk_v$  the unit
ball with respect to the $v$-adic component $\vert \cdot \vert_v$ of the adelic metric of ${\overline{L}}$. 
Let $\lambda_L$  be a Haar measure on $H^0(X,L)\otimes \bfA_\bfk$. The quantity
\begin{equation}
\chi(X,{\overline{L}})=\log \frac{\lambda_L(\prod_{v\in M_\bfk} B_v)}{\lambda_L(H^0(X,L)\otimes \bfA_\bfk/H^0(X,L))}
\end{equation}
does not depend on the choice of Haar measure. Taking tensor products, we get a sequence of adelic metrized line bundles
$({\overline{L}}^{\otimes n})_{n\geq 1}$, and by definition the arithmetic volume of 
$\overline L$ is  
\begin{equation}
{\widehat{\vol}}_\chi(X, {\overline{L}})=\limsup_{n\to +\infty} \; \frac{ \chi(X, {\overline{L}}^{\otimes n})}{n^3/6}.
\end{equation}
This is to be compared with the usual  volume $\vol(X, L)$ of $L$,  which by definition is the limsup of 
$  \frac{2}{n^2}h^0(X,L^{\otimes n})$, as $n$ tends to $+\infty$.
  A fundamental inequality of Zhang asserts that if $(x_j)$ is a generic sequence in $X(\Qbar)$, 
\begin{equation}
\liminf_j \; {\hat{h}}_L(x_j)\geq \frac{{\widehat{\vol}}_\chi(X, {\overline{L}})}{3 \vol(X, L)}.
\end{equation}
This follows from an adelic version of the Minkowski theorem on the existence of integer points in lattices
 (see e.g.~\cite{SWZhang:1995b} or Lemma 5.1 in~\cite{Chambert-Loir-Thuillier}).

 As for the usual volume, the arithmetic volume can be interpreted in terms of arithmetic intersection.
Indeed, to $\overline{L}$ is associated an arithmetic degree $\widehat{\deg}(c_1(\overline{L})^3)$, and 
it is shown in~\cite{SWZhang:1995b} that 
${\widehat{\vol}}_\chi(X, {\overline{L}})=\widehat{\deg}(c_1(\overline{L})^3)\geq 0$ (see also~\cite[Thm~2.3.1]{Kawaguchi:Crelle}).
%Since $L$ is ample, 
%\begin{equation}
%\widehat{\deg}(c_1(\overline{L})^2)\geq 0,
%\end{equation}
%as shown in Theorem~2.3.1 of \cite{Kawaguchi:Crelle}. Furthermore since $L$ is ample it is shown
% in~\cite{SWZhang:1995b} that ${\hat{\vol}}_\chi(X, {\overline{L}})=\hat{\deg}(c_1(\overline{L})^2)$.
Thus, the existence of a generic sequence of periodic points $(x_j)$
shows that ${\widehat{\vol}}_\chi(X,{\overline{L}})=0$ and ${\hat{h}}_L(x_j)={\widehat{\vol}}_\chi(X,\overline{L})$ 
for all~$j$. 

We are now in   position to apply  Yuan's equidistribution theorem  
(see~\cite{Yuan:Inventiones, Berman-Boucksom:Inventiones}): the  sequence of   measures $(m_{x_j})$ 
converges towards the probability measure $S_\nu\wedge S_\nu$ as $j$ goes to $\infty$.  
If $f$ is any element of $\Gamma_\nu$,  the points $f(x_j)$    also form a generic sequence of $\Gamma$-periodic points.
Since the actions of $\Gamma$ and $\Gal(\overline{\Q}:\bfk)$ commute, we infer that $f_*(m_{x_j})=m_{f(x_j)}$, so taking the limit 
as $j\to\infty$ yields  $f_*(S_\nu\wedge S_\nu)=S_\nu\wedge S_\nu$, and finally $S_\nu\wedge S_\nu$ is $\Gamma_\nu$-invariant.
\end{proof}

%%%%%%%%%%%%%%%%%%%%%%%%%%%%%%%%%%%%%%%%
\subsection{Density of active saddle periodic points}\label{subs:active}
%%%%%%%%%%%%%%%%%%%%%%%%%%%%%%%%%%%%%%%%

Let $f$ be a loxodromic automorphism of $X$. We say that a periodic point of $f$ is {\bf{active}} if it is contained in the support of the 
measure of maximal entropy $\mu_f$. From~\cite{Dujardin:Duke, Cantat:Milnor} 
we know that a saddle periodic point that is
not contained in any $f$-periodic curve is active (see~\cite{Cantat:Milnor}, Theorem~8.2).  

\begin{thm}\label{thm:dense_active_saddles}
Let $X$ be a compact K\"ahler surface, and $\Gamma$ be a non-elementary subgroup of $\Aut(X)$ that contains a parabolic
automorphism. Then, given any non-empty open subset $\cV\subset X$ (for the Euclidean topology), there exists a point $x\in \cV$
and a loxodromic element $f\in \Gamma$ such that $x$ is an active saddle periodic point of $f$. 
In particular, the union of the supports of the measures $\mu_f$, for $f\in \Gamma_{\mathrm{lox}}$, is dense in $X$.
\end{thm}

Before proceeding to the proof, let us point out the following fact, which readily follows from 
Lemma~\ref{lem:control_periodic_curves_strong}, together with the fact that an irreducible 
curve with  negative self-intersection is  determined by its class in $\NS(X; \Z)\subset \NS(X; \R)$.

\begin{lem}\label{lem:control_periodic_curves}
Let $U$ and $U'$ be two disjoint open subsets of $\P(\NS(X;\R))$ containing nef classes and introduce the set 
\[
A(U,U')=\{ f\in \Aut(X)\; ; \; f \; {\text{is loxodromic}},  \; \P([T^+_f])\in U \; {\text{and}} \; \P([T^-_f])\in U'\} .
\]
Then, the union of all periodic curves of all elements of $A(U,U')$  is a finite set of curves.
\end{lem}
 
\begin{proof}[Proof of Theorem~\ref{thm:dense_active_saddles}] 
Pick  $g$ in $\Gamma_{\mathrm{par}}$. Since $\Gamma$ is non-elementary we can conjugate  $g$ by an element of 
$\Gamma_{\mathrm{lox}}$ to produce a pair $g, h \in \Gamma_{\mathrm{par}}$ with distinct fixed points $\partial \Hyp_X$. 

\smallskip

{\noindent{\bf{Step 1}}}.--  Assume that $X$ is a blow-up of
 an abelian surface $A$, and pick   $f$ in $\Gamma_{\mathrm{lox}}$. 
By Lemma~\ref{lem:lox_on_tori}, its periodic points are dense, and all of them are active because
$\mu_f$ is the pull back to $X$ of the  Haar measure on $A$. Thus   any open subset of $X$ contains
 active saddle periodic  points.
 
From now on, assume that $X$ is not a blow-up of an abelian surface. 

\smallskip

{\noindent{\bf{Step 2}}}.-- 
From Section~\ref{par:basics_on_par_and_lox}, $g$ preserves a unique fibration $\pi_g\colon X\to B_g$ and
the automorphism  induced by $g$  on $B_g$ is periodic. Replacing $g$ by some iterate,
we assume that $\pi_g\circ g=\pi_g$. Let $\cU\subset B_g$ be a small disk containing no
critical value of $\pi_g$. There is a real analytic diffeomorphism $\Phi\colon \pi^{-1}(\cU)\to \cU\times \R^2/\Z^2$ and a real 
analytic map $\varphi\colon \cU\to \R^2$ such that 
$\pi_\cU\circ \Phi=\pi_g$ and $g_\Phi:=\Phi\circ g \circ \Phi^{-1}$ satisfies 
\begin{equation}
g_\Phi(b,z)=(b,z+\varphi(b))
\end{equation}
for all points $(b,z)\in \cU\times \R^2/\Z^2$. According to \cite{cantat_groupes, invariant},  $\varphi$ is generically of maximal rank: there is a 
finite set $Z\subset\cU$ such that $(D\varphi)_b\colon T_b\cU\to \R^2$ has rank $2$ for every $b\in \cU\setminus Z$; hence, $\{b\in \cU\; ; \; \varphi(b)\in \Q^2/\Z^2\}$ is dense in $\cU$. If $\varphi(b)=(a_0/N,b_0/N)$ for some integers $a_0$, $b_0$
and $N$, then every point $q=(b,z)$ in the fiber is fixed by $g_\Phi^N$ and 
\begin{equation}
(Dg_\Phi^N)_x = \left( 
\begin{array}{cc} \id_2 & 0 \\
N(D\varphi)_b & \id_2\end{array} 
\right).
\end{equation}
Thus, in any holomorphic coordinate system $(x,y)$ in which $\pi_g$ expresses as 
$\pi_g(x,y)=x$, the differential of $g^N$ at the fixed point $\Phi^{-1}(q)$ is of the form 
$\left(\begin{smallmatrix} 1 &0 \\a&1 
\end{smallmatrix}\right)$ with $a\neq 0$. 
 
\smallskip

{\noindent{\bf{Step 3}}}.-- The invariant fibrations $\pi_g$ and $\pi_h$ are transversal in the complement of a proper Zariski closed set ${\mathrm{Tang}}(\pi_g,\pi_h)$. According to Lemma~\ref{lem:control_periodic_curves} and Lemma~\ref{lem:ping-pong_parabolic}, we can find an integer $N>0$, 
and a divisor $F\subset X$ such that all elements $g^{\ell N}\circ h^{\ell N}$ with $\ell\geq 1$ are loxodromic and do not have any periodic curve
outside $F$.

\smallskip

{\noindent{\bf{Step 4}}}.--  Let $D$ be the union of the singular and multiple fibers of $\pi_g$ and of $\pi_h$,
of ${\mathrm{Tang}}(\pi_g,\pi_h)$, and of the divisor $F$; $D$ is a divisor of $X$. 
Let $\cV$ be an open subset of $X$. Then $\cV$ contains 
a small ball $\cV'$ such that 
\begin{itemize}
\item $\cV'$ does not intersect $D$;
\item  $\pi_g(\cV')$ and $\pi_h(\cV')$ are topological disks $\cU_g$ and $\cU_h$ in $B_g$ and $B_h$ respectively; 
\item there are local coordinates $(x,y)$ in $\cV'$ (resp. $x$ in $\cU_g$ and $y$ in $\cU_h$) such that $(\pi_{g})_{\vert \cV'}(x,y)=x$ and $(\pi_{h})_{\vert \cV'}(x,y)=y$.
\end{itemize}
Step 2 provides a point $(x_0,y_0)\in \cV'$ and an integer $N>0$ such that $g^N$ fixes the fiber of $\pi_g$ through $(x_0,y_0)$ pointwise, 
$h^N$ fixes the fiber of $\pi_h$ through $(x_0,y_0)$ pointwise, and 
\begin{equation}
(Dg^N)_{(x_0,y_0)} = 
\left(
\begin{array}{cc} 1 & 0 \\
a & 1 \end{array}
\right) \text{, and }
  (Dh^N)_{(x_0,y_0)}=
\left(
\begin{array}{cc} 1 & b \\
0 & 1 \end{array}
\right)
\end{equation}
for some non-zero complex numbers $a$ and $b$. If $\ell\in \Z$ is sufficiently large,  
$f_{\ell N}=(Dg^{\ell N})_{(x,y)}\circ(Dh^{\ell N})_{(x,y)}$ 
is a loxodromic automorphism, $(x_0,y_0)$ is a   fixed point of $f_{\ell N}$ which is not contained in a 
periodic curve of $f_{\ell N}$ (because $(x_0,y_0)$ is not in $F$), which is shown to be a saddle 
by an explicit computation. Thus, as explained before the proof, $(x_0,y_0)$ is   
active, and we are done.
\end{proof}

\subsection{Measure rigidity and Kummer examples}\label{par:finite_orbits_zariski}

\begin{thm}\label{thm:finite_orbits_zariski}
Let $X$ be a complex projective surface and let $\Gamma$ be a subgroup of~$\Aut(X)$. 
Assume that 
\begin{itemize}
\item[(i)] $X$ and $\Gamma$ are defined over a number field $\bfk\subset \overline{\Q}$; 
\item[(ii)] $\Gamma$ is non-elementary and contains a parabolic automorphism.
\end{itemize}
If $\Gamma$ has a Zariski dense set of finite orbits, then every loxodromic automorphism in 
$\Gamma$ is a Kummer example. %\romain{on pourrait ajouter qu'ils ont la même mesure max}
\end{thm}

\begin{proof} 

{\noindent}{\bf{Step 1.}}-- \emph{From the Zariski dense set of finite orbits 
 we can extract a 
generic sequence of $\Gamma$-periodic points $(x_j)\in X(\overline{\Q})^\N$.}

Since $\Gamma$ is non-elementary, it contains a loxodromic element $f$. The isolated periodic points of $f$ are defined over
$\Qbar$, because $X$ and $f$ are defined over $\Qbar$, and the non-isolated periodic points of $f$ form 
a finite number of $f$-periodic curves (see \S \ref{par:periodic_curves_of_loxodromic}). Thus, we can find a Zariski dense set of $\Gamma$-periodic
points $x'_i$ in $X(\Qbar)$.
If $Z\subset X$ is an irreducible curve that contains infinitely many of the $x_i'$, then $Z$ is defined over $\Qbar$ too. 
There are only countably many curves defined over $\Qbar$. Thus, by a diagonal argument, we find an infinite
sequence of periodic points $x_j\in X(\Qbar)$ such that $(x_j)$ is generic. 

In what follows, $(x_j)$ denotes such a generic sequence of periodic points. 
Consider the contraction $\pi_0\colon X\to X_0$ of the union 
$D_\Gamma$ of all $\Gamma$-periodic curves (see Proposition~\ref{pro:contraction}); the group $\Gamma$ also acts on the normal projective surface $X_0$. Note that the  projection 
$(\pi_0(x_j))\in X_0(\overline{\Q})^\N$ is also generic.

\smallskip

{\noindent}{\bf{Step 2.}}--   \emph{There exists a $\Gamma$-invariant measure 
$m$ such that $\mu_f = m$ for all loxodromic $f$}. 

Fix an arbitrary element $f$ in $\Gamma_{\mathrm{lox}}$. By \cite[Lemma 2.9]{stiffness}, $\Pi_\Gamma$ 
is defined over $\Q$ so applying Proposition~\ref{pro:kawa_sequence_nun}   we obtain
a sequence of probability measures $(\nu_n)$. 
Denote by $S_{\nu_n}$ and $S_{\nu_n,0}$ the currents, on $X$ and 
$X_0$ respectively, given by Proposition~\ref{pro:kawaguchi} and Remark~\ref{rem:kawaguchi}; by 
construction $\pi_0^*S_{\nu_n,0}=S_{\nu_n}$, where the pull-back is obtained by locally pulling back the continuous potentials. 

For the moment we fix the integer $n$. 
 Theorem~\ref{thm:berman-boucksom_periodic} shows that the sequence of probability measures $m_{\pi_0(x_j)}$ 
converges towards $S_{\nu_n,0}\wedge S_{\nu_n,0}$ 
as $j$ goes to $+\infty$. In particular,  $S_{\nu_n,0}\wedge S_{\nu_n,0}$
is a fixed $\Gamma$-invariant probability measure $\mu_0:=\lim_{j} m_{\pi_0(x_j)}$ that does not depend on 
$n$. Since $S_{\nu_n, 0}$ has continuous potentials, this measure gives no mass to  proper analytic subsets of~$X_0$. 
Let $\mu$ be the probability measure which is equal to $\pi_0^*(\mu_0)$ on $X\setminus D_\Gamma$ and 
gives no mass to  $D_\Gamma$. Since $S_{\nu_n}$ has continuous potentials, $\mu=S_{\nu_n}\wedge S_{\nu_n}$.
In $X$, the sequence $\lrpar{m_{x_j}}$ converges  to  $\mu$. Indeed, if 
a subsequence of $\lrpar{m_{x_j}}$ converges towards some probability measure $\lambda$, 
then  $(\pi_0)_*\lambda=\mu_0$, and since $\mu_0$ does not charge
any point of $X_0$, we infer that $\lambda$ is equal to 
$\pi_0^*(\mu_0)$ on $X\setminus D_\Gamma$ and does not charge $D_\Gamma$, 
which means that $\lambda = \mu$. Thus, by compactness of the set of probability measures, $m_{x_j}$ converges towards $\mu$.

Now, we  let $n\to\infty$. By Proposition~\ref{pro:kawa_sequence_nun} (5),  Proposition~\ref{pro:kawaguchi}, and Lemma \ref{lem:convergence}, 
$S_{\nu_n}\wedge S_{\nu_n}= \mu$ converges towards $\mu_f$ as $n$ goes to $+\infty$. Thus $\mu=\mu_f$ for all loxodromic elements 
$f$ in $\Gamma$. In particular, $\mu$ is $f$-ergodic hence $\Gamma$-ergodic.

\smallskip

{\noindent}{\bf{Step 3.}}-- {\emph{Conclusion}}. 

As already explained, $\mu$ gives no mass to proper algebraic subsets of $X$.   
Furthermore, Theorem~\ref{thm:dense_active_saddles} implies that 
the support of  $\mu$ is equal to $X$. Thus, Theorem 0.2 of \cite{cantat_groupes} 
(see also \cite{invariant})
shows that $\mu$ is absolutely continuous with  a smooth density.  
Since $\mu=\mu_f$, the Main Theorem of~\cite{cantat-dupont} implies that $(X,f)$ is a
 Kummer example, as was to be  shown. \end{proof}

%%%%%%%%%%%%%%%%%%%%%%%%%%%%%%%%%%%%%%%%%%%%%%%%
%%%%%%%%%%%%%%%%%%%%%%%%%%%%%%%%%%%%%%%%%%%%%%%%
\subsection{From Kummer examples to Kummer groups}\label{par:Kummer}
%%%%%%%%%%%%%%%%%%%%%%%%%%%%%%%%%%%%%%%%%%%%%%%%
%%%%%%%%%%%%%%%%%%%%%%%%%%%%%%%%%%%%%%%%%%%%%%%%

In this paragraph we prove the following theorem, which together with 
Theorem \ref{thm:finite_orbits_zariski} implies  Theorem \ref{mthm:main}. 
Its formulation, involving subgroups generated by unipotent parabolic elements,   
is intended  for further use in \S \ref{sec:canonical_vector_height}.  
 
\begin{thm}\label{thm:Kummer_for_groups}  
Let $X$ be a compact K\"ahler surface, and $\Gamma$ be a non-elementary subgroup of $\Aut(X)$ 
containing parabolic elements.  Assume that, given any pair of unipotent parabolic elements $(g,h)\in \Gamma_{\mathrm{par}}^2$,
every loxodromic element  $f\in \langle g, h\rangle$ is  a Kummer 
example. Then $(X,\Gamma)$ is a Kummer group.
\end{thm}

\begin{proof}
Consider the birational morphism $\pi_0\colon X\to X_0$ given in Proposition~\ref{pro:contraction}. 
By Proposition~\ref{pro:invariant_curve_loxodromic_precised}, there exists a loxodromic 
transformation $f$, of the form $h\circ g$ for some unipotent elements $g$, $h$ in $\Gamma_{\mathrm{par}}$, such that its maximal invariant curve $D_f$ coincides 
with~$D_\Gamma$. 
By assumption, $f$ is a Kummer example, which entails  that $X_0$ is a quotient $A/G$, with $A=\C^2/\Lambda$ an abelian surface and $G$ a finite
 subgroup of $\Aut(A)$ generated by a diagonal map $g_0\in \GL_2(\C)$  of  order $2$, $3$, $4$, $5$, $6$ or $10$ (see \S \ref{par:Kummer_classification}).
 
The group $\Gamma$ induces a group of automorphisms of $X_0$.
View $X_0$ as an orbifold: its fundamental group is 
$\Lambda\rtimes G$ and its universal cover $\widetilde{X_0}$ is $\C^2$. Concretely, this means that 
$X_0$ is the quotient of $\C^2$ by the group of affine transformations with linear part in $G$ and translation part in $\Lambda$. 
The canonical hermitian metric on $\C^2$ is invariant 
under the affine action of $\Lambda\rtimes G$. If $\tilde{h}\colon \C^2\to \C^2$ is a lift 
of some $h\in \Gamma$ to $\widetilde{X_0}$ 
(\footnote{To prove the existence of such a lift, note that $h$ maps the regular part of $X_0$ to itself, so 
first lift $h\rest{\mathrm{Reg}(X_0)}$ to $\C^2\setminus \pi\inv (\mathrm{Sing}(X_0))$, 
which is simply connected, and then use Hartogs extension to extend  $\tilde h$ accross the discrete set $\pi\inv (\mathrm{Sing}(X_0))$.}),
the norm of 
$D{\tilde{h}}_{(x,y)}$ with respect to this
hermitian metric is constant along the orbits of $\Lambda\rtimes G$, hence it is bounded since 
the action is co-compact.   This implies that the holomorphic map 
\begin{equation}
(x,y)\in \C^2\mapsto D{\tilde{h}}_{(x,y)}
\end{equation}
is constant. So, if we denote by $\widetilde{\Gamma}$ the group  of all possible lifts
of all elements of $\Gamma$ to $\C^2=\widetilde{X_0}$, then  $\widetilde{\Gamma}$ is
 a group of affine transformations that
contains $\Lambda\rtimes G$ as a normal subgroup and satisfies 
$\tilde{\Gamma}/(\Lambda \rtimes G) = \Gamma$. The action by
conjugation of $\tilde{\Gamma}$ on $\Lambda\rtimes G$ preserves the subgroup $\Lambda$ of translations. 
Therefore, $\Lambda$ is also normal in $\widetilde{\Gamma}$: this shows
that $\widetilde{\Gamma}$ induces a group of automorphisms $\Gamma_A=\widetilde{\Gamma}/\Lambda$ of the abelian surface $A=\C^2/\Lambda$ that covers $X_0$. 
The proof is complete. 
\end{proof}

\section{Around Theorem~\ref{mthm:main}: consequences  and comments}\label{sec:corollaries}

\subsection{Corollaries} The following corollary of Theorem~\ref{mthm:main}  
  applies for instance  to  general Wehler examples defined over $\overline \Q$.  

\begin{cor}\label{cor:finite_orbits_finite} 
Let $X$ be a smooth projective surface and let $\Gamma$ be a subgroup of $\Aut(X)$. 
Assume that:
\begin{enumerate}[\em (i)]
\item $X$ and $\Gamma$ are defined over a number field; 
\item $X$ is not an abelian surface; 
\item $\Gamma$ contains a parabolic automorphism, and has no invariant curve.
\end{enumerate}
Then  $\Gamma$ admits only finitely many finite orbits.
\end{cor}

\begin{proof} Suppose $\Gamma$ has infinitely many finite orbits; since $\Gamma$ does not
preserve any curve, these orbits form a  Zariski dense subset.  Let $g$ be a parabolic automorphism of $\Gamma$. 
If the fibration $\pi_g$ were $\Gamma$-invariant, then $\Gamma$ would preserve the curve $\bigcup_{y\in \Gamma(x)} \pi_g^{-1}(\pi_g(y))$
for every $\Gamma$-periodic point $x$. Thus, there is an element $h$ in $\Gamma$ that does not preserve $\pi_g$, and $h\inv \circ g\circ h\in \Gamma$ is a parabolic map associated to a different fibration. Hence  $\Gamma$ is non-elementary  (see \S \ref{subsub:non_elementary})
and Theorem \ref{mthm:main} shows that $\Gamma$ is a Kummer group. 
But, since $X$ is not abelian, a Kummer subgroup of $\Aut(X)$ admits an invariant curve 
(see Lemma~\ref{lem:invariant_curve_kummer}): this contradiction concludes the proof. 
\end{proof}

The next result is in the spirit of the ``dynamical Manin-Mumford problem''. 

\begin{cor}\label{cor:DMM}
Let $X$ be a smooth projective surface  
and $\Gamma$ be a subgroup of $\Aut(X)$, both 
defined over a number field. Suppose that  $\Gamma$ is non-elementary and contains parabolic elements. 
Let $C\subset X$ be an irreducible curve containing infinitely many periodic points of $\Gamma$. Then,
\begin{enumerate}[\em (1)]
\item either $C$ is $\Gamma$-periodic and is fixed pointwise by a finite index subgroup of $\Gamma$;
\item or $(X, \Gamma)$ is a Kummer group and $C$ comes from a translate of  an abelian subvariety (of dimension $1$).  
\end{enumerate}
In both cases the genus of $C$ is $0$ or $1$.  Thus, a
 curve of genus $\geq 2$  contains at most finitely many periodic points of $\Gamma$. 
\end{cor}

To be specific, with the notation of \S \ref{par:Kummer_structures},   the second assertion means the following: 
there is a translate $E+t$ of an elliptic curve $E\subset A$ such that 
$q_X(C)=q_A(E+t)$. Moreover, if we choose the origin of $A$ at a periodic point  of $\Gamma_A$, 
we can choose $t$ to be a torsion point of $A$. We keep these notations in the following proof.

\begin{proof}
Let ${\mathrm{Per}}(C)$ be the set of periodic points of $\Gamma$ in $C$; it
is Zariski dense in $C$, for $C$ is irreducible. The Zariski closure of $\Gamma({\mathrm{Per}}(C))$ is either a $\Gamma$-invariant curve  or $X$. 

In the first case $C$ is contained in  $D_\Gamma$, a finite index subgroup 
$\Gamma'\subset\Gamma$ 
preserves $C$, and the restriction $\Gamma'_{\vert C}$ has infinitely many periodic points in $C$. 
In this case $C$ has (arithmetic) genus 0 or 1 by \cite[Theorem 1.1]{diller-jackson-sommese}.
A   group of automorphisms of a curve with at least three periodic orbits is finite, because  it admits a finite index subgroup  fixing $3$ points; thus, a finite index 
subgroup of $\Gamma$ fixes $C$ pointwise.   

In the second case, Theorem~\ref{mthm:main} shows that  $(X, \Gamma)$ is a Kummer group. Since $C$ cannot be periodic, 
 its image $q_X(C)\subset A/G$ is a non-trivial curve  whose lift to $A$ contains a Zariski dense subset of  
 $\Gamma_A$-periodic points.  Choose one of these periodic points as the origin of $A$.
By  Proposition~\ref{pro:finite_orbits_torus} and Remark~\ref{rmk:torsion}, the
$\Gamma_A$-periodic points are exactly the torsion points of $A$, and 
conclusion~(2) follows from 
 Raynaud's theorem (formerly known as the Manin-Mumford conjecture) \cite{raynaud}.
\end{proof}

\subsection{Finitely generated groups} \label{subs:finitely_generated}
It turns out that  $\Gamma$ is often  defined over a number field when $X$ is. 

\begin{pro} \label{pro:automorphism_k}
Let  $X$ be a projective surface defined over a number field $\bfk$. Assume that 
$\Aut(X)$ contains a loxodromic element, and that $X$ is not an abelian surface.  
Then any finitely generated subgroup of $\Aut(X)$ is defined over a finite extension of $\bfk$. 
\end{pro}

\begin{cor}
If $X$ is a K3 or Enriques surface defined over a number field $\bfk$,  $\Aut(X)$ is defined over a finite 
extension of $\bfk$. 
\end{cor}

Indeed, $\Aut(X)$ is finitely generated in this case (see~\cite{Sterk}).

\begin{proof}[Proof of the proposition] 
It is enough to show that any automorphism $f\in \Aut(X)$ is defined over a finite extension of $\bfk$. 
Under our assumption, $\Aut(X)^*\subset\GL(H^*(X, \Z))$ is infinite, $\Aut(X)^0$ 
is trivial, and the homomorphism $f\in \Aut(X)\mapsto f^*\in \GL(H^*(X, \Z))$ has finite kernel (see~\cite[Theorem 10.1]{Cantat:Milnor}); 
more precisely, if $H$ is any ample divisor, the stabilizer of $[H]$ is a finite subgroup $\Aut(X;[H])$ of $\Aut(X)$.

Fix a finite extension $\bfk'$ of $\bfk$ and a basis of $\NS(X;\Z)$ given by classes of divisors $D_i$ which are defined over $\bfk'$.
Fix an ample divisor $H$ defined over $\bfk'$.
By assumption $X$ and the $D_i$ are defined by polynomial equations over $\bfk'$, in some~$\P^N$. Now, consider an automorphism $f$ of $X$, defined 
by polynomial formulas with coefficients in  some extension $\mathbf{K}$ of  $\bfk'$. Any field automorphism $\varphi\in \Gal(\mathbf{K}:\bfk')$
conjugates $f$ to an automorphism $f^\varphi$ of $X$: this defines a map
$\varphi\in \Gal(\mathbf{K}:\bfk')\mapsto f^\varphi \in \Aut(X)$. 
On the other hand 
$\langle (f^\varphi)^* [D_i]\vert [D_j]\rangle = \langle  f^* [D_i]\vert [D_j]\rangle$ for any pair   $(i,j)$ because 
the divisors $D_i$ are defined over $\bfk'$; thus, $(f^\varphi)^*=f^*$ on $\NS(X;\Z)$, and $f^\varphi\circ f^{-1}$ belongs to 
the finite group $\Aut(X;[H])$, so    the set $\set{f^\varphi\; ; \;  \varphi \in \Gal(\mathbf{K}:\bfk)}$ is finite, and we are done. 
\end{proof}

\subsection{Open problems}\label{subs:open} In the case of the affine plane ${\mathbb{A}}^2$, it follows from  \cite{dujardin-favre}
that any non-elementary subgroup of $\Aut(  {\mathbb{A}}^2_\bfk)$, for any number field
$\bfk$, has at most finitely many finite orbits (see~\cite{dujardin-favre} for the definition of ``non-elementary'' in this case). This motivates the following question:
 
\begin{que}\label{que:no_parabolic}~
Is Theorem~\ref{mthm:main} true without assuming the existence of a parabolic 
element in $\Gamma$?  \end{que}
 
To understand the difficulties behind Question~\ref{que:no_parabolic}, 
let us comment on  three arguments that required the hypothesis $\Gamma_{\mathrm{par}}\neq \emptyset$.
First, it was used to show that $\Pi_\Gamma\subset \NS(X;\R)$ is defined over
$\Q$ and to construct the projective surface $X_0$ (which is then    used  in the construction 
of the canonical stationary  height). The point is that in general 
 the contraction of the divisor $D_\Gamma$ is a well-defined complex analytic surface, but it is 
not projective (see~\cite[Example 2.10]{stiffness} and~\cite[\S 11]{cantat-dupont}). 
We expect that this issue could be circumvented by applying more advanced techniques from 
Arakelov geometry. 
Second, Theorem~\ref{thm:dense_active_saddles}   also relies
on the existence of parabolic elements; the point was to show that all
active periodic points of all loxodromic elements of $\Gamma$ cannot be simultaneously contained in some real surface. 
For instance, it is unclear to us whether there can exist a real projective surface $X_\R$, with 
a non-elementary subgroup $\Gamma\subset \Aut(X_\R)$, such that all periodic points of all elements $f\in \Gamma\setminus \{\id\}$ are contained 
in the real part $X(\R)$ of $X$. Third, parabolic automorphisms are 
crucially used in the classification of $\Gamma$-invariant probability measures given in~\cite{cantat_groupes, invariant}. We expect that 
the techniques from~\cite{br, stiffness} will soon lead to a complete classification 
of $\Gamma$-invariant probability measures, for any non-elementary group $\Gamma\subset \Aut(X)$. 
Such a classification would then open the way to 
an extension of  Theorem~\ref{mthm:main} to all non-elementary groups (defined over a number field).

\begin{rem}\label{rem:questions_Kawaguchi} 
 In \cite[Question 3.3]{Kawaguchi:2013}, Kawaguchi formulates
two interesting questions which are closely related to our 
main results as well as to Question~\ref{que:no_parabolic}.
\begin{enumerate}[(1)]
\item First, he asks whether two loxodromic automorphisms $f$ and $g$ of a complex projective surface 
$X$ with a Zariski dense set of common periodic points automatically have the same periodic orbits. 
As it is formulated, the answer is no, because of Kummer examples: if we start with two loxodromic automorphisms of an abelian surface $A$ 
fixing the origin and generating a non-elementary subgroup, then one can blow-up the origin, and the automorphisms 
lift to automorphisms with the same periodic orbits (coming from torsion points of $A$), 
{\emph{except for their fixed points on the exceptional divisors, which do differ}} (see Lemma~\ref{lem:invariant_curve_kummer}, Assertion (2)).  
So, his question needs to be modified by asking whether $f$ and $g$  have the same periodic points, \emph{except for finitely many of them}. 
  
\item The second part of \cite[Question 3.3]{Kawaguchi:2013} asks whether two loxodromic automorphisms 
of a Wehler surface  having  a Zariski dense set of common periodic points automatically generate an elementary group. There are (singular)  Kummer examples in 
the Wehler family (see \cite[\S 8.2]{Cantat:Panorama-Synthese}), and they provide counter-examples to this question. 
Taking these comments into consideration, Kawaguchi's second question can now be reformulated as:
{\emph{if  two loxodromic automorphisms $f$ and $g$ of a complex projective surface $X$ have a Zariski dense set of common 
periodic points, then is it true that
either $f^m=g^n$ for some $m,n\geq 1$,  or $f$ and $g$  generate a Kummer group?}} This seems
harder than Question~\ref{que:no_parabolic}, 
because common periodic points do not directly provide common periodic orbits.
A natural companion to the last question is: \emph{when do two loxodromic automorphisms have the same measure 
of maximal entropy?}  
\end{enumerate}
\end{rem}

One may also ask for effective bounds on the cardinality of a maximal finite $\Gamma$-invariant subset of $X(\C)$ in terms of the data
(compare \cite{demarco-krieger-ye}). Proposition \ref{pro:number_fixed_points} says that such a bound 
should at least depend on the degrees of the generators of $\Gamma$. 
 
Lastly, a natural question is whether the number field assumption in Theorem~\ref{mthm:main} is necessary at all: this is what the  next section is about. 

%%%%%%%%%%%%%%%%%%%%%%%%%%%%%%%%%%%
%%%%%%%%%%%%%%%%%%%%%%%%%%%%%%%%%%%
\section{From number fields to $\C$}\label{sec:KtoC}
%%%%%%%%%%%%%%%%%%%%%%%%%%%%%%%%%%%
%%%%%%%%%%%%%%%%%%%%%%%%%%%%%%%%%%%

In this section we show how a specialization argument allows 
to extend Corollary~\ref{cor:finite_orbits_finite} beyond the number field case.  
A full generalization of Theorem~\ref{mthm:main} to complex coefficients would  
require further ideas (see \S \ref{subs:discussion} for a short discussion). For concreteness we first treat the case 
of Wehler surfaces and then explain the extra ingredients required to address the general case. 

%%%%%%%%%%%%%%%%%%%%%%%%%%%%%%%%%%%
\subsection{Wehler surfaces}
%%%%%%%%%%%%%%%%%%%%%%%%%%%%%%%%%%%
We resume  the notation from  \S \ref{sec:wehler}. 
The  complete linear system~$\abs{L}$ parameterizing Wehler surfaces is a projective space of dimension $26$, 
which yields a moduli space of dimension $17$ modulo the action of $\Aut(\P^1)^3$.
There is a dense, Zariski open subset $W_0\subset \abs{L}$ such that if $X\in W_0$, then $X$ is a smooth Wehler surface
and for every $1\leq j\neq k\leq 3$, 
$\pi_{j,k}:X\to \P^1\times  \P^1$ is a finite morphism. Let $\Gamma_X$ be the group generated by the three 
involutions $\sigma_i$. 
 
\begin{thm}\label{thm:wehler_complex}
If $X$ is a smooth Wehler surface for which the projections $\pi_{j,k}\colon X\to \P^1\times \P^1$ are finite maps,
then $\Gamma_X$ admits only finitely many finite orbits. 
\end{thm}

For the proof we  follow the approach of  \cite[\S 5 and Thm D]{dujardin-favre} closely. 

\begin{proof}   Let $G = (\Z/2\Z)\star(\Z/2\Z)\star (\Z/2\Z)$ with generators $a_1, a_2, a_3$ 
and let  $\chi: G\to \Aut(X)$ be the unique homomorphism such that $\chi(a_i) = \sigma_i$. By definition,   $\Gamma_X = \chi(G)$. 
Let $c_i$ denote the class of the curve $X\cap \set{z_i = \cst}$. The subspace $\Z c_1\oplus \Z c_2 \oplus \Z c_3$  of $\NS(X, \Z)$ is invariant by $\chi(G)^*\subset \GL(\NS(X, \Z))$ and
this representation does not depend on $X\in W_0$: the matrices of 
the involutions $\sigma_i^*=\chi(a_i)^*$ in the basis $(c_1,c_2,c_3)$ have constant integer coefficients (see e.g. \cite[Lem. 3.2]{stiffness}). 
Thus we  can define $G_{\mathrm{lox}}$ (resp.  $G_{\mathrm{par}}$) to be the set of elements $h\in G$ such that 
for any $X\in W_0$,  $\chi(h)$ acts as a loxodromic (resp. parabolic) map on  $\NS(X, \Z)$.
Here, we implicitly use the fact that  the type of $h\in \Aut(X)$ is the same as the type of $h^*$ in restriction to any $h^*$-invariant 
subspace of $H^{1,1}(X;\R)$ 
on which the intersection form is not negative definite. 
In particular, the type of $h$ coincides with the type of $h^*$ as an isometry of ${\mathrm{Vect}}(c_1,c_2,c_3)$.

Fix a system of affine coordinates $(x,y,z)$ and write the equations of Wehler surfaces as in Equation~\eqref{eq:general_X}; 
this gives a system of homogeneous coordinates on $\abs{L}$, and $\abs{L}$ can be considered as a projective space
defined over $\Q$. Then, endow $\abs{L} \simeq \P^{26}(\C)$ with the $\overline \Q$-Zariski topology. 
Fix $X\in W_0$, let  $b\in \P^{26}$ (for ``base point'') denote   the  parameter corresponding to   $X$, and 
$S$ be the closure of 
$\set{b}$ for this topology: this is a subvariety of $\P^{26}$ defined over $\overline \Q$ in 
which $b$ is, by construction, a generic point. We    put $S_0=S\cap W_0$,
and we restrict the universal family $\mathcal X \to \P^{26}$ 
of Wehler surfaces to a family  
 $\mathcal X_{S_0}  \to S_0$, 
with  a fiber preserving action of the group $G$. The fiber over $s$ is denoted by $\X_s$ 
and the natural homomorphism $G\to \Aut(\X_s)$ by $\chi_s$; thus, $X$ coincides with $\X_{b}$.  

\begin{lem}\label{lem:generic1}
For every $s\in S_0(\C)$, 
\begin{enumerate}[{\em (1)}]
\item $\X_s$ is a smooth $K3$ surface which does not contain any  fiber of $\pi_{i,j}$, $i\neq j\in \{1,2,3\}$;
\item $h\in G$ belongs to $G_{\mathrm{lox}}$ (resp.  $G_{\mathrm{par}}$) if and only if $\chi_s(h)$ is 
a loxodromic (resp. parabolic) element of $\Aut(\X_s)$;
\item $\chi_s(G)$ is a non-elementary subgroup of $\Aut(\X_s)$ without invariant curve. 
\end{enumerate}
\end{lem}

\begin{proof}[Proof of Lemma~\ref{lem:generic1}]
The first assertion follows from the results of \S \ref{par:notation_wehler} and 
the inclusion $S_0\subset W_0$. Likewise $\chi_s(G)$ has no invariant curve by \S \ref{par:invariant_curves_wehler}. 
The second assertion follows  from our preliminary remarks on the 
definition of $G_{\mathrm{lox}}$ and  $G_{\mathrm{par}}$, and it also implies that 
$\chi_s(G)$ is non-elementary.\end{proof}

Assume  now by contradiction that $\Gamma_X$ admits infinitely many finite orbits. Then:

\begin{lem}\label{lem:generic2}
For every $s\in S_0(\C)$, $\chi_s(G)$ has 
infinitely many finite orbits. 
\end{lem}

This lemma concludes the proof of the theorem. Indeed pick  $s\in S_0(\overline \Q)$. By 
Lemma \ref{lem:generic1},  $\chi_s(G)$ is non-elementary,   contains parabolic elements; 
and the Zariski closure of the set of finite orbits of $\chi_s(G)$ coincides with
$\X_s$, for otherwise it would be an invariant curve. 
Then, by Theorem \ref{mthm:main}, $(\X_s, \chi_s(G))$ must be a Kummer group. But $\X_s$ 
is a K3 surface and a Kummer group on a 
non-abelian surface   admits   an invariant curve, so that we get a contradiction with Lemma \ref{lem:generic1}.(3). 
\end{proof}

\begin{proof}[Proof of Lemma \ref{lem:generic2}] We first 
 describe the set of finite $\Gamma_X$-orbits as a countable union of subvarieties by 
 arguing as in \S \ref{par:Zd}. Let $G_d$ be the intersection of the kernels of all homomorphisms from 
 $G$ to groups of order $\leq d!$; it is a   finite index subgroup of $G$.  
 For any action of $G$, if the orbit of a point $x$ has cardinality $\leq d$, then 
 $x$  is fixed by $G_d$. Conversely if $x$ is fixed by $G_d$, then its $G$-orbit is finite. 
 Define a subvariety $Z_d$ of $X$ by 
 \begin{equation}
 Z_d = \set{x\in X \; ; \;  \ \forall g\in G_d, \  \chi(g )(x) = x}.
 \end{equation}
Finally put $Z = \bigcup_{d\geq 1} Z_d$. Then the $\Gamma_X$-orbit of $x\in X$ is finite  if and only if $x\in Z$. 
We can now define a subvariety $\mathcal Z_d$ of $\mathcal X_{S_0}$ which is 
the  fibered analogue of $Z_d$, namely
\begin{equation}
\mathcal Z_d = \set{(s,x)\; ; \;\ x\in \X_s, \  \forall g\in G_d, \  \chi_s(g )(x) = x},
\end{equation}
and put $\mathcal Z = \bigcup \mathcal Z_d$. We let $\mathcal Z_s$ (resp. $\mathcal Z_{d,s}$) 
be the intersection of $\mathcal Z$ (resp. $\mathcal Z_{d}$) with $\X_s$.
  
Set $f = a_3a_2a_1 \in G$. An explicit computation shows that 
$f\in G_{\mathrm{lox}}$ and the eigenvalues of $\chi_s(f)^*$ on ${\mathrm{Vect}}(c_1,c_2,c_3)$
are $-1$,  $\lambda(f)=9+4\sqrt{5}$, and $1/\lambda(f)$. The eigenline corresponding to 
$-1$ is $\R\cdot(c_1-3c_2+c_3)$, its orthogonal complement in ${\mathrm{Vect}}(c_1,c_2,c_3)$ is the plane 
$\Pi_{\chi_s(f)}$, and this plane contains the class $c_1+2c_2+c_3$. This class is ample, because 
it is a convex combination, with positive   coefficients, of the Chern classes $c_i$ of the line bundles 
$\pi_i^*({\mathcal{O}}_{\P^1}(1))$, $i=1,2,3$. Since any invariant curve must be orthogonal to $\Pi_{\chi_s(f)}$,
we deduce that $\chi_s(f)$ has no invariant curve (for all $s\in S_0$).

Now assume by contradiction that there is a parameter $t\in S_0(\C)$ such that $ \mathcal Z_t$ is finite. Let 
$P_n$ be the set of fixed points of $\chi(f^n)$, so that $P   = \bigcup_n P_n$ is the set of all periodic points 
of $\chi(f)$; likewise let $\mathcal P_n$ and $\mathcal P$ be their respective fibered versions. 
Note that $\mathcal Z\subset \mathcal P$. 
For fixed $n$, let $\mathcal Y_n$ be the (reduced)  subvariety  of $\mathcal X_{S_0}$ 
whose underlying set is  $\mathcal Z\cap \mathcal P_n$.
More precisely the sequence  of subvarieties 
$\mathcal Y_n^m:= \bigcup_{d=1}^m \mathcal Z_d\cap \mathcal P_n$ is non-decreasing with $m$, 
so it stabilizes, and 
we define $\mathcal Y_n =\mathcal Y_n^m $ for $m$ sufficiently large; its fibers will be denoted by $\mathcal Y_{n,s}$ ($\mathcal Y_{n,s}$ is the 
intersection of $\mathcal Y_n$ with $\X_s$, it may be non-reduced).  For the generic point 
$b$ the cardinality of $\mathcal Y_{n,b}$ tends to infinity with~$n$. 

Now,  the argument is identical to that of Lemma 5.3 and Theorem D in \cite{dujardin-favre} (\footnote{Our 
   setting   is actually simpler   since we are dealing  with automorphisms on a projective surface 
rather than birational mappings, so the properness issue analyzed 
 in \cite{dujardin-favre} is not relevant here.}). For $x\in\mathcal Y_{n, s}$, 
its multiplicity $\mult(x, \mathcal Y_{n,s})$ 
as a point  in $\mathcal Y_{n, s}$ 
is equal to its multiplicity as a fixed point of $\chi_s(f)^n$. 
 Nakayama's lemma implies  that the function 
 $$s\mapsto \sum_{x\in   
 \mathcal Y_{n,s }} \mult(x, \mathcal Y_{n,s})$$  is upper semicontinuous for the Zariski 
 topology, hence 
\begin{equation}\label{eq:multiplicities}
\sum_{x\in\mathcal Y_{n, t}} \mult(x,\mathcal Y_{n,t}) \geq \sum_{x\in   
\mathcal Y_{n,b }} \mult(x,\mathcal   Y_{n, b }) \tendvers +\infty. 
\end{equation}
 On the other hand, $ \mathcal Z_t$ is a finite set, so there exists $n_0$ such that for all $n\geq 1$, 
$\mathcal Y_{n, t}\subset\mathcal P_{n_0,t}$, and the theorem of Shub and Sullivan \cite{shub-sullivan} asserts that for every 
$x\in\mathcal P_{n_0,t}$, the multiplicity of $x$ as a fixed point of $\chi_t(f)^n$ is bounded as $n\to \infty$. This 
contradicts \eqref{eq:multiplicities} and concludes the proof. 
\end{proof}

\subsection{Groups without invariant curve}
Let us recall Theorem~\ref{mthm:KtoC}. 

\begin{thm}\label{thm:KtoC}
Let $X$ be a compact K\"ahler surface and let $\Gamma$ be a  subgroup of $\Aut(X)$. Assume that (i)  $X$ is not an abelian surface, and (ii) $\Gamma$ contains a parabolic element and has no invariant curve. 
Then $\Gamma$ admits only finitely many finite orbits. 
\end{thm}

\begin{proof} The idea is of the proof is the same as that 
of Theorem \ref{thm:wehler_complex}, however new technicalities arise. 
As in Corollary~\ref{cor:finite_orbits_finite}, $\Gamma$ is automatically non-elementary, so  $X$ is projective.
Arguing by contradiction,  we suppose that $\Gamma$ admits infinitely many 
finite orbits. Applying Theorem~\ref{mthm:invariant_curve_loxodromic}, 
we fix $f\in \Gamma_{\mathrm{lox}}$ without invariant curve. We also fix a parabolic element $g\in \Gamma$.

{\bf{Step 1.--}} {\emph{Geometry of $X$}}.-- Since $\Gamma$ is non-elementary, $X$ is   a blow-up
of an abelian surface,   of a K3 surface,   of an Enriques surface, or of the projective plane (see~\cite[Thm. 10.1]{Cantat:Milnor}). 
In the first three cases, there is a unique minimal model $\varphi\colon X\to \overline{X}$, and the exceptional 
divisor of $\varphi$ is $\Aut(X)$-invariant. Since $\Gamma$ has no invariant curve, $X$ is already 
equal to its minimal model $\overline{X}$, and since by assumption  $X$ is not abelian, in this case
 $X$ is a K3  or an Enriques surface.

\smallskip

{\bf{Step 2.--}} {\emph{Reduction to a finitely generated subgroup}}.-- 
The group generated by $f$ and $g$ satisfies assumption~(ii) and since it is contained
in $\Gamma$ it also admits   infinitely many finite orbits. From now on,
 we replace $\Gamma$ by $\langle f, g\rangle$ 
and assume $\Gamma$ to be finitely generated. 
 
\smallskip
 
{\bf{Step 3.--}} {\emph{Specialization formalism}}.-- Embed $X$ into a projective space $\P^N_\C$. Fix a finite set of divisors $E_j$ in $X$ 
whose classes form a basis of $\NS(X;\Z)$, let $H\subset X$ be a hyperplane section, and let $\Omega$ 
be a non-trivial rational section of $K_X^{\otimes 2}$, where $K_X$ is the canonical bundle. If $X$ is a K3 or an Enriques surface, we assume that $\Omega$ is regular, hence does not vanish.
Let $R\subset \C$ be the $\overline{\Q}$-subalgebra generated by the coefficients of a system of homogeneous equations for 
$X$,  the $E_j$,  $H$, and $\Omega$, and by the coefficients of the formulas defining a finite symmetric set of generators of 
$\Gamma$. (We shall actually further enlarge $R$ in \S \ref{par:halphen}.)

Let $K=\mathrm{Frac}(R)$. It is the field of rational functions of some algebraic variety $V$, 
defined over $\overline{\Q}$. 
There is a dense, Zariski open subset $S$ of $V$, which   may be assumed to be an affine subset, 
such that all elements of $R$ 
correspond to regular functions on~$S$. Note that in what follows, 
by Zariski topology we mean the $\overline \Q$-Zariski topology. Nevertheless, since 
 we will use transcendental arguments, $S(\C)$ will also be considered as a complex analytic space endowed with its Euclidean topology. 

By specialization, i.e. evaluation of the elements of $R$ at $s\in S$, 
we can view $X\subset \P^N$, $\Gamma$, the $E_j$, $H$, and $\Omega$  as 
families over  $S$;  that is, there is  a proper morphism  $\pi: \X\to S$ 
endowed with a group of fiber preserving automorphisms %\romain{vraiment des automorphismes de $\mathcal X$ au dessus de $S$?}
$\widetilde \Gamma$, together with a (complex) base point $b\in S$ so that 
the fiber $\mathcal X_b$  may be identified with $X$, and furthermore  $\widetilde{\Gamma}_b = \Gamma$, ${\mathcal {E}}_{j,b}=E_j$, ${\mathcal{H}}_b=H$, $\widetilde \Omega_b   = \Omega$, etc.
The point $b\in S$ may be thought of as the generic point of $S$ 
(i.e. its closure for the   Zariski topology is $S$)  so
in particular $b$ is a regular point of $S$ and $S$ is smooth in 
a complex neighborhood of $b$. If $X$ is a K3 or an Enriques surface, changing $S$ into some Zariski dense affine open subset, 
we may assume that $\widetilde{\Omega}_s$ does not vanish on any $\cX_s$. 

\smallskip

{\bf{Step 4.--}} {\emph{Types of automorphisms and invariant curves}}.--

\begin{lem}\label{lem:semi-continuity_type}
There is a Zariski open subset $S_1\subset S$ such that: 
\begin{enumerate}[\em (1)]
\item above $S_1(\C)$, the projection $\cX\to S$ is a submersion; for $s\in S_1(\C)$, $\mathcal X_s$ is smooth  and it  is not an abelian surface;
\item for $s\in S_1(\C)$,  $f_s$ is loxodromic and there exists a Euclidean neighborhood $B$ of $b$ such that for $s\in B$, $f_s$ admits no invariant curve;  
\item  for $s\in S_1(\C)$,  $g_s$ is parabolic. 
\end{enumerate}
\end{lem}

\begin{proof}[Proof of (1)] The surface  $\cX_b$ is smooth, and by construction there is a Zariski dense open subset $S_1$ of $S$ above which 
$\cX\to S$ is a submersion, so the set of parameters $s$ for which  $\mathcal X_s$ is singular is a proper Zariski 
closed set $S_1$ (we will further reduce $S_1$ finitely many times in the proof, keeping the same notation).
For the second conclusion, we observe that there is a Zariski open subset on which $\pi: \X\to S$ is 
a submersion, hence by Ehresmann's lemma the fibers in this open subset are diffeomorphic to $X$.   On the 
other hand a  surface which is diffeomorphic to a complex torus and possesses
 a non-elementary group of automorphisms is automatically 
an abelian surface. Since $X$ is not abelian, we  conclude that the same is true for  
any fiber $\cX_s$, $s\in S_1$.
\end{proof}

\begin{proof}[Proof of (2)] The loxodromic nature 
 of $f_s$ follows from the lower semi-continuity of the dynamical degree for birational mappings on surfaces  (see   \cite[Thm. 4.3]{Xie:Duke}), which we apply here for the $\C$-Zariski topology. 
 
Let us show that $f_s$ has no invariant curve for   $s\in S_1$ close to $b$.
Indeed, recall from Proposition~\ref{pro:degree_invariant_curve} that there is a uniform 
bound on the degree of an invariant curve. Here we compute the degree of a curve on $\cX_s$ 
(resp. of  $f_s$)
with respect to the normalized ample class  $ \langle\mathcal H_s\vert \mathcal H_s\rangle^{-1/2}[\mathcal H_s]$ induced by the hyperplane section $\mathcal H_s$. 
If $c_X$ is as in Proposition~\ref{pro:degree_invariant_curve},
 the inequality $\frac{1}{2}[{\mathrm{Div}}(\widetilde \Omega_s)]\leq c_X \langle \mathcal H_s\vert \mathcal H_s\rangle^{-1/2}[\mathcal H_s]$  is satisfied on a Zariski open set. 
Then by Bishop's theorem, if $(s_i)$ is a sequence of points converging to $b$ such that $f_{s_i}$ preserves a curve $C_i$, 
we can extract a subsequence along which $(C_i)$ converges towards a curve $C$ in $\cX_b$ 
(see \cite[\S 16]{chirka} for the relevant notions). 
This curve is $f_b$-invariant, which contradicts  our assumption on~$f$.
\end{proof}

The proof of the the third assertion of Lemma~\ref{lem:semi-continuity_type} is a little tedious and 
will be  postponed to Section~\ref{par:parabolic_semi-continuity} below. 

\smallskip

{\bf{Step 5.--}}  {\emph{Conclusion}}.-- We pick a point $t$ in $S_1(\overline{\Q})\cap B$, and argue exactly as in the case of Wehler surfaces. 
Indeed observe first that the assumptions of Corollary~\ref{cor:finite_orbits_finite} are satisfied at the parameter $t$. Next, 
since all periodic  points of $f_t$ of a given period are isolated, we 
 can apply to $f_t$ the strategy of the proof of Theorem \ref{thm:wehler_complex},
based on Nakayama's lemma and the theorem of Shub and Sullivan; it implies that $\Gamma_t$ has infinitely many periodic orbits on $\cX_t$, thereby reaching the desired contradiction. 
\end{proof}

\subsection{Proof of Lemma~\ref{lem:semi-continuity_type}(3)} \label{par:parabolic_semi-continuity} 
By Step 1, $X$ is a K3 surface, an Enriques surface, or a blow-up of the projective plane.  
By the lower semi-continuity of the dynamical degree, for every $s\in S_1$, $g_s$ is parabolic or elliptic, 
so we need to show that the set of parameters for which $g_s$ is elliptic is Zariski closed.  

\subsubsection{K3 and Enriques surfaces}
Assume that $X$ is a K3 (resp. an Enriques) surface.
Above  $S_1(\C)$, every fiber $\cX_s$ has the diffeomorphism type of $\cX_b$, in particular it is simply connected and $K_{\cX_s}$ is trivial (resp. its fundamental group is $\Z/2\Z$ and $K_{\cX_S}^{\otimes 2}$ is trivial),  so it is also a K3 (resp. an Enriques) surface, for K3 (resp. Enriques) are characterized by these properties (see~\cite[Chap. VI]{BHPVDV}). For such a surface, the group $\{h\in \Aut(\cX_s)\; ; \; h^*=\id {\text{ on }} H^2(\cX_s;\Z)\}$ has at most $4$ 
elements (see~\cite{Mukai-Namikawa}).  The second Betti number is fixed, equal to 
$22$ (resp. $10$), and if $h^*\in \GL(H^2(\cX_s;\Z))$ has finite order, then its 
order divides some fixed integer $k$, because $\GL_{22}(\Z)$ (resp. $\GL_{10}(\Z)$) 
contains a finite index, torsion free subgroup. Thus, $g_s$ 
is elliptic if and only if $g_s^{4k}=\id$. This  implies that the set of
 parameters $s$ for which $g_s$ is elliptic is Zariski closed and 
 does not contain $b$, and we are done in this case.

%%%
\subsubsection{Rational surfaces}\label{par:halphen}
%%%
Now, we assume that $X$ is rational.
This case is slightly more delicate because there exists automorphisms of $\P^2$ of arbitrary large finite order. 

Let  $\pi_g\colon X\to B$ be the invariant fibration of $g$, with $B=\P^1$ since $X$ is rational. Changing $g$ by some positive iterate, we may assume that the action of 
$g$ in the base $B$ is the identity. 
As explained in~\cite{cantat-dolgachev, cantat-guirardel-lonjou}, $\pi_g$ comes from a Halphen pencil; in particular, there is 
a pencil of curves in $\P^2$, defined by some rational function $\varphi\colon \P^2\dasharrow \P^1$, and a birational morphism $\eta\colon X\to \P^2$ that blows up the base points of this pencil (and possibly other points too), such that $\pi_g$ coincides with $\varphi\circ \eta$. The last blow-up which is necessary 
to resolve the indeterminacies of $\varphi$ provides a curve  which is transverse to the fibration and has negative self-intersection. So, there is 
an irreducible multi-section $E$ of $\pi_g$ such that $E^2<0$.

Let us add to our $\overline{\Q}$-algebra $R$ the coefficients of the formulas defining $\pi_g$, $\varphi$, $\eta$, $E$, etc. 
Reducing $S_1$ if necessary, we get a family of automorphisms $g_s$ preserving each fiber of a 
genus $1$ fibration $\pi_{g,s}\colon \cX_s\to \P^1$, with an irreducible multisection $E_s$ of negative self-intersection.
As for K3 and Enriques surfaces, the following lemma finishes the proof. 

\begin{lem}
There is an integer $\ell>0$ such that if $s\in S_1$ and $g_s$ is elliptic, then $g_s^\ell=\id$.
\end{lem}

\begin{proof} Set $m=\langle [E]\vert [F]\rangle$ where $F$ is any fiber of $\pi_g$.
Above $S_1$ the surfaces $\cX_s$ are pairwise diffeomorphic, so they have the same second Betti number and there is 
an integer $k>0$ such that $(h^*)^k=\id$ for every elliptic automorphism of $\X_s$, for every $s\in S_1$. Now, if $g_s$ is elliptic, 
then $(g_s^k)^*[E_s]=[E_s]$ and this implies $g_s^k(E_s)=E_s$. Since $g_s$ preserves every fiber, and $E_s$ intersects every 
fiber in at most $m$ points, we deduce that $g_s^{k\cdot m!}$ fixes a point in each fiber. 
But an automorphism of a curve of genus $1$ which fixes a point has order at most $12$, so $g_s^{12k \cdot m!}=\id_{\cX_s}$, and we are done.
\end{proof}

\subsection{Discussion}\label{subs:discussion}
 It would be interesting  to extend   Theorem~\ref{mthm:main} in its general form beyond number fields, 
that is,  without assuming that $D_\Gamma=\emptyset$. 
Fix $(f,g)\in \Gamma_{\mathrm{lox}}\times \Gamma_{\mathrm{par}}$, as above.
The main difficulty appears in the following situation: 
$\Gamma$ fixes $D_\Gamma$ pointwise, and for every parameter $s\in S({\overline{\Q}})$, 
the alleged Zariski dense set of finite orbits of $\Gamma$ specializes as a 
 finite subset of $\cX_s$ which intersects $(D_\Gamma)_s$. In that case, the theorem of Shub and Sullivan 
does not apply directly because it only deals with  isolated fixed points; so, a finer understanding 
 of the Lefschetz fixed point formula is required. 
 We believe that the tools introduced in~\cite{Iwasaki-Uehara} and in an unpublished
chapter of   Xie's  thesis~\cite{xie:thesis} may lead to a solution of this problem. 

%%%%%%%%%%%%%%%%%%%%%%%%%%%%%%%%%%%%%%%
%%%%%%%%%%%%%%%%%%%%%%%%%%%%%%%%%%%%%%%
%%%%%%%%%%%%%%%%%%%%%%%%%%%%%%%%%%%%%%%
\section{Canonical vector heights}\label{sec:canonical_vector_height}
%%%%%%%%%%%%%%%%%%%%%%%%%%%%%%%%%%%%%%%
%%%%%%%%%%%%%%%%%%%%%%%%%%%%%%%%%%%%%%%
%%%%%%%%%%%%%%%%%%%%%%%%%%%%%%%%%%%%%%%
Let $\bfk$ be a number field and ${\overline \bfk}\simeq {\overline \Q}$ be an algebraic closure of $\bfk$. 
Let $X$ be a projective surface defined over $\bfk$ and $\Gamma$ be a subgroup of $\Aut(X_\bfk)$. 
We consider the vector space  
\begin{equation}
\Pic(X; \R) = \Pic(X_{\overline \bfk})\otimes_\Z \R
\end{equation}
of  $\R$-divisors of $X_{\overline\bfk}$ modulo linear equivalence; doing so, we annihilate the torsion part of $\Pic^0(X)$.
Keep in mind  that  when $X$ is birational to an abelian variety,  the vector 
space $\Pic^0(X; \R):=  \Pic^0 (X  )\otimes_\Z \R$ is infinite-dimensional.
The Weil height machine extends to $\Pic (X; \R)$ by $\R$-linearity (see~\cite[\S B.3.2]{Hindry-Silverman}). 
Recall from \S \ref{eq:canonical_vector_height} that 
a {\bf{canonical vector height}} on $X(\overline \bfk)$ for the group $\Gamma$ is, by definition, a function
$h: \Pic(X; \R) \times X(\overline \bfk) \to \R_+$ such that 
\begin{enumerate}[{\em (a)}]
\item[(a)] $h$ is {\emph{linear}} with respect to the first factor $L\in \Pic (X; \R) $;
\item[(b)] 
for every $L\in \Pic (X; \R)$, $h(L, \cdot)$ is a {\emph{Weil height}} associated to $L$;
\item[(c)] $h$ is {\emph{$\Gamma$-equivariant}}: for every $f\in \Gamma$, 
 $h(L, f(x)) = h(f^*L, x)$.
\end{enumerate}
Note that if  $\Pic(X_{\overline{\bfk}})$ is tensorized by $\Q$ instead of $\R$ and Property (a) is stated over $\Q$ we get an equivalent notion. 
Given any $\Gamma$-invariant subspace $V\subset \Pic (X; \R)$, one may also study the notion of 
{\bf{restricted}} canonical vector height $h\colon V\times X(\overline{\bfk})\to \R$. This is most significant when $V$  
contains classes with positive self-intersection, in which case it surjects onto $\Pi_\Gamma$ under the natural map $D\in \Pic (X; \R) \to [D]\in \NS(X; \R)$.
 (In the following we use brackets to distinguish a class in $\NS(X; \R)$
 from a class in $\Pic(X; \R)$.) 
 
 If $A_{\bfk}$ is an abelian variety and $\Gamma$ is a subgroup of 
$\Aut(A_{\bfk})$ fixing its neutral element, the Néron-Tate height on $A$ is a canonical vector height for $\Gamma$ (see~\cite[Theorem B.5.6]{Hindry-Silverman}). The same holds more generally if the neutral element is $\Gamma$-periodic, because 
in this case  $\Gamma(0)$ is made of torsion points (see Remark \ref{rmk:torsion}).  

In this section, we describe automorphism groups of surfaces which are non-elementary, contain parabolic elements, and  
possess a (restricted) canonical vector height $h_{\mathrm{can}}$: Theorems~\ref{mthm:canonical_vector_height},~\ref{mthm:prime}, 
and~\ref{mthm:second}  show that $(X,\Gamma)$ is a Kummer group and  $h_{\mathrm{can}}$ is  derived from the N\'eron-Tate height. 

\subsection{Invariant classes and canonical vector heights}\label{par:derived_height}
In the following lemmas, $h_{\mathrm{can}}$ is a restricted 
canonical vector height for $(X,\Gamma)$, defined in  some $\Gamma$-invariant subspace 
$V_{\mathrm{can}}\subset \Pic(X;\R)$. We shall say that a class $[E]$ in $\NS(X;\R)$ is {\bf{almost $\Gamma$-invariant}} 
if  $f^*[E]=\pm [E]$  for all $f$ in $\Gamma$.

\begin{lem}\label{lem:perturbation_canonical_vector_height} 
Let $[E]\in\NS(X;\R)$ be almost $\Gamma$-invariant. The function  
\[
h_{[E],\varphi} (D,x) = h_{\mathrm{can}}(D,x)+ \langle [E] \vert D\rangle \varphi(x)
\]
is a restricted canonical vector height on $V_{\mathrm{can}}\times X(\overline{\bfk})$ if and only if 
either $[E]$ is orthogonal to $V_{\mathrm{can}}$, 
or $\varphi\colon X({\overline \bfk})\to \R$ is bounded and satisfies $\varphi(x)[E]=\varphi(f(x))f^*[E]$ for all $f\in \Gamma$.
\end{lem}

In this situation, we shall say that $h_{[E], \varphi}$ is {\textbf{derived from}} the  height $h_{\mathrm{can}}$.

\begin{proof} If $[E]$ is orthogonal to $V_{\mathrm{can}}$, then $h_{[E],\varphi}=h_{\mathrm{can}}$ on $V_{\mathrm{can}}\times X(\overline{\bfk})$ and there is nothing to prove. 
Otherwise, we can fix a class $D\in V_{\mathrm{can}}$  such that $\langle [E] \vert D\rangle\neq 0$.
If $h_{[E],\varphi}$ is a canonical vector height,  
then $\varphi =\langle [E] \vert D\rangle^{-1} (h_{[E],\varphi}-h_{\mathrm{can}})(D,\cdot)$ is bounded, because 
$h_{[E],\varphi}(D,\cdot)$ and $h(D,\cdot)$
are Weil heights associated to
 the same divisor. Furthermore $\varphi$ satisfies  $\varphi(x)[E]=\varphi(f(x))f^*[E]$  because $h_{[E],\varphi}$
and $h_{\mathrm{can}}$ are $\Gamma$-equivariant and $\langle [E]\vert f^*(D)\rangle = \langle f^*[E]\vert  D\rangle$ for all $f\in \Gamma$. 
The reverse implication is straightforward.
\end{proof}

\begin{lem}\label{lem:no_inv_curve_canonical_height}
Assume that $\Gamma$ contains a loxodromic element. 
Let $C\subset X$ be an irreducible $\Gamma$-periodic  curve. If  the class of $C$ belongs to $V_{\mathrm{can}}$, or if $V_{\mathrm{can}}$ 
contains a $\Gamma$-periodic class $D$ such that ${\mathcal{O}}(D)_{\vert C}$ is ample, then  the restriction homomorphism 
$f\in \mathrm{Stab}_\Gamma(C) \mapsto f_{\vert C}$
has  finite image. 
\end{lem}

\begin{proof} If $C$ is $\Gamma$-periodic then its self-intersection is negative, 
the restriction of ${\mathcal{O}}_X(-C)$ to $C$ has positive degree, and $({\mathcal{O}}_X(-C))_{\vert C}$ is therefore ample. 
So, it is enough to consider the case where  $V_{\mathrm{can}}$ contains a periodic class  $D$ such that 
${\mathcal{O}}(D)_{\vert C}$ is ample.

By changing $\Gamma$ in a finite index subgroup, we may assume $\Gamma(C)=C$.
If $\sigma$ is an automorphism of $\C$ over $\bfk$, it maps $C$ to a $\Gamma$-invariant curve $C^\sigma$. There are only finitely 
many $\Gamma$-invariant irreducible curves (the components of $D_\Gamma$), because $\Gamma$ contains a loxodromic element (see \S \ref{par:periodic_curves_of_loxodromic}).
Thus, the orbit of $C$ under the group of automorphisms of $\C$ over $\bfk$ is finite and $C$ is defined over a number field.
In particular, $C({\overline \bfk})$ is dense in~$C(\C)$.

Set $\Gamma' = \mathrm{Stab}_\Gamma(D)$,  and  Pick  $x_0\in C({\overline \bfk})$.  Then 
$h_{\mathrm{can}}(D, y)= h_{\mathrm{can}}(D, x_0)$ for every $y$ in  $\Gamma'(x_0)$; since
$h_{\mathrm{can}}(D,\cdot)$ is a Weil height for $D$, and ${\mathcal{O}}(D)_{\vert C}$ is ample, Northcott's theorem implies that 
$\{ x\in C(\bfk')\; ; \;  h_{\mathrm{can}}(D,x)=h_{\mathrm{can}}(D,x_0)\}$ is finite for every number field $\bfk'$; thus,  
$\Gamma'(x_0)$ is a finite set. 
Since $C({\overline \bfk})$ is infinite, we can argue as in the proof of Corollary~\ref{cor:DMM}  to  deduce that $\Gamma'_{\vert C}$ is finite, as asserted. 
\end{proof}
 
\begin{lem}\label{lem:height_0_on_inv_curve}
Assume  $\Pic^0(X)=0$ and identify $\Pic(X; \R)$ with $\NS(X; \R)$. 
\begin{enumerate}[\em (1)]
\item If  $V_{\mathrm{can}}$ contains a class with positive self-intersection, then it contains $\Pi_\Gamma$.
\item If $V_{\mathrm{can}}$ contains $\Pi_\Gamma$, and if $C$ 
is an irreducible  rational  $\Gamma$-periodic curve, then, $h_{\mathrm{can}}(D,x)=0$ for every 
$D\in \Pi_\Gamma$ and  $x\in C(\overline{\bfk})$.
\end{enumerate}
\end{lem}

\begin{proof} 
If $V_{\mathrm{can}}$ contains a class in the positive cone it contains  the limit set $\Lim(\Gamma)$, hence also $\Pi_\Gamma$ (see \cite[\S 2.3]{stiffness}); this proves the first assertion. 
For the second one, 
pick a probability measure $\nu$ on $\Gamma$ with finite support, and assume that $P_\nu^*(D)=\alpha(\nu) D$
for some $D$ in $\Pi_\Gamma $ and some $\alpha(\nu)\in \R$. Then, $\sum_f\nu(f) h_{\mathrm{can}}(D,f(x))=\alpha(\nu) h_{\mathrm{can}}(D,x)$ 
by equivariance and linearity. On the other hand, ${\mathcal{O}}(D)_{\vert C}$ has degree $0$, because $\langle D\vert C\rangle =0$, and is therefore
trivial because $C$ is rational. So, $h_{\mathrm{can}}(D,\cdot)$ is bounded on $C(\overline{\bfk})$. Since $\alpha(\nu)>1$, this implies that 
$h_{\mathrm{can}}(D, x)=0$ for every $x\in C(\overline{\bfk})$. To conclude, note that such eigenvectors $D$ generate $\Pi_\Gamma$ when 
we vary $\nu$ (see \S \ref{par:rational_invariant_classes}).
\end{proof}

\subsection{From canonical vector heights to Kummer groups} 
\begin{mthm}\label{mthm:canonical_vector_height}
Let $X$ be a smooth projective surface  
and $\Gamma$ be a subgroup of $\Aut(X)$, both defined over a number field $\bfk$.
Suppose that 
\begin{enumerate}[{\em (i)}]
\item $\Gamma$ is non-elementary and contains parabolic elements;
\item there exists a canonical vector height $h_{\mathrm{can}}$ for 
$(X(\overline \bfk), \Gamma)$ on a $\Gamma$-invariant subspace of $\Pic(X; \R)$ which contains a divisor with positive self-intersection. 
\end{enumerate}
Then $(X,\Gamma)$ is a Kummer group. If in addition $h_{\mathrm{can}}$ is defined on $\Pic (X; \R)\times X({\overline{\bfk}})$, then 
$X$ is an abelian surface. 
\end{mthm}

The smoothness of $X$ is essential 
for the last conclusion to hold; for instance, if  $(X_0,\Gamma)$ is a singular
Kummer example with no $\Gamma$-invariant curve, we shall see that the N\'eron-Tate height induces a canonical vector height on  
$\Pic (X_0; \R)\times X_0({\overline{\bfk}})$. 

The remainder  of this subsection  is devoted to the proof of Theorem~\ref{mthm:canonical_vector_height}. 
Let us already observe that once  $(X,\Gamma)$ is known to be a Kummer group, the last conclusion 
 readily  follows from Lemmas~\ref{lem:invariant_curve_kummer} and~\ref{lem:no_inv_curve_canonical_height}. 
 So all we have to show is that $(X,\Gamma)$ is a Kummer group. 

\subsubsection{Reduction to $\Pic^0(X)=0$} Suppose $\Pic^0(X)\neq \set{0}$. Then,
$\Gamma$ being non-elementary, \cite[Theorem 10.1]{Cantat:Milnor} shows that
$X$ is either an abelian surface or a blowup of such a surface along a finite  orbit of $\Gamma$, and 
  by definition $(X, \Gamma)$ is   a Kummer group.
  
So, from now on, we assume  $\Pic^0(X)=\set{0}$ and identify 
$\Pic (X; \R)$ with $\NS(X; \R)$. 

\subsubsection{A key lemma} Assumption (ii) provides  a canonical vector height $ h_{\mathrm{can}}$ for  $(X, \Gamma)$ 
in restriction to  $\Pi_\Gamma$ (see Lemma~\ref{lem:height_0_on_inv_curve}). 
Recall from \cite{Silverman:1991, Kawaguchi:AJM} that for every $f\in \Gamma_{\mathrm{lox}}$ 
there exist canonical  heights $h_f^\pm $, respectively associated to the classes $\theta_f^\pm$,
such that $h_f^+(f(x)) = \lambda(f) h_f^+(x)$ and $h_f^-(f\inv(x)) = \lambda(f) h_f^-(x)$. 
They satisfy: 
\begin{itemize}
\item $h_f^\pm \geq 0$ on $X(\overline \bfk)$; 
\item if $D_f$ denotes the maximal invariant 
curve of $f$ then, for $x\in X(\overline \bfk)$, $h_f^+(x)+ h_f^-(x) = 0$ if  and only if $x$ is a periodic point 
or $x\in D_f$ (see  \cite[\S 5]{Kawaguchi:AJM}). 
\end{itemize}
Furthermore any Weil height $h$ associated to $\theta_f^+$ 
such that $h (f(x)) = \lambda(f) h (x)$ coincides with $h_f^+$: indeed 
$k := h - h_f^+$ is bounded because $h$ and $h_f^+$ are Weil heights associated to the same class, so 
the relation $k (f(x)) = \lambda(f) k (x)$ forces it to be identically zero. 
Thus, the next lemma follows immediately from the defining Properties (a), (b), and (c) 
(see \cite[Prop. 3.4]{Kawaguchi:2013}  or  \cite[\S1]{Baragar:2004}). 
 
\begin{lem} \label{lem:hcan}
If $\Pic^0(X)=\{0\}$ and if $h_{can}$ is a canonical vector height for $(X, \Gamma)$ in restriction to $\Pi_\Gamma$, then  $h_{\mathrm{can}}(\theta_f^\pm, \cdot) = h_f^\pm(\cdot)$ for every $f\in \Gamma_{lox}$. 
\end{lem}
    
\begin{rem}\label{rem:hcan_zero}
If $x$ belongs to the maximal invariant curve $D_\Gamma$ and $c$ belongs to $\Pi_\Gamma$ then 
$h_{\mathrm{can}}(c,x) = 0$. Indeed for  every $f\in \Gamma_{\mathrm{lox}}$, $D_\Gamma\subset D_f$ so  $h_{\mathrm{can}}(\theta_f^\pm, x)=h^+_f(x)=0$, and
since the classes $(\theta_f^\pm)_{f \in \Gamma_{\mathrm{lox}}}$ span  $\Pi_\Gamma$, the result follows by linearity. This extends the second assertion  of Lemma~\ref{lem:height_0_on_inv_curve} to invariant curves which are not rational. 
\end{rem}

 Now, recall that the classes $\theta_f^\pm$  are normalized by 
$\langle \theta_f^\pm\vert [\kappa_0]\rangle = 1$. Let us view  $\Hyp_X$ and  $\P(\Hyp_X)$ 
as subsets of   $\set{ u\in H^{1,1}(X;\R)\; ; \; \langle u   \vert u \rangle >0, \langle u  
\vert [\kappa_0]\rangle = 1}$. Setting  
\begin{equation}
\widetilde \Pi_\Gamma = \Pi_\Gamma\cap  \set{\langle \cdot \vert [\kappa_0]\rangle = 1},
\end{equation}   
$\Lim(\Gamma)$ can now be 
viewed  as a subset of $\widetilde \Pi_\Gamma$ which generates $\Pi_\Gamma$ as a vector space. 
The starting point of the proof of Theorem~\ref{mthm:canonical_vector_height} is the following key   lemma, inspired from the  approach of Kawaguchi in~\cite{Kawaguchi:2013}.

\begin{lem}\label{lem:convex}
In addition to the assumptions of Theorem \ref{mthm:canonical_vector_height}, suppose that 
\begin{enumerate}[{\em (i)}]
\item[\em (iii)] there exists $f\in \Gamma_{\mathrm{lox}}$  such that 
$ [\theta_f^+, \theta_f^-]\cap \mathrm{Int}(\mathrm{Conv}(\Lim(\Gamma)))\neq\emptyset$,
where $\mathrm{Conv}(\cdot)$ is the convex hull and $\mathrm{Int}(\cdot)$ stands for  the interior   relative to $\widetilde \Pi_\Gamma$.
\end{enumerate}
Then $(X, \Gamma)$ is a Kummer group.
\end{lem}
 
\begin{proof} 
Set $d=\dim \widetilde \Pi_\Gamma$. 
Replacing  $\bfk$ by a
finite extension, we may assume that the birational morphism $\pi_0\colon X\to X_0$ constructed in Proposition~\ref{pro:contraction}
is defined over~$\bfk$; this morphism contracts the maximal $\Gamma$-invariant curve $D_\Gamma$.

Let $\nu$ be a probability measure on $\Gamma$, whose support is finite and contains $f$ as 
well as elements of $\Gamma_{\mathrm{par}}$.
Let $w_\nu$ be the eigenvector of the operator $P_\nu$ for the eigenvalue $\alpha(\nu)$ given by 
Lemma \ref{lem:unique_w}.  
As in Proposition~\ref{pro:kawa_sequence_nun}, we may assume that $w_\nu$ is a rational class and is the pull-back of an ample class $[A_0]$ on $X_0$;
by muliplying  $w_\nu$ by a positive integer, we also assume that $w_\nu$ is an integral class.

Let $L$ be the line bundle given by the class $w_\nu$, and $\hat{h}_L$ be the associated 
 canonical stationary height, as in the proof of Theorem~\ref{thm:berman-boucksom_periodic}. This is the unique Weil height such that $\sum_h\nu(h) \hat{h}_L\circ h = \alpha(\nu) \hat{h}_L$.
By the  linearity of the canonical vector height and the 
uniqueness of $\hat{h}_L$, we get $\hat{h}_{L}(\cdot)=h_{\mathrm{can}}(w_\nu,\cdot)$.

Pick $w = a\theta_f^++b\theta_f^-$ in the interior of $\mathrm{Conv}(\Lim(\Gamma))$, 
with $a$, $b$ in $\R_+$ and $a+b=1$. Then by linearity and 
Lemma \ref{lem:hcan}, $h_{\mathrm{can}}(w, \cdot) = ah_f^++bh_f^-$. 
Caratheodory's theorem provides a subset $\Lambda$ of 
$\Lim(\Gamma)$ such that $\vert \Lambda\vert = d+1$ and $w$ belongs to  
 the interior of the simplex $\mathrm{Conv}(\Lambda)$.
By the density of fixed points of loxodromic elements   in $\Lim(\Gamma)$,  
we may assume that  $\Lambda$ is made of  classes $\theta^+_{g}$ for $g$ in a finite subset $\Lambda_\Gamma$ of $\Gamma_{\mathrm{lox}}$.  
 If $\e >0$ is small enough,   
 $w-\e  w_\nu$ stays in $\Int(\mathrm{Conv}(\Lambda))$;  so, there are positive 
coefficients $\beta_g $, for $g\in \Lambda_\Gamma$, such that $w-\e  w_\nu= \sum_{g\in \Lambda_\Gamma} {\beta_g}\theta^+_g$.
By the linearity of $h_{\mathrm{can}}$, we infer that 
\begin{equation}
a h_f^+(\cdot) + bh_f^- (\cdot) = \sum_{g\in \Lambda_\Gamma} {\beta_g} h_g^+(\cdot) +\e  \hat{h}_{L}(\cdot).
\end{equation}

Now, if $x\in X(\overline \bfk)$ is $f$-periodic, then $h_f^+(x)=h_f^-(x)=0$, 
and since $\hat{h}_{L}$ and  the  $h_g^+(x)$ 
are non-negative, we deduce that $\hat{h}_L(x)=0$. The line bundle
$L$ is the pull-back of an ample line bundle $A_0$ on $X_0$; thus, by
Lemma~\ref{lem:kawaguchi_northcott_periodic}, the  $\Gamma_\nu$-orbit of 
 $\pi_0(x)$ in $X_0$ is a finite set. Since $f$ has 
a Zariski dense set of periodic points,  Theorem~\ref{mthm:main} implies 
 that $(X,\Gamma_\nu)$ is a Kummer group. 

Since we can further  choose  $\Gamma_\nu$ to  contain any \emph{a priori} given finite subset of $\Gamma$, we conclude from  Theorem~\ref{thm:Kummer_for_groups}  
 that $(X,\Gamma)$ is a Kummer group.
\end{proof}

From this point, the proof of Theorem~\ref{mthm:canonical_vector_height} is completed in two steps. We first deal with the case $\dim \Pi_\Gamma\leq 4$
by directly checking the Assumption~(iii) of 
  Lemma~\ref{lem:convex}.  This covers general Wehler surfaces
(which is the setting of   \cite{Kawaguchi:2013}) since  $\dim \Pi_\Gamma=3$ in this case. The general 
case is treated in a second stage by a dimension reduction argument. 

\subsubsection{Conclusion when $\dim \Pi_\Gamma\leq 4$}\label{par:thmE_dim_4}
Since $\Gamma$ is non-elementary,  $\dim\Pi_\Gamma\geq 3$, so we need to consider the cases 
 $\dim\Pi_\Gamma=3$ and $\dim\Pi_\Gamma=4$ .

For $\dim\Pi_\Gamma=3$, i.e. $d=\dim(\widetilde{\Pi}_\Gamma )=2$,   
the intersection of $\widetilde{\Pi}_\Gamma$ with the positive cone is the Klein model of   the  hyperbolic disk $\Hyp^2$. 
If  Assumption~(iii) is not satisfied, then for every $f\in \Gamma_{\mathrm{lox}}$, $\Lim(\Gamma)$ is  entirely contained on one 
side of the geodesic $[\theta_f^+, \theta_f^-]$. 
Fix $4$ points in $\Lim(\Gamma)\subset \partial\Hyp^2\simeq {\mathbb{S}}^1$, labelled in circular order $(p_1, p_2, p_3, p_4)$. 
By Lemma~\ref{lem:density_limit_set} provides elements  $f$ and $g$ in $\Gamma_{\mathrm{lox}}$ such that 
$(\theta_f^+, \theta_g^+, \theta_f^-, \theta_g^-)$ is arbitrary close to 
$(p_1, p_2, p_3, p_4)$. Then  $[\theta_f^+, \theta_f^-]$ intersects $[\theta_g^+, \theta_g^-]$  transversally in the disk $\Hyp^2$, so $\Lim(\Gamma)$
intersects both sides of $[\theta_f^+, \theta_f^-]$, a contradiction.

Now assume $\dim\Pi_\Gamma=4$, i.e. $d=3$.  Then $\mathrm{Conv}({\Lim(\Gamma)})$ is 
a convex body in dimension $3$. The conclusion relies  on 
 the following lemma (see below for a proof). 

\begin{lem}\label{lem:polytope}
Let $p_1, \ldots, p_5$ be five points in general position in $\R^3$. Then there is a pair of 
 indices $ i\neq j$ such that any  plane containing 
 $p_i$ and $p_j$ separates the remaining points into two non-trivial sets.
\end{lem}

Now fix such a $5$-tuple of points in $\Lim(\Gamma)$ and approximate 
the given pair $(p_i,p_j)$ by $(\theta^+_f,\theta^-_f)$, for some $f\in \Gamma_{\mathrm{lox}}$. 
Then, $[\theta_f^+, \theta_f^-]$ intersects  the interior of $\mathrm{Conv}({\Lim(\Gamma)})$, and  Lemma \ref{lem:convex} finishes the proof of Theorem~\ref{mthm:canonical_vector_height}  (when $\dim(\Pi_\Gamma)\leq 4$).

\begin{proof}[Proof of Lemma \ref{lem:polytope}]
Express $p_5$ as a barycenter of the remaining points:  
\begin{equation}
p_5 = \mathrm{Bar}((p_1; \alpha_1), \ldots , (p_4; \alpha_4)).
\end{equation} 
The general position assumption means that the $\alpha_i$ are non-zero. 
If all the coefficients $\alpha_i$ are positive, $p_5$ 
lies  in the interior of the tetrahedron ${\mathrm{Conv}}\{p_1, \ldots, p_4\}$, 
so any plane containing  $p_4$ and $p_5$ separates the vertices of 
 the triangle ${\mathrm{Conv}}\{p_1, p_2, p_3\}$
into two non-trivial parts. Therefore, the pair of indices $(i,j)=(4,5)$ works in this case.
If exactly  one of the coefficients is negative, say  $\alpha_\ell$, 
 then  consider the pair $(\ell, 5)$, and denote by $i$, $j$, $k$ the remaining indices (i.e. $\{i,j,k\}=\{1,2,3,4\}\setminus\{ \ell\}$).
The segment $[p_\ell,p_5]$ cuts the triangle ${\mathrm{Conv}}\{p_i,p_j,p_k\}$ in a point $h$ of its relative interior. Thus, any plane containing $p_\ell$ and $p_5$ 
also contains $h$, and separates $\{p_i,p_j,p_k\}$ non-trivially.
The only remaining case  is when 
there are  three positive and two negative coefficients, say $\alpha_3$ and $\alpha_4$; then, $p_1$ is a barycenter of the remaining points with $3$ positive and 
$1$ negative coefficients, so we are done in this case too.\end{proof}

\begin{rem} The theorem of Steinitz (see \cite[\S 13.1]{grunbaum}) is a far-reaching generalization 
 of Lemma~\ref{lem:polytope}. There is no analogue of this lemma in higher dimension  
  (see \cite[\S 4.7]{grunbaum}), hence the need for a different argument 
when $\dim \Pi_\Gamma\geq 5$. 
\end{rem}

\subsubsection{Conclusion of the proof of Theorem \ref{mthm:canonical_vector_height}}
Recall from  Lemma~\ref{lem:parabolic_unipotent} that   $g^*$  is virtually unipotent for every $g\in \Gamma_{\mathrm{par}}$.   
 
 \begin{lem}\label{lem:2_unipotent_parabolic}
 Let $g_1$ and $g_2$ be unipotent parabolic elements in $\Aut(X)$ with distinct invariant fibrations. Then,  
 $\Gamma_0 := \langle g_1, g_2\rangle$ is non-elementary and $\dim(\Pi_{\Gamma_0})\leq 4.$  
 \end{lem}

\begin{proof} 
Since $\pi_{g_1}\neq \pi_{g_2}$, 
$\Gamma_0$ is non-elementary (see \S \ref{subsub:non_elementary}).
The subspace  $W:=\mathrm{Fix}(g_1^*)\cap \mathrm{Fix}(g_2^*)$ of $\NS(X;\R)$  is fixed 
 pointwise   by $\Gamma_0$. 
 Thus  $W^\perp$ is $\Gamma_0$-invariant, 
it contains $\Pi_{\Gamma_0}$ (see \cite[Prop. 2.8]{stiffness}), and all we need to show is that 
 $\dim(W^\perp)\leq 4$.
To see this, note that  a unipotent Euclidean isometry is the identity, thus 
if $g\in\O^+(1,d)$ is parabolic and unipotent,   the structure of parabolic isometries of $\mathbb{H}_d$
(see \cite[\S I.5]{franchi-lejan}) 
implies that  $\mathrm{Fix}(g^*)\subset \R^{d+1}$ is a subspace of codimension $2$, and we are done. 
 \end{proof}
 
From Lemma~\ref{lem:2_unipotent_parabolic} and Section~\ref{par:thmE_dim_4}, we deduce that $(X, \langle g_1, g_2\rangle)$ is a Kummer 
group for every pair of   unipotent elements 
$g_1, g_2\in \Gamma_{\mathrm{par}}$ generating a non-elementary subgroup. 
Thus by Theorem \ref{thm:Kummer_for_groups},  $(X, \Gamma)$ itself is a Kummer group, 
and Theorem~\ref{mthm:canonical_vector_height} is established. 
 
\subsection{Canonical vector heights on abelian surfaces }\label{par:complement:NT} 
In this section, $A$ is an abelian surface, defined over some number field $\bfk$ and $\Gamma\subset \Aut(A_\bfk)$ is non-elementary.
Denote by $h_{\mathrm{NT}}\colon \Pic(A)\times A(\overline{\bfk})\to \R$ the Néron-Tate height on $A$; it vanishes identically on the torsion part
of $\Pic(A)$, so we may also consider it as a function on  $\Pic(A;\R)\times A(\overline{\bfk})$. 
When $0\in A$ has a finite $\Gamma$-orbit,  $h_{\mathrm{NT}}$ is a canonical vector height (see \cite[Thm. B.5.6]{Hindry-Silverman}).

Let $h_{\mathrm{can}}$ be a restricted canonical vector height for $(A_\bfk,\Gamma)$, defined on some $\Gamma$-invariant subspace $V_{\mathrm{can}}$ 
of $\Pic(A;\R)$. Our goal is to compare it to $h_{\mathrm{NT}}$.

By definition, a divisor $D$ on   $A$ is symmetric if $[-1]^*D$ is linearly equivalent to $D$, where $[m]$ denotes multiplication by $m$; likewise 
it is antisymmetric if $[-1]^*D\simeq -D$ or equivalently if $D\in \Pic^0(A)$ (see \cite[Prop. A.7.3.2]{Hindry-Silverman}). If $f\in \Aut(A)$ fixes 
the origin, it commutes to $[-1]$, so that $f^*$ preserves symmetry and antisymmetry. 

\begin{rem}\label{rem:NT}
Any class $[D]\in \NS(A)$ can be lifted to a symmetric divisor class $D\in \Pic(A)$, which is unique up to a 
2-torsion element in $\Pic^0(A)$. Thus, $D$ admits a unique symmetric lift in $\Pic(A;\R)$. By using such a lift it makes sense to consider also
$h_{\mathrm{NT}}(\cdot, \cdot)$ (resp. $h_{\mathrm{can}}(\cdot, \cdot)$) as  a function on $\NS(A;\R)\times A(\overline{\bfk})$ (resp. on the projection 
of $V_{\mathrm{can}}$ in $\NS(A;\R)$). This observation will be used repeatedly in the following.  
\end{rem}

\begin{rem}\label{rem:picard_abelian}
The Picard number of any complex abelian surface satisfies $\rho(A)\in\set{ 1, 2, 3, 4}$. 
When $\Aut(A)$ contains a non-elementary group $\Gamma$, we obtain  
$3 \leq \dim \Pi_\Gamma \leq \rho(A)   \leq 4$. 
Moreover, $\rho(A)=4$  if and only if $A$ is isogenous to $B\times B$, 
for some elliptic curve $B$ with complex multiplication (see  \cite[ex.10 p.142]{Birkenhake-Lange}). 
\end{rem}

\begin{pro}\label{pro:NT1}
If  $V_{\mathrm{can}}$ contains $\Pic^0(A)\otimes_\Z\R$, then $\Gamma$ has a finite orbit in $A(\overline{\bfk})$.
\end{pro}

\begin{proof} In this proof it is enough to  consider $h_{\mathrm{can}}$ as a function on $\Pic^0(A)\times X(\overline{\bfk})$, by composing 
with the natural homomorphism $\Pic^0(A)\to \Pic^0(A;\R)$. 

\smallskip

{\bf{Step 1.}}-- {\emph{If $D$ is an element of $\Pic^0 (A)$, then  for every $f\in \Gamma_{\mathrm{lox}}$, every periodic point
$x$ of $f$ satisfies 
$h_{\mathrm{can}}(D,x)=0$.}}

Assume $f^q(x)=x$ for some $q\geq 1$.
The endomorphism $f^q-\id$ is an isogeny of $A$ because $f$ is loxodromic (see Section~\ref{par:lox_on_tori}). Thus, 
its dual $(f^q)^*-{\id}$ is an isogeny of $\Pic^0(A)$ and we can find $E\in \Pic^0(A)$ such that $(f^q)^*E- E=D$.
By equivariance $h_{\mathrm{can}}((f^q)^*E,x)=h_{\mathrm{can}}(E,x)$, and then by linearity $h_{\mathrm{can}}(D,x)=0$. 
 
\smallskip

{\bf{Step 2.}}-- {\emph{Let $\bfk'$ be a finite extension of $\bfk$, and let $P$ be a subset of $A(\bfk')$. If, for every  $D\in \Pic^0(A)$, the set 
$\set{h_{\mathrm{NT}}(D,x)\; ; \; x\in P}\subset \R$   is bounded,  then $P$ is finite.}}

To see this, consider the abelian group $A(\bfk')$; by the Mordell-Weil theorem, its rank is finite, so modulo torsion it is isomorphic to $\Z^r$ for some $r\geq0$. 
Set $W_{\bfk'}=A(\bfk')\otimes_\Z\R$, a real vector space of dimension $r$.
Let $H$ be an  ample symmetric divisor on $A_{\overline{\bfk}}$, then $h_{\mathrm{NT}}(H,\cdot)$ determines a positive definite quadratic 
form on $V$; let $\langle \cdot \vert \cdot\rangle_H$ be the bilinear pairing associated to $h_{\mathrm{NT}}(H, \cdot)$. 
If $s$ is an element of $A(\overline{\bfk})$, and $t_s\in \Aut(A_{\overline{\bfk}})$ is the translation by $s$, then $D_s:=H-t_s^*H$ is 
an element of $\Pic^0(A)$ and $h_{\mathrm{NT}}(D_s, \cdot)$ induces an affine linear form $A(\bfk')\to \R$; namely,  $h_{\mathrm{NT}}(D_s, \cdot)=-2\langle s \vert \cdot\rangle_H$. 
Since $\langle \cdot \vert \cdot\rangle_H$ is positive definite on $W_{\bfk'}$ (see \cite[Prop. B.5.3]{Hindry-Silverman}), 
one can find $r$ elements $s_i\in A(\bfk')$ such that the linear
forms $\ell_i:=\langle s_i \vert \cdot\rangle_H$ constitute a basis of the dual of $W_{\bfk'}$. 
%(see the proof of \cite[Thm. B.5.8]{Hindry-Silverman} for a similar argument).
Our assumption  says that each $\ell_i(P)$ is a relatively compact subset of $\R$; this
implies that $P$ is contained in a compact, hence finite, subset of the
lattice $A(\bfk')\subset V$.

\smallskip

{\bf{Step 3.}}-- {\emph{$\Gamma$ has a finite orbit.}}

Let $f$ be a loxodromic element of $\Gamma$, and $x$ be a fixed point of $f$. Its $\Gamma$-orbit is made of fixed points of conjugates
of $f$. Note that  $\Gamma(x)$ is contained in  $A(\bfk')$ for some finite extension of $\bfk$.
By the first step, $h_{\mathrm{can}}$ vanishes on $\Pic^0(A)\times \Gamma (x)$. Since $h_{\mathrm{can}}$ and $h_{\mathrm{NT}}$ are Weil 
heights, $\vert h_{\mathrm{can}}(D,\cdot) -h_{\mathrm{NT}}(D, \cdot) \vert \leq B(D)$ for each divisor class $D\in \Pic^0(A)$, where $B(D)\geq 0$ depends
on $D$. Thus, $ \vert h_{\mathrm{NT}}(D, \Gamma ( x)) \vert \leq B(D)$ for every $D\in \Pic^0(A)$, and the second step 
implies that $\Gamma (x)$ is finite. 
\end{proof}

\begin{pro}\label{pro:NT2}  
Assume that the neutral element has a finite $\Gamma$-orbit. 
Then $h_{\mathrm{can}}$ coincides with the Néron-Tate height on
\begin{itemize}
\item  the set of symmetric divisors whose numerical class belongs to $\Pi_\Gamma$,
\item  the set of antisymmetric divisors,
\end{itemize}
whenever one of these sets is contained in $V_{\mathrm{can}}$.
\end{pro}

In the following proofs, we denote by $\Pi_\Gamma^{\mathrm{s}}$ the subspace $\Pic(A;\R)$ made of symmetric 
elements $E\in \Pic(A;\R)$ such that $[E]\in \Pi_\Gamma$.

\begin{proof} Let $\Gamma_0\leq \Gamma$ be the finite index subgroup fixing the origin. 
 Let us show that
$ h_{\mathrm{can}} = h_{\mathrm{NT}} $ on $\Pi_\Gamma^{\mathrm{s}}\times A(\overline{\bfk})$.
For this, we use Remark~\ref{rem:NT}, identify $\Pi_\Gamma^{\mathrm{s}}$ with $\Pi_\Gamma$, 
and consider $h_{\mathrm{can}}$  and $h_{\mathrm{NT}}$ as functions on $\Pi_\Gamma\times A(\overline{\bfk})$.
Now, if $f\in \Gamma_{0,{\mathrm{lox}}}$, we get $h_{\mathrm{NT}}(\theta^+_f,\cdot)=h_{\mathrm{can}}(\theta^+_f,\cdot)$
because the difference is bounded on $A(\overline{\bfk})$, and is multiplied by $\lambda(f)>1$ under the action of $f$ (as in Lemma~\ref{lem:hcan}).
Since the classes  $\theta^+_f$, for $f\in \Gamma_{\mathrm{lox}}$, generate $\Pi_\Gamma$,  
our claim is established.  

 Let us now deal with antisymmetric divisors. 
Identifying $\Pic^0(A_{\overline \bfk})$ with the dual abelian variety $A^\vee_{\overline \bfk}$ 
of $A$, we have to show that $h_{\mathrm{can}}$ coincides with $h_{\mathrm{NT}}$ on 
$A^\vee(\bfk')$ for every finite extension $\bfk'$  of $\bfk$. By the Mordell-Weil theorem 
$A^\vee(\bfk')$ is a finitely generated abelian group so 
\begin{equation}
W^\vee_{\bfk'}:=A^\vee(\bfk')\otimes_\Z\R
\end{equation}
is a real vector space of dimension $r$, for some $r<+\infty$.
Consider the function $\Phi:(D, x) \mapsto h_{\mathrm{can}}(D,x)-h_{\mathrm{NT}}(D,x)$. 
When $D$ is fixed, $\Phi(D, \cdot)$ is bounded:   $\abs{\Phi(D, x)}\leq B(D)$ for all $x\in A(\overline \bfk)$. 
On the other hand when $x$ is fixed, $\Phi_x(D)  :=   \Phi (D, x)$ defines a linear form $\Phi_x\colon W^\vee_{\bfk'}\to \R$. 
Applying  the previous boundedness property to $f(x)$, for  $f$ ranging 
in $\Gamma_0$, and using the equivariance  $\Phi (D, f(x)) = \Phi(f^*D, x)$
we obtain that for every $x\in A(\overline \bfk)$, $\Phi_x$ is bounded on every $\Gamma_0^*$-orbit  $\Gamma_0^*(D)\subset W^\vee_{\bfk'}$. 

We claim that this forces $\Phi_x$ to vanish, which is the desired result. For this we analyze the dual action of 
$\Gamma_0$. Let $f$ be a loxodromic element of $\Gamma_0$, and $f^\vee_{{\bfk'}}$ be the induced 
linear map on $W^\vee_{\bfk'}=A^\vee(\bfk')\otimes_\Z\R$. Let $L_f$ be the linear lift of $f$ to $\C^2$, as in 
\S \ref{par:abelian} and~\ref{par:lox_on_tori}.

\begin{lem}\label{lem:semi-simple}
The endomorphism $f^\vee_{{\bfk'}}$ is semi-simple and its complex eigenvalues are complex conjugate to those of $L_f$; none of them has modulus $1$.
\end{lem}

Let us take this for granted and conclude the proof. Since $f^\vee_{{\bfk'}}$ is semi-simple, $W^\vee_{\bfk'}$ is a direct sum of 
$f^\vee_{{\bfk'}}$-invariant irreducible factors $\bigoplus W^\vee_i$, each of dimension $1$ or~$2$. 
For each $W^\vee_i$, denote by $\lambda_i$ the corresponding eigenvalue of $f^\vee_{{\bfk'}}$, and pick some $D_i\in W^\vee_i\setminus\{0\}$.
If $W^\vee_i $  is a line, then $\lambda_i\in \R^*$ and $\abs{\lambda_i}\neq 1$. 
Since $\Phi_x$ is bounded on $\set{(f^\vee_{{\bfk'}})^n(D_i)\; ; \; n\in \Z}$, the line $W^\vee_i$ is contained in $\ker \Phi_x$. 
If $W^\vee_i$ is a plane, 
then $f^\vee_{{\bfk'}}\rest{W^\vee_i}$ is a similitude with  $\abs{\lambda_i}\neq 1$ and   ${\mathrm{Arg}}(\lambda_i)\neq 0\mod (2\pi \Z)$. 
If $\Phi_{x}\rest{W^\vee_i}\neq 0$,  $\set{\abs{\Phi_x} \leq B(D_i)}\cap W^\vee_i$ is a strip, 
which furthermore  contains
the orbit  $\set{(f^\vee_{{\bfk'}})^n(D_i), n\in \Z}$. This is not compatible  with
 the properties of $\lambda_i$, and this contradiction shows that 
 $W^\vee_i \subset \ker \Phi_x$, so finally $\Phi_x = 0$, as claimed. 
 \end{proof}

\begin{proof}[Proof of Lemma \ref{lem:semi-simple}]
The complex torus underlying $A^\vee_\C$ is isomorphic to a quotient of the space of 
$\C$-antilinear forms on $\C^2$. So,  if $f\in \Aut(A) $ is induced by a linear map 
$L_f\in \GL_2(\C)$, the automorphism of  $A^\vee $ determined by $f^*$ 
is induced by the conjugate transpose $\overline L_f^t$ (see \cite[\S 2.4]{Birkenhake-Lange}). When $f$ is loxodromic, the eigenvalues of  $L_f$ satisfy $\abs{\alpha}<1<\abs{\beta}$;
we deduce that the  automorphism of $ \Aut(A^\vee_\C)$ determined by $f^*$ is also loxodromic, with eigenvalues $\overline{\alpha}$ and $\overline{\beta}$, and the minimal polynomial of $\overline L_f^t$ is  $(X-\overline{\alpha})(X-\overline{\beta})$. Let $P$ 
be the minimal, real, unitary polynomial such that $(X-\overline{\alpha})(X-\overline{\beta})$  divides  $P$ (by construction $\deg(P)\in \set{2, 3,4}$ 
and $P$ has no repeated factors). Since $P( \overline L_f^t) = 0$,
we infer that $P(f^\vee_{\bfk'}) =0$ and the result follows.\end{proof}

\begin{pro}\label{pro:heights_on_tori}
Let $A_\bfk$ be an abelian surface defined over a number field $\bfk$. 
Let $\Gamma$ be a non-elementary subgroup of $\Aut(A_\bfk)$, for which the neutral element $0\in A(\bfk)$ is periodic. Then one of the following situation occurs:
\begin{enumerate}[\em (1)]
\item    $\NS(A, \R) = \Pi_\Gamma$ and  the N\'eron-Tate height is the unique canonical vector height on $\Pic(A;\R)$. 
\item  $\NS(A, \R) = \Pi_\Gamma  \overset{\perp}\oplus  \R [E]$ for some  $[E]\in \NS(A;\R)\setminus\{0\}$,
and the  canonical vector heights on $\Pic(A;\R)$ are exactly the functions  of the form  $h_{\mathrm{can}} (D, x) = h_{\mathrm{NT}} (D, x) +  \langle [E] \vert D\rangle \varphi(x) $,
where $\varphi\colon A({\overline \bfk})\to \R$ is any  bounded  function such that 
$\varphi(f(x)) f^*[E]=\varphi(x)[E]$ for all $f$ in $\Gamma$. 
\end{enumerate}
 \end{pro}
 
 \begin{proof} 
When $\NS(A, \R) = \Pi_\Gamma$, Proposition~\ref{pro:NT2} and the decomposition of any divisor
class as a sum $D=D^{\mathrm{s}}+D^{\mathrm{a}}$ with $D^{\mathrm{s}}$ symmetric and $D^{\mathrm{a}}$ antisymmetric 
imply that   $h_{\mathrm{can}}=h_{\mathrm{NT}}$.  So by Remark~\ref{rem:picard_abelian} we may   assume that 
$\rho(A)=4$ and $\dim(\Pi_\Gamma)=3$. Pick  $[E]\in \Pi_\Gamma^\perp\setminus \{ 0\}$. The line $\R[E]$ is $\Gamma$-invariant, 
 and the intersection form is negative on $\R[E]$; as a consequence, there is a homomorphism $\alpha_{[E]}\colon \Gamma\to \{+1, -1\}$ such that
$f^*[E]=\alpha(f)[E]$ for all $f\in \Gamma$. Then for  fixed $x$, 
\begin{equation}
\Delta_x(D)=h_{\mathrm{can}}(D,x)-h_{\mathrm{NT}}(D,x),
\end{equation}
defines  a linear form on $\Pic(A;\R)$, which by Proposition~\ref{pro:NT2} vanishes identically on $\Pi_\Gamma^{\mathrm{s}}$. 
So, $\Delta(D,x)=\langle [E]\vert D\rangle \varphi(x) $ for some real valued function $\varphi$, and the conclusion follows from Lemma~\ref{lem:perturbation_canonical_vector_height}.
\end{proof}

%%%%%%%%%%%%%%%%%%%%%%%%%
\subsection{Synthesis}\label{subs:synthesis}
%%%%%%%%%%%%%%%%%%%%%%%%%

\subsubsection{Canonical vector heights} Putting together  Theorem~\ref{mthm:canonical_vector_height} and Proposition~\ref{pro:heights_on_tori} gives:

\begin{mthmprime} \label{mthm:prime}
Let $X$ be a smooth projective surface, defined over a number field $\bfk$. 
Let $\Gamma$ be a non-elementary subgroup of $\Aut(X_\bfk)$ that contains parabolic elements. 
Let $h_{\mathrm{can}}$ be a canonical vector height on $\Pic(X;\R) \times X({\overline{\bfk}})$ 
for the group $\Gamma$. Then, $X$ is an abelian surface and $h_{\mathrm{can}}$ is derived
from a translate of the Néron-Tate height by a periodic point $y$ of $\Gamma$:
\[
h_{\mathrm{can}}(D,x)=h_{\mathrm{NT}}(D,x+y)+\langle [E]\vert D\rangle \, \varphi(x)
\]
for  some almost-invariant class $[E]\in \NS(X;\R)$ and some bounded function $\varphi\colon X(\overline{\bfk})\to \R$ 
such that $\varphi(f(x)) f^*[E]=\varphi(x)[E]$ for  $f\in \Gamma$.
\end{mthmprime}

Note that $h_{\mathrm{can}}$ is just a translate of $h_{\mathrm{NT}}$ when $E$ is numerically trivial. 

\subsubsection{Restricted canonical vector heights} Let us add the assumption \begin{equation}
\Pic^0(X)=0,
\end{equation}
to the hypotheses of Theorem~\ref{mthm:canonical_vector_height}. Our goal is to describe all possibilities for $(V_{\mathrm{can}}, h_{\mathrm{can}})$.

Since $\Pic^0(X)=0$, $X$ is not a blow-up of an abelian surface and 
Theorem~\ref{mthm:canonical_vector_height} implies that 
  $(X,\Gamma)$ is a 
Kummer group of type (2), (3), (4), or (5) in the nomenclature of \S \ref{par:Kummer_classification}.  
We make use of the notation of
\S \ref{par:Kummer_pairs} and~\ref{par:Kummer_classification}. The origin $0\in A$ is a fixed point of the cyclic group $G$, 
and the orbit $\Gamma_A(0)$ is finite.
Since $G$  is generated by a finite order homothety $(x,y)\mapsto (\alpha x, \alpha y)$ on $A$, $G$ acts trivially on $\NS(A;\R)$ and on 
symmetric divisors. Thus, $\NS(A/G;\R)$ can be identified to $\NS(A;\R)$ and to the subspace of $\Pic(A;\R)$ generated by symmetric divisors; let 
\begin{equation}
\iota\colon \NS(A;\R) \to \NS(X;\R)
\end{equation}
denote the corresponding embedding, given by $\iota=q_X^*(q_A)_*$. On the space of symmetric divisors, the Néron-Tate height is $G$-invariant and $\Gamma$-equivariant,
so it induces a canonical vector height $h_{\mathrm{NT}}^{A/G}(\cdot, \cdot)$ on $A/G$ for ${\overline{\Gamma}}_A$. Then, it induces a restricted canonical vector height on $\iota(\NS(A;\R))\times X(\overline{\bfk})$,
namely
\begin{equation}
h_{\mathrm{NT}}^X: (D,x)\longmapsto  h_{\mathrm{NT}}^{A/G} ((q_X)_*D, q_X(x)).
\end{equation}

In what follows, we denote by $E_i$ the   disjoint irreducible rational curves contracted by $q_X$ (see Lemma~\ref{lem:invariant_curve_kummer}); their classes generate $\iota(\NS(A;\R))^\perp\subset \NS(X;\R)$. 
The height $h_{\mathrm{NT}}^X$ vanishes on $\bigcup_i E_i(\overline{\bfk})$, because the $E_i$ are mapped to torsion points of $A$.

\begin{lem}\label{lem:inclusion}
We have $\Pi_\Gamma=\iota(\Pi_{\Gamma_A})\subset V_{\mathrm{can}} \subset \iota(\NS(A;\R))$.
\end{lem}

\begin{proof}
The first equality comes from the equivariance of $q_X$ and $q_A$. The first inclusion follows from Lemma~\ref{lem:height_0_on_inv_curve} and the assumption~(ii) of Theorem~\ref{mthm:canonical_vector_height}. It remains to prove the last inclusion.  If $V_{\mathrm{can}}$ is not contained in $q_X^*(\NS(A/G;\R))$, there is an index $i$, and 
a class $D$ in $\Pi_\Gamma^\perp\cap V_{\mathrm{can}}$ such that $\langle D\vert E_i\rangle> 0$, i.e. 
$\mathcal{O}(D)\rest{E_i}$ is ample. The  action of $\Gamma$ on $\Pi_\Gamma^\perp$ factorizes through
a finite group (see~\cite[Lem. 2.9]{stiffness}), so $D$ is $\Gamma$-periodic and  by Lemma~\ref{lem:no_inv_curve_canonical_height}, the 
group $\Gamma\rest{E_i}$ is finite; this contradicts Lemma~\ref{lem:invariant_curve_kummer}, and the conclusion follows.
\end{proof}

Let $D$ be an element of $\Pi_\Gamma$.
By Lemma~\ref{lem:height_0_on_inv_curve}, $h_{\mathrm{can}}(D,x)=0$ for all  $x\in \bigcup_i E_i(\overline{\bfk})$. Thus, 
$(D,x)\mapsto h_{\mathrm{can}}(\iota(D), q_X^{-1}(q_A(x)))$ is a well defined 
restricted  canonical vector height on $\Pi_{\Gamma_A}\times A(\overline{\bfk})$ (see Remark~\ref{rem:NT}), 
which gives height $0$ to the fixed points of elements of $G\setminus\{\id\}$.
By Proposition~\ref{pro:NT2}, this height coincides with the Néron-Tate height on $\Pi_{\Gamma_A}\times A(\overline{\bfk})$. 

This yields a complete description of $h_{\mathrm{can}}$ when $V_{\mathrm{can}}=\Pi_\Gamma$.

By Lemma~\ref{lem:inclusion} and Remark~\ref{rem:picard_abelian}, the remaining possibility is that    $\dim(V_{\mathrm{can}})=4$ and  $\dim(\Pi_\Gamma)=3$. 
Choose an almost $\Gamma_A$-invariant class $[E]$ in $\NS(A;\R)$, as in Proposition~\ref{pro:heights_on_tori}, and a divisor $F$ in $X$ such that $[F]=\iota([E])$. 
Each element $D\in V_{\mathrm{can}}$ decomposes as a sum
\begin{equation}
D= D' + \frac{\langle [F] \vert D\rangle}{\langle [F] \vert [F]\rangle} [F]
\end{equation}
with $D'$ in $\Pi_\Gamma$. Then, for $x$ in $\bigcup_i E_i(\overline \bfk)$, we get 
\begin{equation}
h_{\mathrm{can}}(D,x)=\frac{\langle [F] \vert D\rangle}{\langle [F] \vert [F]\rangle}  h_{\mathrm{can}}([F],x).
\end{equation}
Define a  function by setting $\psi(x)=\langle [F] \vert [F]\rangle^{-1}h_{\mathrm{can}}([F],x)$ on $\bigcup_i E_i(\overline{\bfk})$ and $\psi(x) = 0$ otherwise.
It satisfies the equivariance $\psi(f(x))f^*[F]=\psi(x)[F]$  because $h_{\mathrm{can}}$ is equivariant and 
$[F]$ is almost invariant, and it is bounded   because ${\mathcal{O}}(F)_{\vert E_i}$ is trivial for each $E_i$. 
Now, if we set 
\begin{equation}
h_{\mathrm{can}}'(D,x)=h_{\mathrm{can}}(D,x)- \langle [F] \vert D\rangle\, \psi(x)
\end{equation}
we get a new restricted canonical vector height on $V_{\mathrm{can}}\times X(\overline{\bfk})$ that vanishes on  $\bigcup_i E_i(\overline{\bfk})$.
This height comes from a canonical vector height on $A/G$, and 
since as seen before $\NS(A/G; \R)$ can be identified to $\NS(A; \R)$, it yields a 
  canonical vector height for $(A,\Gamma_A)$ restricted to the  
space of symmetric divisors. The second assertion of  Proposition~\ref{pro:heights_on_tori} entails 
that this last height is derived from the Néron-Tate height for some 
function $\varphi$; 
since $\Gamma_A$ contains $G$, and $G$ fixes $[E]$, $\varphi$ is   $G$-invariant. Coming back to $X$, we get 
that $h_{\mathrm{can}}$ is derived from the Néron-Tate height too. In formulas, 
\begin{equation}
h_{\mathrm{can}}(D,x)=h_{\mathrm{NT}}^X(D,x)+  \langle [F] \vert D\rangle \, \Phi(x)
\end{equation}
where $\Phi\colon X({\overline{\bfk}})\to \R$ is a bounded function which satisfies $\Phi(f(x))f^*[F]=\Phi(x)[F]$ for $f\in \Gamma$. This function   is equal to $\psi $ on $\bigcup_i E_i(\overline{\bfk})$ and 
to $\varphi\circ q_X $ on its complement.

To conclude, using the above notation, let us summarize these results in a (somewhat imprecise) statement.

\begin{mthmsecond} \label{mthm:second}
Let $X$ be a smooth projective surface, defined over a number field $\bfk$, and such that  $\Pic^0(X_{\overline{\bfk}})=0$. 
Let $\Gamma$ be a non-elementary subgroup of $\Aut(X_\bfk)$ containing parabolic elements. Let 
$h_{\mathrm{can}}$ be a restricted canonical vector 
height on $V_{\mathrm{can}} \times X({\overline{\bfk}})$ for the group $\Gamma$, where 
$V_{\mathrm{can}}\subset \Pic (X; \R)$
is $\Gamma$-invariant and  contains classes with positive self-intersection. 
Then $(X,\Gamma)$ is a Kummer group associated to an abelian surface $A$, 
$V_{\mathrm{can}}$  is contained in $\iota(\NS(A; \R))$ and 
\begin{itemize}
\item either $V_{\mathrm{can}} = \Pi_\Gamma$ and $h_{\mathrm{can}}$ coincides 
with the Néron-Tate height $h_{\mathrm{NT}}^X$;
\item or $\Pi_\Gamma$ is a codimension 1 subspace of $V_{\mathrm{can}}$ and 
   $h_{\mathrm{can}}$ is derived from  $h_{\mathrm{NT}}^X$.
\end{itemize}
\end{mthmsecond}
 
\appendix
\section{ Kummer surfaces of type (6)}\label{par:Kummer6}

{\small{
In case (6), we get a cyclic quotient singularity of type $\frac{1}{5}(1,2)$, resolved by a string of two rational curves of respective 
self-intersections $-3$ and $-2$ (see \cite[\S III.6]{BHPVDV} and \cite[\S 2]{Laufer}). The ring of invariant functions 
for the group $G$ is generated by the four 
monomials $u_0=x^5$, $u_1=x^3y$, $u_2=xy^2$, $u_3=y^5$, and the quotient surface is locally isomorphic to the surface $W$ in $\C^4$ determined by the equations $u_0u_2=u_1^2$, $u_1u_3=u_2^3$. Now, take three copies $V_0$, $V_1$, $V_2$ of $\C^2$, with coordinates $(v_0,w_0)$, $(v_1,w_1)$,
and $(v_2,w_2)$ respectively, and glue together the open sets $V_0\setminus\{w_0=0\}$ and $V_1 \setminus\{v_1=0\}$ by $v_1=1/w_0$, $w_1=v_0w_0^3$, and the open sets $V_1\setminus\{w_1=0\}$   by $v_2=1/w_1$, $w_2=v_1w_1^2$. The result is a smooth surface 
$Y$, which is a neighborhood of a string of 
two smooth rational curves: the curve $E_1$ corresponding to the axis $\{(0,w_0)\; ; \; w_0\in \C\}$ (glued to $\{(v_1,0)\; ; \; v_1\in \C\}$) and the 
curve $E_2$ corresponding to the axis $\{(0,w_1)\; ; \; w_1\in \C\}$ (glued to $\{(v_2,0)\; ; \; v_2\in \C\}$); the self-intersections of these curves 
are respectively equal to $-3$ and $-2$. There is a $G$-invariant rational map from $\C^2$ to $Y$, given in coordinates $(v_i,w_i)$ by 
\begin{equation}
v_0=x^5, w_0=y/x^2; \; v_1=x^2/y, w_1= y^3/x; \; v_2=x/y^3, w_2=y^5.
\end{equation}
The surface $Y$ is a desingularization of $W=\C^2/G$, with projection $Y\to W$ given by 
\begin{eqnarray}
u_0 & = & v_0=v_1^3w_1=v_2^5w_2^3, \\
u_1 & = & v_0w_0=v_1^2w_1=v_2^3w_2^2, \\
u_2 & = & v_0w_0^2=v_1w_1=v_2w_2, \\
u_3 & = & v_0^2w_0^5  = v_1w_1^2 =  w_2.
\end{eqnarray}
If $F$ is a linear automorphism of $\C^2$ that normalizes $G$ and $F$ induces a loxodromic automorphism of the torus $\C^2/\Lambda_5$, 
then $F(x,y)=(\alpha x, \beta y)$ for some eigenvalues $\alpha$, $\beta$ with $\vert \alpha\vert > 1 > \vert \beta\vert$. On the quotient space $W$ 
it acts by $(u_0, u_1, u_2, u_3)\mapsto (\alpha^5 u_0, \alpha^3\beta u_1, \alpha \beta^2 u_2, \beta^5 u_3)$, and on $Y$ it acts locally by 
$(v_0,w_0)\mapsto (\alpha^5 v_0, \alpha^{-2}\beta w_0)$ (resp. $(\alpha^2\beta^{-1} v_1, \alpha^{-1}\beta^3 w_1)$ and $(\alpha\beta^{-3} v_2, \beta^5w_2)$). In particular, the linear projective map induced by $F$on  $E_1$ (resp. $E_2$) is given by  $w_0\mapsto \alpha^{-2}\beta w_0$
(resp. $w_1\mapsto \alpha^{-1}\beta^3w_1$). Since $\vert \alpha^{-2}\beta\vert< 1$
and $\vert\alpha^{-1}\beta^3\vert< 1$, $F$ has exactly two periodic points on each $E_i$, namely two saddle fixed points. \qed 
}}

\bibliographystyle{plain}
\bibliography{biblio-serge}

\end{document}